\documentclass[11pt]{article}
\usepackage[margin=1in]{geometry}
\usepackage[utf8]{inputenc}
\usepackage[T1]{fontenc} %

\usepackage{mathpazo}
\usepackage{xcolor}
\usepackage{color}
\definecolor{niceRed}{RGB}{190,38,38}
\definecolor{niceYellow}{HTML}{f5b400}
\definecolor{blueGrotto}{HTML}{059DC0}
\definecolor{royalBlue}{HTML}{057DCD}
\definecolor{navyBlue}{HTML}{0B579C}
\definecolor{limeGreen}{HTML}{81B622}
\definecolor{nicePurple}{HTML}{9c27b0}
\definecolor{lightRoyalBlue}{HTML}{def2ff}  
\definecolor{ivory}{HTML}{FFFFF0}

\newcommand{\white}[1]{\textcolor{white}{#1}}

\usepackage{color-edits} 
\addauthor[Jane]{jl}{blue} 
\addauthor[Anay]{am}{niceRed}
\addauthor[Manolis]{mz}{teal}

\usepackage{hyperref}
\hypersetup{
  colorlinks = true, 
  urlcolor = {blueGrotto},
  linkcolor = {royalBlue}, 
  citecolor = {navyBlue} 
}
\usepackage[
  natbib=true,
  style=alphabetic,
  sorting=alphabeticlabel,
  sortcites=true,
  minbibnames=3,
  maxbibnames=999,
  mincitenames=1,
  maxcitenames=4,
  minalphanames=1,
  maxalphanames=6,
]{biblatex}

\DeclareSortingTemplate{alphabeticlabel}{
  \sort[final]{%
    \field{labelalpha} %
  }
  \sort{%
    \field{year}   %
  }
  \sort{%
    \field{title}  %
  }
}

\AtBeginRefsection{\GenRefcontextData{sorting=ynt}}
\AtEveryCite{\localrefcontext[sorting=ynt]}
\AtEveryBibitem{%
  \clearname{editor}%
}
\AtEveryBibitem{%
  \clearlist{location}%
} 

\usepackage{booktabs} 
\usepackage{multicol} 
\usepackage{multirow} 
\usepackage{longtable}

\usepackage{subfigure}
\usepackage{graphicx}
\usepackage{float} 
\usepackage{tikz}
\usepackage{pgfplots}
\pgfplotsset{compat=1.17}
\usepgfplotslibrary{fillbetween}

\usepackage{physics} %
\usepackage{amsmath}
\usepackage{amssymb}
\usepackage{amsthm}
\usepackage{bbm}
\usepackage{blkarray}
\usepackage{mathtools}
\usepackage{nicefrac}
\usepackage{dsfont}
\usepackage{upgreek} %
\usepackage{pifont} %
\usepackage{thm-restate} %
\usepackage[capitalise,noabbrev,nameinlink]{cleveref} %
\renewcommand{\gamma}{\upgamma}

\usepackage{enumerate} 
\usepackage[inline]{enumitem} 
\usepackage{tcolorbox}
\usepackage[framemethod=TikZ]{mdframed}
\mdfsetup{%
backgroundcolor=ivory, 
roundcorner=4pt,
linewidth=1pt}
\usepackage{changepage} %

\usepackage{algorithm}
\usepackage{algpseudocode}[1]

\algnewcommand\algorithmicoracle{\textbf{Oracles:}}
\algnewcommand\Oracle{\item[\algorithmicoracle]}

\algdef{SE}[SUBALG]{Indent}{EndIndent}{}{\algorithmicend\ }%
\algtext*{Indent}
\algtext*{EndIndent}

\usepackage{mathrsfs} %
\usepackage{soul} %

\theoremstyle{plain} 
\newtheorem{theorem}{Theorem}[section]

\newtheorem{corollary}[theorem]{Corollary}
\newtheorem{proposition}[theorem]{Proposition}
\newtheorem{lemma}[theorem]{Lemma}
\newtheorem{fact}[theorem]{Fact}

\newtheorem{assumption}{Assumption}
\newtheorem{infassumption}{Informal Assumption}
\newtheorem{inftheorem}{Informal Theorem}

\newtheorem{definition}{Definition}
\newtheorem*{definition*}{Definition}

\theoremstyle{definition} 

\newtheorem{remark}[theorem]{Remark}

\theoremstyle{remark}

\AfterEndEnvironment{definition}{\noindent\ignorespaces}
\AfterEndEnvironment{assumption}{\noindent\ignorespaces}
\AfterEndEnvironment{lemma}{\noindent\ignorespaces}
\AfterEndEnvironment{theorem}{\noindent\ignorespaces}
\AfterEndEnvironment{proposition}{\noindent\ignorespaces}
\AfterEndEnvironment{fact}{\noindent\ignorespaces}
\AfterEndEnvironment{question}{\noindent\ignorespaces}
\AfterEndEnvironment{corollary}{\noindent\ignorespaces}
\AfterEndEnvironment{model}{\noindent\ignorespaces}
\AfterEndEnvironment{remark}{\noindent\ignorespaces}
\AfterEndEnvironment{proof}{\noindent\ignorespaces}
\AfterEndEnvironment{fact}{\noindent\ignorespaces}
\AfterEndEnvironment{minftheorem}{\noindent\ignorespaces}
\AfterEndEnvironment{inftheorem}{\noindent\ignorespaces}
\AfterEndEnvironment{maintheorem}{\noindent\ignorespaces}
\AfterEndEnvironment{restatable}{\noindent\ignorespaces}
\AfterEndEnvironment{infassumption}{\noindent\ignorespaces}

\crefname{section}{Section}{Sections}
\crefname{theorem}{Theorem}{Theorems}
\crefname{lemma}{Lemma}{Lemmas}
\crefname{definition}{Definition}{Definitions}
\crefname{conjecture}{Conjecture}{Conjectures}
\crefname{corollary}{Corollary}{Corollaries}
\crefname{construction}{Construction}{Constructions}
\crefname{conjecture}{Conjecture}{Conjectures}
\crefname{claim}{Claim}{Claims}
\crefname{observation}{Observation}{Observations}
\crefname{proposition}{Proposition}{Propositions}
\crefname{fact}{Fact}{Facts}
\crefname{question}{Question}{Questions}
\crefname{problem}{Problem}{Problems}
\crefname{remark}{Remark}{Remarks}
\crefname{example}{Example}{Examples}
\crefname{equation}{Equation}{Equations}
\crefname{appendix}{Section}{Sections}
\crefname{algorithm}{Algorithm}{Algorithms}
\crefname{model}{Model}{Models}
\crefname{figure}{Figure}{Figures}
\crefname{inftheorem}{Informal Theorem}{Informal Theorems}
\crefname{infassumption}{Informal Assumption}{Informal Assumptions}
\crefname{minftheorem}{Main Informal Theorem}{Main Informal Theorems}
\crefname{maintheorem}{Main Theorem}{Main Theorems}
\crefname{assumption}{Assumption}{Assumptions}

\newlist{asmpenum}{enumerate}{1} %
\setlist[asmpenum]{label={\arabic*.},ref=\theassumption.{\arabic*}}
\crefname{asmpenumi}{Assumption}{Assumptions}
 
\newlist{infasmpenum}{enumerate}{1} %
\setlist[infasmpenum]{label={\arabic*.},ref=\theassumption.{\arabic*},leftmargin=20pt}
\crefname{infasmpenumi}{Informal Assumption}{Informal Assumptions}

\usepackage{etoolbox}
\newenvironment{subenvironment}{%
    \begin{adjustwidth}{2em}{2em}%
}{
    \end{adjustwidth}%
}

\BeforeBeginEnvironment{subenvironment}{\white{.}\\ \vspace{-10mm}}
\AfterEndEnvironment{subenvironment}{\vspace{4mm}}

\newcommand{\yesnum}{\addtocounter{equation}{1}\tag{\theequation}} 
\makeatletter
\newcommand{\tagnum}[2]{%
    \refstepcounter{equation}%
    \tag{#1) \ (\theequation}%
    \protected@write \@auxout {}{%
        \string \newlabel {#2}{{\theequation}{\thepage}{}{equation.\theequation}{}}%
    }%
}
\makeatother

\newcommand{\Stackrel}[2]{\stackrel{\mathmakebox[\widthof{\ensuremath{#2}}]{#1}}{#2}}

\newcommand{\quadtext}[1]{\quad\text{#1}\quad}
\newcommand{\qquadtext}[1]{\qquad\text{#1}\qquad}
 
\newcommand{\quadand}{\quadtext{and}}
\newcommand{\qquadand}{\qquadtext{and}}

\newcommand{\qquadwhere}{\qquadtext{where}}

\def\abs#1{\left| #1 \right|}
\def\sabs#1{| #1 |}

\newcommand{\sinparen}[1]{(#1)}
\newcommand{\sinbrace}[1]{\{#1\}}

\newcommand{\inbrace}[1]{\left\{#1\right\}}
\newcommand{\sinangle}[1]{\langle#1\rangle}
\newcommand{\inparen}[1]{\left(#1\right)}
\newcommand{\insquare}[1]{\left[#1\right]}
\newcommand{\inangle}[1]{\left\langle#1\right\rangle}

\newcommand{\snorm}[1]{\ensuremath{\| #1 \|}}
\let\norm\relax
\newcommand{\norm}[1]{\ensuremath{\left\lVert #1 \right\rVert}}

\newcommand{\midsepremove}{\aboverulesep = 0mm \belowrulesep = 0mm}

\newcommand{\midsepdefault}{\aboverulesep = 0.605mm \belowrulesep = 0.984mm}

\newcommand{\R}{\mathbb{R}}

\newcommand{\evE}{\ensuremath{\mathscr{E}}}

\renewcommand{\d}{{\rm d}}
\newcommand{\E}{\operatornamewithlimits{\mathbb{E}}} 
\newcommand{\Ex}{\E}

\newcommand{\cov}{\ensuremath{\operatornamewithlimits{\rm Cov}}}
\renewcommand{\var}{\ensuremath{\operatornamewithlimits{\rm Var}}}
\newcommand\ind{\mathds{1}}

\newcommand{\argmin}{\operatornamewithlimits{arg\,min}}

\newcommand{\tv}[2]{\operatorname{d}_{\mathsf{TV}}\sinparen{#1,#2}}
\newcommand{\kl}[2]{\operatornamewithlimits{\mathsf{KL}}\sinparen{#1\|#2}} 
\newcommand{\chidiv}[2]{\chi^2\inparen{#1\|#2}}

\newcommand{\renyi}[3]{\cR_{#1}\inparen{#2\|#3}}

\newcommand{\zo}{\ensuremath{\inbrace{0, 1}}}

\newcommand{\sfrac}[2]{{#1/#2}} 
\newcommand{\nfrac}[2]{\nicefrac{#1}{#2}}
\newcommand{\OPT}{\mathrm{OPT}} 
\newcommand{\opt}{\textrm{OPT}}
\newcommand{\Err}{\mathrm{Err}}
\newcommand{\err}{\textrm{\rm Err}}
\newcommand{\poly}{\mathrm{poly}}
\newcommand{\polylog}{\mathrm{polylog}}
\newcommand{\vc}{\textrm{\rm VC}}

\newcommand{\iid}{i.i.d.}

\newcommand{\tick}{\ding{51}}
\let\cross\relax
\newcommand{\cross}{\ding{55}}

\newcommand{\eps}{\varepsilon}
\renewcommand{\epsilon}{\varepsilon}
\makeatletter
\newcommand*{\tran}{{\mathpalette\@tran{}}}
\newcommand*{\@tran}[2]{\raisebox{\depth}{$\m@th#1\intercal$}}
\makeatother

\mathchardef\NABLA"272
\newcommand*{\Nabla}{\boldsymbol\NABLA}
\let\nabla\Nabla

\newcommand{\wh}[1]{\widehat{#1}}

\renewcommand{\bar}{\overline}

\renewcommand{\tilde}{\widetilde}
\newcommand{\wt}[1]{\widetilde{#1}}

\newcommand{\customcal}[1]{\euscr{#1}}

\newcommand{\cD}{\customcal{D}}
\newcommand{\cE}{\customcal{E}}

\newcommand{\cM}{\customcal{M}}
\newcommand{\cN}{\customcal{N}}

\newcommand{\cP}{\customcal{P}} 
\newcommand{\cQ}{\customcal{Q}} 
\newcommand{\cR}{\customcal{R}} 
\newcommand{\cS}{\customcal{S}}
 
\newcommand{\cU}{\customcal{U}}

\newcommand{\cX}{\customcal{X}}
\newcommand{\cY}{\customcal{Y}}

\DeclareMathAlphabet{\mathdutchcal}{U}{dutchcal}{m}{n}
\SetMathAlphabet{\mathdutchcal}{bold}{U}{dutchcal}{b}{n}
\DeclareMathAlphabet{\mathdutchbcal}{U}{dutchcal}{b}{n}
\DeclareMathAlphabet\urwscr{U}{urwchancal}{b}{n}%
\DeclareMathAlphabet\rsfscr{U}{rsfso}{m}{n}
\DeclareMathAlphabet\euscr{U}{eus}{m}{n}
\DeclareFontEncoding{LS2}{}{}
\DeclareFontSubstitution{LS2}{stix}{m}{n}
\DeclareMathAlphabet\stixcal{LS2}{stixcal}{m} {n}

\renewcommand{\paragraph}[1]{\medskip \noindent\textbf{#1}~}

\newcommand\blfootnote[1]{%
  \begingroup
  \renewcommand\thefootnote{}\footnote{#1}%
  \addtocounter{footnote}{-1}%
  \endgroup
}

\newcommand{\eat}[1]{}

\newcommand{\hypo}[1]{\mathdutchcal{#1}}
\newcommand{\hyH}{\hypo{H}}

\newcommand{\hyP}{\hypo{P}}
\newcommand{\hyS}{\hypo{S}}
\renewcommand{\cS}{\hypo{S}}

\newcommand{\pmle}{Perturbed MLE} 
\newcommand{\lreg}{\textsf{L1-Regression}} 
\newcommand{\negLL}{\ensuremath{\mathscr{L}}}

\usepackage{setspace} 

\newcommand{\thetapmle}{\wh{\theta}_{\rm PMLE}}

\newcommand{\sigmaTheta}{\Sigma_{\theta}}
\newcommand{\muTheta}{\mu_{\theta}}
\newcommand{\sigmaStarTheta}{\Sigma^\star} 
\newcommand{\muStarTheta}{\mu^\star} 
\newcommand{\muStar}{\ensuremath{\mu^\star}}
\newcommand{\SigmaStar}{\ensuremath{\Sigma^\star}}
\newcommand{\thetaStar}{\ensuremath{\theta^\star}}
\newcommand{\wStar}{\ensuremath{w^\star}}
\newcommand{\tauStar}{\ensuremath{\tau^\star}}
\newcommand{\Sstar}{\ensuremath{S^\star}}

\newcommand{\sigmaStarThetaExplicit}{\Sigma_{\theta^\star}}
\newcommand{\muStarThetaExplicit}{\mu_{\theta^\star}}

\newcommand{\fourthevmin}[1]{\lambda_{\min}(#1)}
\newcommand{\fourthevmax}[1]{\lambda_{\max}(#1)}
\newcommand{\fourthevlower}[2]{\ell(#1,#2)}
\newcommand{\fourthevupper}[2]{u(#1,#2)}

\usepackage{array}
\newcolumntype{L}[1]{>{\raggedright\let\newline\\\arraybackslash\hspace{0pt}}m{#1}}
\newcolumntype{C}[1]{>{\centering\let\newline\\\arraybackslash\hspace{0pt}}m{#1}}
\newcolumntype{R}[1]{>{\raggedleft\let\newline\\\arraybackslash\hspace{0pt}}m{#1}}

\title{
    Efficient Statistics With Unknown Truncation,\\ 
    Polynomial Time Algorithms, Beyond Gaussians 
}
    
\author{
        {\begin{tabular}{C{4.75cm}C{4.75cm}C{4.75cm}}
        {\bf Jane H. Lee} & {\bf Anay Mehrotra} & {\bf Manolis Zampetakis}\\[2mm]
        {Yale University} & {Yale University} & {Yale University} \\[-4mm]
        {\small\phantom{....................}} \mbox{\small\href{mailto:jane.h.lee@yale.edu}{\url{jane.h.lee@yale.edu}}} & {\small \phantom{............}} \mbox{\small\href{mailto:anaymehrotra1@gmail.com}{\url{anaymehrotra1@gmail.com}}} & \mbox{\small\href{mailto:manolis.zampetakis@yale.edu}{\url{manolis.zampetakis@yale.edu}}}
        \\
        \end{tabular}}
}

\date{}

\begin{document}

\maketitle
\thispagestyle{empty}

\begin{abstract}
    We study the estimation of distributional parameters when samples are shown only if they fall in some \textit{unknown} set $S \subseteq \R^d$.
    Kontonis, Tzamos, and Zampetakis~(FOCS'19) gave a $d^{\poly(1/\eps)}$ time algorithm for finding $\eps$-accurate parameters for the special case of Gaussian distributions with diagonal covariance matrix. 
    Recently, Diakonikolas, Kane, Pittas, and Zarifis~(COLT'24) showed that this exponential dependence on $1/\eps$ is necessary even when $S$ belongs to some well-behaved classes. 
    These works leave the following open problems which we address in this work:
    \begin{center}
      \centering
      \emph{
        Can we estimate the parameters of any Gaussian or even extend the results beyond Gaussians?\\[4pt]
        Can we design $\poly(d/\eps)$ time algorithms when $S$ is a simple set such as a halfspace? %
      }
    \end{center}
    We make progress on both of these questions by providing the following results:
    \begin{enumerate}
      \item Toward the first question, we provide an estimation algorithm with sample and time complexity $d^{\poly(\ell/\eps)}$ for any exponential family that satisfies some structural assumptions and any unknown set $S$ that is $\eps$-approximable by a degree-$\ell$ polynomial. This result has two important applications:
    \begin{itemize}
      \item[(a)] The first algorithm for estimating arbitrary Gaussian distributions (even with non-diagonal covariance matrix) from samples truncated to an unknown set $S$; and  
     \item[(b)] The first algorithm for linear regression with unknown truncation and Gaussian features.
    \end{itemize}     
    \item To address the second question, we provide an algorithm with $\poly(d/\eps)$ sample and time complexity that works for a set of exponential families (that contains all multivariate Gaussians) when $S$ is a halfspace or an axis-aligned rectangle. This is the first fully polynomial time algorithm for estimation with an unknown truncation set.
      \end{enumerate}
        Along the way, we develop new tools that may be of independent interest, including:
        \begin{itemize}
            \item[3.] A reduction from PAC learning with positive and unlabeled samples to PAC learning with positive and negative samples that is robust to certain covariate shifts; and 
            \item[4.] The first polynomial time algorithm for learning halfspaces using \textit{only} positive examples when the samples have an \textit{unknown} Gaussian distribution.
        \end{itemize}

    \blfootnote{{Accepted for presentation at the 65th IEEE Symposium on Foundations of Computer Science (FOCS), 2024}}
\end{abstract}

\clearpage 

\thispagestyle{empty} 

{   
    \setstretch{1.02}
    \tableofcontents
    
}

\thispagestyle{empty}

\newpage

\clearpage
\pagenumbering{arabic}

\section{Introduction}

    Statistical estimation from \textit{truncated samples} has a long history in statistics, going back to at least Daniel Bernoulli who, in 1760, used truncated information about individuals who \textit{survived} smallpox to estimate the expected gain in life expectancy if smallpox was eliminated as a cause of death \cite{bernoulli1760essai}. 
    Truncation occurs when samples falling outside of some subset $S^\star$ of the support of the distribution, called \textit{survival set}, are not observed. 
    Truncation arises in a variety of fields from Econometrics \cite{maddala1983limited}, to Astronomy and other physical sciences \cite{woodroofe1985astronomyTruncation}, to Causal Inference \cite{imbens2015causal,hernan2023causal}.
    In the presence of truncation, standard statistical methods can fail as is readily demonstrated by Berkson's and Simpson's Paradoxes \cite{berkson1946limitations,blyth1972simpson}. 
    The failure of standard methods has inspired a long line of research, by famous statisticians including Galton, Pearson, and Fisher, that develops statistical methods robust to truncation in the data \cite{Galton1897,Pearson1902,PearsonLee1908,Lee1915,fisher31,maddala1983limited,Cohen91,hannon1999estimation,raschke2012inference}.

    \smallskip

   The algorithms in most of these works, however, are computationally inefficient in high dimensions and only work for simple and known survival sets such as intervals.
    The recent work of \citet{daskalakis2018efficient} established the first, provably, sample and computationally efficient algorithm for recovering the parameters of a multivariate Gaussian distribution given samples truncated to an arbitrarily complex but \textit{known} set. 
    Since this work, a line of follow-up work designs efficient algorithms for estimation from truncated data in different settings \cite{Kontonis2019EfficientTS,daskalakis2019computationally, trunc_regression_unknown_var,ons_switch_grad,truncated_sm,lee2023learning}. 
    These follow-up works significantly extend the work of \citet{daskalakis2018efficient}, but still require the knowledge of $S^\star$ (via at least a membership oracle) or {Gaussianity of data if not both.}

    \smallskip

    Currently, the important practical case of \textit{unknown} survival set is still poorly understood.
    In this case, a recent line of work tackles the problem of \textit{testing} whether a given source of data is truncated or not \cite{canonne2020learning, de2023testing, de2024detecting}.
    Regarding statistical estimation with unknown survival sets, the only existing method is by \citet{Kontonis2019EfficientTS}.
    This, however, has two important limitations. 
    First, it heavily relies on the assumption that the data is Gaussian and, in fact, is not capable of even learning Gaussian distributions with arbitrary covariance matrices. 
    This is an important limitation since extending the known results to general Gaussians will enable the first algorithm for truncated linear regression with unknown survival sets, where the arising Gaussian distributions do not satisfy the structural assumptions of \citet{Kontonis2019EfficientTS}. 
    This leads to the first open question in existing literature. 
    \vspace{-1mm}

    \begin{mdframed}
      \textbf{Question 1:}
        Is the estimation of general Gaussians possible from samples truncated to unknown survival set $S^\star$? Is this possible for distributions beyond Gaussians?
    \end{mdframed}
    \vspace{-1mm}
    \noindent The second important limitation of \citet{Kontonis2019EfficientTS} is that to estimate the parameters up to error $\eps$, they require time exponential in $\sfrac{1}{\eps}$ even for simple families of survival sets such as halfspaces. %
    This inspired the recent work of \citet{diakonikolas2024statistical} that shows that this exponential dependence on $\sfrac{1}{\eps}$ is necessary even when we know that the survival set $S^\star$ belongs to a class that simultaneously has constant VC dimension and constant Gaussian surface area. This brings us to the following open question.
    \vspace{-1mm}
 
    \begin{mdframed}
        {\textbf{Question 2:} Is $\eps$-accurate estimation in $\poly(\sfrac{d}{\eps})$ time possible from samples truncated to an unknown survival set $S^\star$ when $S^\star$ is known to belong to a restricted family such as halfspaces?}        
    \end{mdframed}
    
    \paragraph{Summary of Our Results.} 
        We revisit these open questions by studying the problem of estimating the parameters of an \textit{exponential family} under truncation with an unknown survival set. %
        Exponential families are a versatile class that includes many fundamental distribution families such as Gaussian, Exponential, and Weibull distributions. %
        Our results show that recovery of parameters to arbitrary accuracy is possible for general Gaussians (\cref{infthm:gaussians}) as well as exponential families satisfying certain requirements (\cref{infthm:exponential}), resolving Question 1. 
        The former result, in particular, implies the first algorithm for linear regression with unknown truncation (\cref{infthm:linearRegression}). 
        Furthermore, when the survival set is known to be a halfspace or axis-aligned rectangle, we give a $\poly(d/\eps)$ time algorithm (\cref{infthm:polyTime}), answering Question 2.

    \subsection{Our Contributions} \label{sec:contributions}

    In this section, we provide high-level statements of our results. 
    First, we present our results in the context of Gaussian distributions. %
    Then, we explain how our results generalize to certain exponential families. 
    Finally, we present sample complexity results that improve the results of \citet{Kontonis2019EfficientTS} in the sense that they are based on a single optimization oracle call, as opposed to, their inherently inefficient algorithms based on $\eps$-covers.

    \subsubsection*{Revisiting Estimation of Gaussian Distributions with Unknown Truncation} \label{sec:contributions:Gaussians}
    
    We first present our results in the context of multivariate Gaussian distributions which already answer several open questions in \citet{Kontonis2019EfficientTS}.
    Let $\cN(\mu^{\star}, \Sigma^{\star})$ be the multivariate normal distribution with mean $\mu^{\star}$ and covariance matrix $\Sigma^{\star}$ that we are trying to estimate. 
    Throughout the paper, we assume that the mass of $S^\star$ under $\cN(\mu^\star, \Sigma^\star)$ is at least a constant, e.g., $1\%$, which is known to be necessary even in the case \mbox{where the $\Sstar$ is known \cite{daskalakis2018efficient}.}
    
    \begin{inftheorem}[Efficient Estimation of Gaussians; see \cref{thm:main}]\label{infthm:gaussians} 
      Let $\hyS$ be a class of sets with Gaussian surface area at most $\Gamma(\hyS)$, suppose that the unknown survival set $S^{\star}\in \hyS$ has mass $\Omega(1)$, and let $\ell = \poly(1/\eps) \cdot \Gamma(\hyS)^2$. Given $N = d^\ell$ samples from a $d$-dimensional Gaussian $\cN(\mu^\star, \Sigma^\star)$, truncated on an unknown set $S^\star \in \hyS$, we can find parameters $\wh{\mu}$ and $\wh{\Sigma}$ such that with probability $2/3$
      \[ 
            \tv{
                \cN\sinparen{\wh{\mu},\wh{\Sigma}}
              }{
                ~\cN\sinparen{\mu^\star, \Sigma^\star}
            } 
            \leq \eps\,. 
      \]
      The time required to compute $\wh{\mu}$ and $\wh{\Sigma}$ is $\poly(N)$.
    \end{inftheorem}
    This theorem generalizes the efficient algorithm of  \citet{Kontonis2019EfficientTS} to work for non-diagonal Gaussians. This result is a corollary of the more general \cref{infthm:exponential} which applies to certain exponential families. The running time of \cref{infthm:gaussians} is ``tight'' for any statistical-query-based algorithm, such us our algorithm, due to a recent lower bound \cite{diakonikolas2024statistical}.
    To get a sense of the running times in \cref{infthm:gaussians}, see the known bounds on $\Gamma(\hyS)$ in
    \cref{tab:gsa:GaussianSurfaceArea}. %
    \medskip

    \noindent \textbf{Applications to Linear Regression.} 
    Our result for general Gaussian allows us to solve linear regression with unknown truncation jointly in covariates and dependent variables, under the common assumption that the covariates are Gaussian. 
    In contrast, prior work is limited to truncations only on the dependent variable or only on dependent variables and, in addition, also requires the survival set to be known \cite{daskalakis2019computationally, daskalakis2020truncated,ilyas2020theoretical,trunc_regression_unknown_var,ons_switch_grad}.
    
    \begin{inftheorem}[Linear Regression with Unknown Truncation; see \cref{thm:trunc_linear_reg}]
        \label{infthm:linearRegression}
      Let $x$ be drawn from an unknown $d$-dimensional Gaussian $\cN(\mu^{\star}, \Sigma^{\star})$ and $y = x^\top w^{\star} + b^{\star} + \zeta$ where $\zeta \sim \cN(0, 1)$. 
      Let $\hyS$ be a class of sets with Gaussian surface area at most $\Gamma(\hyS)$, suppose that the unknown survival set $S^{\star}\in \hyS$ has mass $\Omega(1)$, and let $\ell = \poly(1/\eps)\cdot \Gamma(\hyS)^2$. Then, there is an algorithm that, given $N = d^{\ell}$ samples $(x,y)$ truncated to $S^{\star} \in \cS$, outputs estimates $\wh{w}$ and $\wh{b}$ such that with probability $2/3$ 
      \[ \norm{\wh{w} - w^{\star}}_2 \leq \eps \quadand \sabs{\wh{b} - b^\star} \leq \eps\,. \]
      The time required to compute $\wh{w}$ and $\wh{b}$ is $\poly(N)$.
    \end{inftheorem}
    This is the first result for linear regression with unknown truncation and it crucially relies on the fact that \cref{infthm:gaussians} can handle general multivariate Gaussian distributions.
    \medskip

    \noindent \textbf{Polynomial Time Algorithm for Simple Sets.} Finally, we provide the first $\poly(d/\eps)$-time algorithm for estimating the parameters of a Gaussian truncated to an unknown survival set.

    \vspace{-1.5mm}

    \begin{inftheorem}[Polynomial-Time Algorithms for ``Simple'' Survival Sets; see \cref{thm:mainPolynomial}]\label{infthm:polyTime}
            Let $\hyS$ be the class of axis-aligned rectangles or halfspaces, and suppose that the survival set $S^{\star}\in \hyS$ has mass $\Omega(1)$ with respect to $\cN(\muStar,\SigmaStar)$.
            Given $\poly(d/\eps)$ samples from a $d$-dimensional Gaussian $\cN(\mu^\star, \Sigma^\star)$, truncated on an unknown set $S^\star \in \hyS$, we can find parameters $\wh{\mu}$ and $\wh{\Sigma}$ such that with high probability
                \[
                    \tv{
                        \cN(\wh{\mu},\wh{\Sigma})
                    }{
                        ~\cN(\mu^\star, \Sigma^\star)
                    } 
                        \leq \eps\,. 
                \]
            The time required to compute $\wh{\mu}$ and $\wh{\Sigma}$ is $\poly(d/\eps)$.
        \end{inftheorem}
   The result above complements the recent lower bound of \citet{diakonikolas2024statistical} where they show that a slightly more complicated survival set, i.e., the complement of a union of axis-aligned rectangles, needs exponential in $1/\eps$ runtime.

   \vspace{-1.5mm}

    \subsubsection*{Generalizing to Exponential Families}
    \label{sec:contributions:exponential}
        In this section, we present our results in the context of certain exponential families. 
        A probability distribution $\cE(\theta)$ belongs to an exponential family if its density has the form $\cE(x; \theta) \propto \exp\sinparen{\theta^{\top} t(x)}$, where $\theta$ is the vector of parameters of $\cE$ and $t(\cdot)$ is its sufficient statistics. 
        Exponential family distributions arise as solutions to certain natural optimization problems over the space of distributions (namely, constrained maximum entropy problems).
        This is one reason for their presence in various real-world domains \cite{jaynes1982maximum}.
        Moreover, they are also a pragmatic statistical model as they are analytically tractable, e.g., they are amenable to learning and optimization (because, e.g., of the convexity of their log-likelihood) and 
        {tractable conjugate priors exist (i.e., the posterior distribution is in the same exponential family as the prior, and the posterior often has a closed-form expression)} \cite{dasgupta2008asymptotic}.
        
        Estimating the parameters of an exponential family is a classical problem in statistics with many classical and recent works \cite{fisher1934two,learning_exp_fam_high_dim,shah2021exponential,Efron_2022,pabbaraju2023provable}. 
        Without any additional assumptions, however, learning the parameters of exponential families even with \textit{no} truncation can be computationally hard \cite{shah2021exponential,pabbaraju2023provable}. 
        Below, we give a high-level description of the assumptions we need the exponential family to satisfy for applying our results. 
        The formal statement of this assumption appears as \cref{asmp:1:sufficientMass,asmp:1:polynomialStatistics,asmp:2}.

    \vspace{-1.5mm}

    \begin{infassumption}[Conditions on the Exponential Family; see \cref{asmp:1:sufficientMass,asmp:1:polynomialStatistics,asmp:2}] \label{infasmp:1}
      We assume that the non-truncated density with true parameters $\cE(\theta^{\star})$ satisfies the following: (1) The mass of $S^{\star}$ with respect to $\cE(\theta^{\star})$ is at least a constant, e.g., 1\%, and (2) $\cE(\thetaStar)$ is log-concave and its sufficient statistic $t(\cdot)$ is a multivariate polynomial. 
      In addition, we make some standard assumptions which we state here informally but discuss in detail in \cref{sec:overview:assumptions,sec:assumptions}: (3) we assume that the non-truncated log-likelihood of the exponential family is smooth and strongly concave over a known set of parameters $\Theta$, (4) that {there is a parameter $\theta$ whose expected sufficient statistics match those of the truncated data}, (5) that we can project to $\Theta$, 
      (6) for any two distributions $\cE_1,\cE_2$ in the family there is a third \emph{bridge} distribution that is close to $\cE_1,\cE_2$ in the $\chi^2$-divergence, and 
      (7) one can sample from $\cE(\theta)$ efficiently given $\theta$.
    \end{infassumption}
    In \cref{sec:preprocess}, we show that, after some pre-processing, Gaussians and product Exponential distributions satisfy all of these assumptions.
    Under this assumption, we prove the following.

    \begin{inftheorem}[Efficient Estimation of Truncated Exponential Families; see \cref{thm:main,thm:mainPolynomial}]\label{infthm:exponential}
            Fix any $d$-dimensional exponential family distribution $\cE(\theta^\star)$ with unknown $m$-dimensional parameter $\theta^\star$ that satisfies \cref{infasmp:1}. Then the generalization of \cref{infthm:gaussians} holds for the family $\cE$ instead of the normal distribution, where the sample complexity and running time have the same dependence on the dimension $d,\eps$ and $\ell\geq 1$ is the minimum degree of polynomials that $\poly(\eps)$-approximate $\Sstar$ in $L_2$-norm with respect to the distribution $\cE(\thetaStar)$. %
        \end{inftheorem}

    \noindent This result illustrates that our method does not depend on the Gaussianity of data. 
    In contrast, the results of \citet{Kontonis2019EfficientTS} cannot be applied to other distributions because they operate with the properties of Hermite polynomials that form an orthonormal basis only with respect to the Gaussian measure. 
    Further, if $\Sstar$ is an axis-aligned rectangle, we also get a polynomial time algorithm for the exponential families satisfying \cref{infasmp:1}.
    \begin{inftheorem}[Polynomial-Time Algorithm of Truncated Exponential Families; see \cref{thm:polyTime:exponential}]\label{infthm:polyTime:exponential}
        Consider the setting in \cref{infthm:exponential}.
        Suppose $\Sstar$ is an axis-aligned rectangle.
        Given $\poly(dm/\eps)$ samples from $\cE(\thetaStar)$ truncated on the (unknown) set $S^\star$, we can find a parameter $\wh{\theta}$ such that with high probability $\tv{\cE(\wh{\theta})}{\cN(\thetaStar)} \leq \eps.$
        The time required to compute $\wh{\theta}$ is $\poly(d/\eps)$.
    \end{inftheorem}
    \cref{tab:contributions} summarizes our results and compares them to prior works.

    {
        \begin{table}[ht!]
            \centering
            \vspace{-2mm}
            \midsepremove{}
                \caption{
                    Our results in the context of prior works. 
                    The survival set $S^\star\in\hyS$ is assumed to have mass $\Omega(1)$ under the underlying distribution $\cE(\theta^\star)$.
                    For a set family $\hyS$, $\Gamma(\hyS)$ denotes its Gaussian Surface Area.
                    $\hyP(\ell)$ is the family of polynomials of degree at most $\ell$.
                }
                \vspace{2mm} 
                \small 
            \begin{tabular}{c | p{4.05cm} p{5.45cm} c c }
            & \shortstack[c]{Assumptions on the\\ Exponential Family} & \shortstack[c]{Set family $\hyS$} & \shortstack[c]{Works with\\unknown $S^\star$} & Run time\\ 
            \midrule
            \cite{daskalakis2018efficient} & Gaussian & Any & \cross{} & $\poly(d/\eps)$ \\
            \cite{lee2023learning} & \cref{asmp:1:sufficientMass,asmp:1:polynomialStatistics,asmp:2}~\footnotemark{} 
            & Any & \cross{}  & $\poly(d/\eps)$ \\
            \cite{Kontonis2019EfficientTS} & Near-Diagonal Gaussian~\footnotemark{} & Finite $\Gamma(\hyS)$ & \tick{} & $d^{\poly(\Gamma(\hyS)/\eps)}$ \\
            \midrule 
            \multicolumn{1}{p{1.6cm}|}{
                \multirow{4}{*}{
                    {This work}
                }  
            }
            &  
                    \multirow{2}{*}{
                        \shortstack[c]{
                            \cref{asmp:1:sufficientMass,asmp:1:polynomialStatistics,asmp:2}
                        }
                    }
                    & \mbox{Axis-aligned box} & 
                        \tick{} & 
                            $\poly(d/\eps)$ \\ 
            \multicolumn{1}{l|}{}  & & $\eps$-approximable by $\hyP(\ell)$ in $L_2$ norm& \tick{} & $\poly(d^\ell/\eps)$\\
            \cmidrule(r){2-5} 
            \multicolumn{1}{l|}{}  & \multirow{2}{*}{
                    \shortstack[c]{
                        Gaussian
                    }
                } & \mbox{Halfspaces/axis-aligned box} & \tick{} & $\poly(d/\eps)$\\
            \multicolumn{1}{l|}{}      & & Finite $\Gamma(\hyS)$ & \tick{} & $d^{\poly{(\Gamma(\hyS)/\eps)}}$\\
            \end{tabular}
            \label{tab:contributions}
            \midsepdefault{}
            \vspace{-6mm}
        \end{table}
    }

    \subsubsection*{Sample Complexity via Single ERM Query} 
      
    Our proof techniques have interesting consequences for getting sample complexity results that only make a single call to an optimization oracle. 
    When this optimization oracle can be heuristically implemented in practice, this result can enable new practical algorithms.

    \begin{inftheorem}[Single ERM Sample Complexity; see \cref{thm:sampleComplexity}]\label{infthm:sampleComplexity}
            Fix any $d$-dimensional exponential family distribution $\cE(\theta^\star)$ with unknown $m$-dimensional parameter $\theta^\star$ that satisfies \cref{infasmp:1} and any class $\hyS$ with VC-dimension $\vc(\hyS)$.
            Given access to an ERM oracle for $\hyS$ and $\wt{\Omega}\sinparen{\sinparen{{{\vc(\hyS) + m}}{}}/\eps^{O(1)}}$ samples from $\cE(\theta^\star)$ truncated to a set $S^\star\in \hyS$ with mass $\Omega(1)$, 
            there is an algorithm that makes one call to the ERM oracle and takes $\poly(dm/\eps)$ additional time
            and outputs an estimate $\wh{\theta}$ such that $\tv{\cE(\wh{\theta})}{~\cE(\theta^\star)} \leq \eps.$ 
        \end{inftheorem}

    \addtocounter{footnote}{-1}
    \footnotetext{
        \citet{lee2023learning} explicitly state \cref{asmp:1:sufficientMass,asmp:1:polynomialStatistics,asmp:cov} and implicitly assume \cref{asmp:int,asmp:moment,asmp:proj}.
        They do not provide pre-processing routines to satisfy these assumptions for any distribution family.
        Hence, in particular, their algorithm is not applicable to arbitrary Gaussians.
        In \cref{sec:preprocess}, we give pre-processing routines that satisfy these assumptions given truncated samples from an arbitrary Gaussian distribution or a product Exponential distribution.
        This also shows how to use \citet{lee2023learning}'s algorithm for these families unconditionally.
    }  
    \addtocounter{footnote}{1}
    \footnotetext{
        They require the covariance matrix $\SigmaStar$ to satisfy 
        $I/16\preceq \SigmaStar \preceq (16/15)\cdot I$ or be diagonal.
    }

    \section{Preliminaries}\label{sec:preliminaries}
        In this section, we give preliminaries that are used in what follows.

        \paragraph{Notation.}
            Given a distribution $\cD$ over $\R^d$ and a set $S\subseteq\R^d$, $\cD(S)$ denotes the mass of $S$ under $\cD$, i.e., $\Pr_{x\sim \cD}[x\in S]$.
            Extending this notation: given a distribution family $\cD(\cdot)$, a parameter $\theta$, and set $S$, $\cD(S;\theta)$ is the mass of $S$ under the distribution $\cD(\theta)$, i.e., $\Pr_{x\sim \cD(\theta)}[x\in S]$.
            For distributions $\cD$ over $\cX \times \cY$, $\cD_X$ denotes the marginal distribution of $\cD$ over $\cX$.
            Given a point $x\in \R^d$ and a set $S\subseteq\R^d$, we use $\mathds{1}_S(x)$ and $\ind\{x \in S\}$ to denote the indicator that $x\in S$.
            We use standard notation for vector and matrix norms:
            For a vector $z\in \R^d$ and $p\geq 1$, the $L_p$-norm of $z$ is $\norm{z}_p\coloneqq \inparen{\sum_i \abs{z_i}^p}^{\sfrac{1}{p}}$ and, taking limit $p\to\infty$ implies, $\norm{z}_\infty=\max_i \abs{z_i}$.
            Fix a $d\times d$ real matrix $A\in \R^{d\times d}$.
            $A$'s Frobenius norm is $\norm{A}_F\coloneqq \snorm{A^\flat}_2$.
            Further, suppose that $A$ is symmetric.
            Its spectral norm is defined as $\norm{A}_2\coloneqq \max_{x\neq 0}\sfrac{\norm{Ax}_2}{\norm{x}_2}$.
            It is well known that $\norm{A}_2=\max_{1\leq i\leq d} \abs{\lambda_i}$ where $\lambda_1,\lambda_2,\dots,\lambda_d$ are eigenvalues of $A$.
            Moreover, it is also well known that $\norm{A}_2\leq \norm{A}_F\leq \sqrt{d}\cdot \norm{A}_2$.
            Using the $L_2$-norm, we define the relative interior of sets: 
                for a set $S$ lying in an affine space $L$ and $\eps>0$, the $\eps$-relative interior of $S$ is the set of all points $x\in S$ such that  $B_2(x,\eps)\cap L\subseteq S$ where $B_2(x,\eps)$ is the $L_2$-ball of radius $\eps$ centered at $x$.
            Finally, we also use standard definitions of distances and divergences between distributions:
            Consider two distributions $\cP$ and $\cQ$ over $\R^d$.
            The total variation distance between $\cP$ and $\cQ$ is $\tv{\cP}{\cQ}\coloneqq  (1/2) \int_{x}\abs{\d \cP(x)-\d \cQ(x)}\d x$.
            Given $q>1$, the Rényi divergence of the $q$-th order from $\cQ$ to $\cP$ is defined as $\renyi{q}{\cP}{\cQ}\coloneqq \inparen{\sfrac{1}{(q-1)}}\cdot \ln\inparen{\Ex_{x\sim \cQ}\insquare{\inparen{\sfrac{\d \cP(x)}{\d \cQ(x)}}^q}}$ when $\cP$ is absolutely continuous with respect to $\cQ$ and otherwise is $\infty$.
            The $\chi^2$-divergence from $\cQ$ to $\cP$ is defined as $\chidiv{\cP}{\cQ}=\exp\inparen{\renyi{2}{\cP}{\cQ}}-1$.

        \paragraph{Exponential Family.}
        We study exponential families in their canonical form.
        \begin{definition}[{Canonical Exponential Family}]
            Given $m\geq 1$, a function $h\colon \R^d\to \R_{\geq 0}$ called the \emph{carrier measure}, a function $t\colon \R^d\to \R^m$ called the \emph{sufficient statistic}, the exponential family $\cE_{m,h,t}$ is a set of distributions parameterized by a vector $\theta\in \R^m$ where, for each $\theta$, $\cE_{m,h,t}(\theta)\propto h(x) \cdot \exp\inparen{\theta^\top t(x)}$. 
        \end{definition}
        Henceforth, $m$, $h$, and $t(\cdot)$ will be clear from convex and, hence, we drop them from the subscript. %
        A number of common distributions (such as Gaussian, Exponential, and Weibull distributions) can be parameterized as exponential families.
        Of specific interest will be Gaussian distributions, which are an exponential family with a constant carrier measure $h(x)=(2\pi)^{-d/2}$,
        $m=d+d^2$, and sufficient statistics $t(x)^\top=[
            x^\top, -\inparen{xx^\top}^{\flat}
        ]$. 
        Where for a $d\times d$ real matrix $A\in \R^{d\times d}$, $A^\flat\in \R^{d^2}$ denotes its canonical flattening.
        Concretely, with this $h(\cdot)$ and $t(\cdot)$, a $d$-dimensional Gaussian distribution $\cN(\mu,\Sigma)$ with mean $\mu$ and covariance matrix $\Sigma$ is the same as $\cE(\theta)$ for $\theta=(\Sigma^{-1}\mu, (\sfrac{1}{2})\Sigma^{-1})$.
        We note that there are multiple equivalent characterizations. 
        For instance, the parameterization is invariant to scaling $t_i(x)$ and $1/\theta_i$ by the same non-zero constant for any $1\leq i\leq m.$
        Throughout the paper, we assume that the distributions $\cE(\theta)$ are full dimensional.
        If $\cE(\theta)$ is not full dimensional, this can be detected by drawing $d$ samples, and subsequently, we can solve estimation problems in the subspace spanned by these samples.
        
        Next, we give preliminaries on exponential families.
        The set of parameters $\theta$ for which the density of $\cE(\theta)$ is well defined is called the \textit{natural parameter space} $\overline{\Theta}$.
        For each $\theta\in \overline{\Theta}$, the \textit{log-partition function} is $A(\theta)\coloneqq \log~{\int h(x) \cdot \exp\inparen{\theta^\top t(x)}\d x}$.
        The log-partition function is important as it is the only (potential) non-linear component of the negative log-likelihood function--a powerful concept in statistical inference--that is at the core of our algorithms.
        Hence, the convexity of the negative log-likelihood is determined by $A(\cdot)$.
        Thankfully, for any canonical exponential family, $A(\cdot)$ is convex and, hence, the negative log-likelihood is also convex \cite{dasgupta2008asymptotic}. %
        The following expressions for the derivatives of $A(\cdot)$ (proved in, e.g., Ch.~3 of \mbox{\citet{wainwright2008graphical}) are useful for later}
        \[
            \grad A(\theta) = \Ex_{\cE(\theta)}[t(x)]
            \qquadand
            \grad^2 A(\theta) = \cov_{\cE(\theta)}[t(x)]
            \,.
            \yesnum\label{eq:derivativesOfA}
        \]
 
    \noindent\textbf{Truncated Distributions.}\indent
        Given a set $S\subseteq \R^d$, the truncation of $\cE(\theta)$ to $S$ is defined as the distribution of $x\sim \cE(\theta)$ conditioned on $x\in S$ and is denoted by $\cE(\theta, S)$.
        The density of the truncated distribution $\cE(\theta, S)$ is 
        \[
            \cE(x; \theta, S) = \cE(x; \theta)\cdot \frac{
                \mathds{1}_S(x) 
            }{
                \cE(S; \theta)
            }\,,
            \tagnum{Truncated Density}{eq:expressionTruncatedDensity}
        \]
        where $\cE(x; \theta)$ is the density of $\cE(\theta)$ at $x$ and $\mathds{1}_S(x)$ denotes its $0-1$ indicator.

    \paragraph{Polynomial Approximability of Sets.} %
        Polynomial approximability of sets is an important notion of complexity and most efficient agnostic learning algorithms only work for sets approximable by polynomials in the following sense.
        \begin{definition}[Polynomial Approximability]
            Given a set $S\subseteq\R^d$, distribution $\cD$ over $\R^d$, numbers $k,p\geq 1$, and constant $\eps>0$, $S$ is said to be $\eps$-approximable by degree-$k$ polynomials in $L_p$-norm with respect to $\cD$ if there exists a degree-$k$ polynomial $f\colon \R^d\to\R$ such that 
            $\Ex_{x\sim \cD}\insquare{\abs{f(x)-\mathds{1}_{S}(x)}^p}^{\sfrac{1}{p}} < \eps$.

            Moreover, a hypothesis class $\hyS$ is said to be $\eps$-approximable by degree-$k$ polynomials in $L_p$-norm with respect to a family of distributions $\mathfrak{D}$ if, for each $S\in \hyS$ and $\cD\in \mathfrak{D}$, $S$ is  $\eps$-approximable by degree-$k$ polynomials in $L_p$-norm with respect to $\cD$.
        \end{definition}
        Our algorithms also require bounds on the polynomial approximability of the survival set.
        In particular, we are interested in the cases where $p\in \inbrace{1,2}$ and $\mathfrak{D}$ is an exponential family.

\section{Technical Overview}
    In this section, we discuss the key ideas in the proofs of our results, leaving the formal algorithms, assumptions, and proofs to \cref{sec:assumptions,sec:computationalEfficiency:exponential,sec:computationalEfficiency:polyTime}.
    
    \subsection{Assumptions on the Survival Set and the Exponential Family}\label{sec:overview:assumptions} %
        We make the following two assumptions on the survival set and the exponential family. 
        \begin{restatable}[Sufficient Mass]{assumption}{assumptionSufficientMass}\label{asmp:1:sufficientMass}
            There exists a known constant $\alpha>0$ such that $\cE(S^\star; \theta^\star)\geq \alpha$.
        \end{restatable}
        \vspace{-9mm}
        \begin{restatable}[Polynomial $t(\cdot)$, Log-Concavity, and Sampling Oracle]{assumption}{assumptionPolynomialStatistics}\label{asmp:1:polynomialStatistics}
            Each component of $t(\cdot)$ is a degree-$k$ polynomial for some $k\geq 1$, 
            $\cE(\cdot)$ is a family of log-concave distributions, and 
            {there is an oracle that given $\theta$ outputs a sample from a distribution $\zeta$-close to $\cE(\theta)$ in TV-distance in {$\poly\inparen{md/\zeta}$} time.}
        \end{restatable}
        As noted before, the lower bound on the mass of $S^\star$ is required even when $S^\star$ is known \cite{daskalakis2018efficient}.
        We need $t(\cdot)$ to be a polynomial and $\cE(\theta)$ to be log-concave in $x$ to use anti-concentration inequalities due to \citet{carbery2001distributional}, which are important tools in truncated statistics and are even required with Gaussian data and known survival sets.
        One should think of both $\alpha$ and $k$ as constants, and, for this overview, we make the simplification that $k=O(1)$.
        Note that this simplification holds for Gaussian distributions for which $k=2$.
        Regarding the sampling oracle, there is a vast body of existing work on sampling from log-concave distributions (see \citet{chewi2022log} for details).
        For many specific distribution families of interest (such as Gaussians or exponential distributions), efficient and high-accuracy samplers are available (\cref{rem:sampling}).
        Many of the known algorithms can be applied for a general log-concave exponential family, but there are some missing parts in the existing analyses of algorithms to get a fully provable guarantee. It is believed, though, that the analysis can be extended to our setting as well (\cref{rem:new sampling algorithms}). 
        \begin{remark}\label{rem:new sampling algorithms}
            For a general log-concave exponential family, existing sampling algorithms require the log-density to be smooth (or semi-smooth). 
            While log-densities are often non-smooth over the entire domain, smoothness can be ensured over a ball centered at the origin (which is sufficient to execute our algorithms).
            We suspect that existing algorithms \cite{liang2023proximal, fan2023improved, altschuler2024faster_logconcave_warmstarts} can be modified to sample from densities truncated to balls; however, to the best of our knowledge, this has not been done in existing work and is an interesting and promising open problem.
        \end{remark}
        \vspace{-9mm}
        \begin{remark}[Boosting Success Probability]
            Directly using the sampler in \cref{asmp:1:polynomialStatistics} in our proofs implies that the running time and sample complexity scale with $\poly(1/\delta)$ where $\delta>0$ is the failure probability.
            To improve the dependence to $\polylog{(1/\delta)}$, we first solve our problem with $\delta_0=\Omega(1)$, and subsequently, boost the success probability to $1-\delta$ with $O(\log(1/\delta))$ repetitions (see, e.g., the remark on Page 13 of \citet{cherapanamjeri2023selfselection}).
        \end{remark}
            Apart from \cref{asmp:1:sufficientMass,asmp:1:polynomialStatistics}, we require ``access'' to a subset $\Theta$ of the natural parameter space of $\cE$ that contains $\theta^\star$ and has certain useful properties.
            \begin{infassumption}[Parameter Space; see \cref{asmp:2}]\label{infasmp:2}
                    There is a {convex} subset $\Theta\subseteq\R^m$ of the natural parameter space and constants $\Lambda\geq \lambda > 0$ and $\eta>0$ such that the following hold.
                    \begin{infasmpenum} 
                        \item\label{infasmp:cov} \textbf{(Boundedness of Covariance)} For each $\theta\in \Theta$,
                        $\lambda I \preceq \cov_{\cE(\theta)}[t(x)]\preceq \Lambda I$. 
                        \item\label{infasmp:int} \textbf{(Interiority)} 
                            $\theta^\star$ is in the $\eta$-relative-interior $\Theta(\eta)$ of $\Theta$. 
                        \item \label{infasmp:start}\textbf{(Starting point)} 
                            \mbox{There is algorithm to find $\theta_0\in \Theta(\eta)$, s.t., $\norm{\thetaStar-\theta_0}\leq \poly(1/\alpha)$ in $\poly(m)$ time.}
                        \item\label{infasmp:proj} \textbf{(Projection Oracles)} 
                        There is a polynomial time projection oracle to $\Theta$. %
                \end{infasmpenum}
            \end{infassumption}
            The covariance matrix in the first property ($\cov_{\cE(\theta)}[t(x)]$) is known as the Fisher Information matrix, and it is the Hessian of the negative log-likelihood function in the \textit{absence} of any truncation.
            Hence, the first property ensures that the negative log-likelihood is $\lambda$-Strongly-Convex and $\Lambda$-Smooth.
            Note, however, that we observe samples \textit{truncated} to $S^\star$ and, hence, do not have access to the untruncated log-likelihood.
            In the second property, we require $\theta^\star$ to be in the strict (relative) interior of $\Theta$.
            This enables us to get a handle on the mass $\cE(\theta^\star)$ assigns to different sets.
            This is a mild assumption as one can often ``blow up'' $\Theta$ to satisfy it while incurring only a slight increase in $\Lambda/\lambda.$
            {The starting point in the third property can be found under mild assumptions and we leave the details to \cref{sec:assumptions,sec:computationalEfficiency:exponential}.}

            \begin{remark}[Satisfying \cref{infasmp:2}]\label{rem:preprocessing}
                In general, one cannot select $\Theta$ to be the entire natural parameter space.
                In \cref{sec:preprocess}, we construct suitable $\Theta$ for Gaussian and product Exponential distributions.
                For both families, the construction guarantees that \cref{infasmp:2} is satisfied with $\lambda, 1/\Lambda,\eta=\poly(\alpha)$.
            \end{remark}
            Throughout this work, we think of $\lambda$, $1/\Lambda$, and $\eta$ as constants bounded away from 0. %
            Further, in this overview, we assume that $\lambda, 1/\Lambda,\eta=\poly(\alpha)$ to simplify the exposition.
            Due to \cref{rem:preprocessing}, this simplification is without loss of generality for Gaussian distributions.

    \subsection{Outline of Proofs} %
    \label{sec:overview}
        In this section, we present the structure of the proofs of our main results.
        First, we present the structure of proof of \cref{infthm:exponential} (which implies \cref{infthm:gaussians,infthm:linearRegression} as corollaries) and then present the key ideas in the proofs of \cref{infthm:polyTime,infthm:polyTime:exponential}.

        \paragraph{Outline of the proof of \cref{infthm:exponential}.}
            Given membership access to $\Sstar$, one can use existing likelihood-maximization algorithms to find a point $\theta$ that is $\eps$-close to $\thetaStar$ in $\poly(d/\eps)$ time and membership queries \cite{daskalakis2018efficient,lee2023learning}.
            The first difficulty in learning $\thetaStar$ is that we do not have membership access to $\Sstar$.
            \begin{description}
                \item[\cref{sec:overview:pmle}:]  To overcome this, we develop a new algorithm that works with an ``approximate'' membership oracle to $\Sstar$ (\cref{alg:psgd}).
                Concretely, given membership access to a set $S\approx \Sstar$, the algorithm finds a point $\eps$-close to $\thetaStar$ in $\poly(d/\eps)$ time and membership queries (\cref{thm:module:reduction}).
                In contrast to existing algorithms \cite{daskalakis2018efficient,lee2023learning}, we need to introduce a novel optimization problem to solve this (\cref{sec:overview:pmle}), since likelihood maximization is not sufficient when we only know $S^{\star}$ approximately.
            \end{description}
            It remains to find a set $S\approx \Sstar$.
            This is a learning problem where we only see positive samples (i.e., samples that fall inside $\Sstar$).
            This is a challenge as only very structured hypotheses can be learned from just positive samples -- for example, even halfspaces or boxes in two dimensions cannot be learned from just positive samples \cite{natarajan1987learning}.
            We overcome this in two steps.
            \begin{description}
                \item[\cref{sec:intro:denis}:] %
                First, we give an algorithm to find a distribution $\cE(\theta_0)$ that is ``not too far'' from $\cE(\thetaStar)$ in $\poly(d)$ time.
                We then use samples from $\cE(\theta_0)$ as approximate unlabeled samples from $\cE(\thetaStar)$.
                Using these positive samples and the approximate unlabeled samples, we can reduce the of finding $S \approx \Sstar$ to an agnostic learning problem with a lot of structure (\cref{thm:efficientLearning:denisReduction}).
                We believe that this reduction may is of independent interest.
                \item[\cref{sec:intro:efficientLearning}:]
                Second, we show that existing learning algorithms (namely, the \lreg{} algorithm of \citet{kalai2008agnostically}) can be used to solve the arising structured agnostic learning problem in $d^{\poly(\ell)}\cdot \poly(1/\eps)$ time, where $\ell$ is the degree of polynomials needed to $\poly(\eps)$ approximate $\Sstar$ with respect to $\cE(\thetaStar)$ in $L_2$-norm (typically $\ell\geq \poly(1/\eps)$).
            \end{description}

        \paragraph{Outline of the proof of \cref{infthm:polyTime}.}
            The algorithm and proof of \cref{infthm:polyTime} has the same two-part structure as the algorithm and proof of \cref{infthm:exponential}: it first learns $S\approx \Sstar$ and, then, finds $\thetaStar$ given membership access to $S.$
            The sub-routine to find $\thetaStar$ is the same as above (\cref{sec:overview:pmle}, \cref{thm:module:reduction}).
            To learn $S\approx \Sstar$, new and faster subroutines are needed as the sub-routine in the proof of \cref{infthm:exponential} runs in $d^{\poly(\ell)}\cdot \poly(1/\eps)$ time -- which is significantly slower than the desired $\poly(d/\eps)$ time as typically $\ell\geq \poly(1/\eps)$.
            We use different sub-routines depending on whether $\Sstar$ is an axis-aligned rectangle or a halfspace:
            \begin{itemize}
                \item When $\Sstar$ is an axis-aligned rectangle, we use a folklore algorithm for learning $d$-dimensional axis-aligned rectangles up to $\eps$-error from $\poly(d/\eps)$ positive samples in $\poly(d/\eps)$ time~\cite{shalev2014understanding}.
                \item \textbf{\cref{sec:intro:polytime}:} When $\Sstar$ is a halfspace and $\cE(\thetaStar)$ is Gaussian, we develop a new algorithm:
                    Suppose the covariates have an \textit{unknown} Gaussian distribution and the target concept $\Sstar$ is a halfspace.
                    This algorithm, given $\poly(d/\eps)$ positive samples, learns $\Sstar$ up to $\eps$-error in $\poly(d/\eps)$ time.
                    Roughly speaking, the algorithm identifies the normal vector of the halfspace using the third moment of the positive samples.
                    Here, using the third moment seems necessary as the first two moments conflate the information about the halfspace and the mean and covariance of the Gaussian distribution.
            \end{itemize}

        \paragraph{Outline of the proof of \cref{infthm:polyTime:exponential}.}
            The algorithm in \cref{infthm:polyTime:exponential}, first, learns an approximation $S$ to $\Sstar$ (which is assumed to be an axis-aligned rectangle) using the folklore algorithm mentioned above and, then, uses a sub-routine from the algorithm in \cref{infthm:exponential} to learn $\thetaStar$ given membership access to $S$.
            The proof of this result follows as a corollary.
    
            \medskip  
            
        \noindent We give detailed technical overviews of the proofs of \cref{infthm:exponential,infthm:polyTime} in \cref{sec:overview:thm:exponential,sec:overview:thm:polyTime} respectively and give the formal algorithms and proofs in \cref{sec:computationalEfficiency:exponential,sec:computationalEfficiency:polyTime} respectively. We also present a high level idea of the proof of \cref{infthm:sampleComplexity} in \cref{sec:outline:query} and the complete proof in \cref{sec:sampleEfficiency}.

    \subsection{Detailed Overview of the Proof of \texorpdfstring{\cref{infthm:exponential}}{Informal Theorem 4}}\label{sec:overview:thm:exponential}
        In this section, we give an overview of the proof of \cref{infthm:exponential} (which implies \cref{infthm:gaussians,infthm:linearRegression}), leaving details to \cref{sec:computationalEfficiency:exponential}.

    \subsubsection{Finding \texorpdfstring{$\theta^\star$}{θ*} with an Approximation of \texorpdfstring{$S^\star$}{S*} via Perturbed MLE}\label{sec:overview:pmle}
            The core difficulty in learning $\theta^\star$ is that $S^\star$ is unknown.
            To begin, we assume that $S^\star$ can \textit{somehow} be approximated from truncated samples.
            Given an approximation of $S^\star$, we present a new optimization problem for learning $\theta^\star$, which we call \textit{Perturbed MLE}.
            This problem deviates from the prior work (e.g., \cite{Pearson1902,fisher1934two,daskalakis2018efficient,lee2023learning}) as it does \textit{not} correspond to a maximum likelihood objective.
            Nevertheless, we show that $\theta^\star$ is close to the optimizer of this program.
            \begin{definition}[{Perturbed MLE}]\label{def:pmle}
                Suppose \cref{asmp:1:sufficientMass,asmp:1:polynomialStatistics,infasmp:2} hold. 
                Given a convex domain $\Omega\subseteq\Theta$ and a set $S\subseteq\R^d$, 
                define the $S$-Perturbed MLE over $\Omega$ as follows
                \begin{align*}
                    \argmin_{\theta\in \Omega}~  
                        {
                            \negLL{}_S\sinparen{\theta}
                            ~~\coloneqq~~
                            -\Ex_{x\sim \cE(\theta^\star,S^\star\cap S)}{\log{\cE(x; \theta, S)}}
                        }\,.
                        \yesnum\label{prog:misspecifiedMLE}
                \end{align*}
            \end{definition}
            Recall that the negative log-likelihood of $\theta$ is $-\Ex_{x\sim \cE(\theta^\star,S^\star)}{\log{\cE(x; \theta, S^\star)}}$.
            Hence, if $S=S^\star$, then \pmle{}'s objective is the negative log-likelihood whose minimizer is $\theta^\star$.
            However, when $S\neq S^\star$, \pmle{}'s objective is different from the negative log-likelihood and, hence, its minimizer may not be $\theta^\star$.
            A simple but revealing observation is that the objective of \pmle{} is different from the function obtained by substituting $S$ for $S^\star$ in the negative log-likelihood: %
            \[  
                -\Ex_{x\sim \cE(\theta^\star,S^\star)}{\log{\cE(x; \theta, S)}}\,.
                \yesnum\label{eq:substitutedLikelihood}
            \]
            (Note that we do not change $\cE(\theta^\star,S^\star)$ to $\cE(\theta^\star,S)$ in the expectation as we see samples from $\cE(\theta^\star,S^\star)$ not $\cE(\theta^\star,S)$.)
            We select \cref{prog:misspecifiedMLE} over \cref{eq:substitutedLikelihood} as \cref{eq:substitutedLikelihood} is not well defined when $S\not\subseteq S^\star$ making it tricky to work with.\footnote{
                To see this, observe that if $S\backslash S^\star$ is non-empty and overlaps with the support of $\cE(\theta^\star, S^\star)$, then \cref{eq:substitutedLikelihood} is not well defined as $\cE(x; \theta, S)=0$ at any $x\in S\backslash S^\star$ which is in support of $\cE(\theta^\star, S^\star)$.
            }

            Our algorithm to optimize the \pmle{} problem comes with the following guarantee.
            \begin{theorem}[Learning $\thetaStar$ Given $S\approx \Sstar$; see \cref{thm:module:reduction}]
                \label{thm:module:reduction:informal}
                Suppose \cref{asmp:1:sufficientMass,asmp:1:polynomialStatistics,infasmp:cov,infasmp:int,infasmp:start,infasmp:proj} hold.
                Fix any $\eps,\delta\in (0,1/2)$.
                Fix a set $S$ satisfying $\cE(S\triangle S^\star; \theta^\star)\leq \alpha \eps$.
                There is an algorithm that, given membership access to $S$, outputs a vector $\theta$, in $\poly(dm/\eps)$ time and \mbox{membership queries, such that, with probability at least $1-\delta$, $\tv{\cE(\theta)}{\cE(\theta^\star)} \leq \eps.$}
            \end{theorem}
            In the remainder of this section, we outline the proof of \cref{thm:module:reduction:informal}.

            \paragraph{A.~ Convexity of \pmle{}.}
                Since \pmle{} is not a likelihood, its optimizer can be far from $\theta^\star$ and, in fact, \pmle{} may also not have properties of likelihood like convexity that enable efficient optimization algorithms.
                We begin by verifying that $\negLL{}_S(\cdot)$ is convex.
                \begin{restatable}[{Gradient and Hessian of $\negLL{}_S(\cdot)$}]{fact}{gradHessNoisyMLE} \label{fact:gradOfNoisyMLE}
                    For any $\theta\in \R^m$ and $S\subseteq \R^d$, 
                    \[ \nabla\negLL_S(\theta) = \Ex_{\cE(\theta^\star,S^\star\cap S)}[t(x)] - \Ex_{\cE(\theta,S)}[t(x)] \qquadand \quad \nabla^2\negLL_S(\theta) = \cov_{\cE(\theta,S)}[t(x)]\,. \]
            \end{restatable}
            Since $\nabla^2\negLL_S(\cdot) $ is a covariance and, hence, positive semi-definite, for \textit{any} set $S$, \pmle{} is a convex program.
            The proof of \cref{fact:gradOfNoisyMLE} appears in \cref{sec:proofof:fact:gradOfNoisyMLE}.

            \begin{figure}[t!]
                \centering
                \begin{subfigure}[\label{fig:convexFunction}]{}
                    \includegraphics{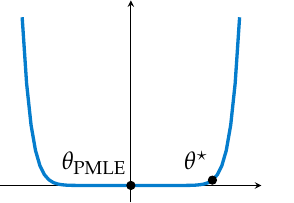}
                \end{subfigure}
                \begin{subfigure}[\label{fig:stronglyConvexFunction}]{}
                    \includegraphics{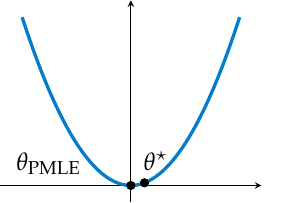}
                \end{subfigure}
                \begin{subfigure}[\label{fig:stronglyConvexFunctionOnDomain}]{}
                    \includegraphics{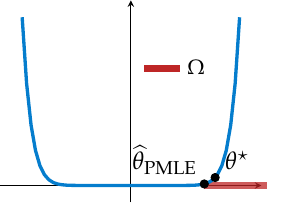}
                \end{subfigure}
                \vspace{-3mm}
                \caption{
                    Illustration that $\theta^\star$ can be far from $\theta_{\rm PMLE}$ even though the gradient at $\theta^\star$ is small (\cref{prop:intro:cor10cor11}).
                    To ensure $\theta^\star$ is close to $\theta_{\rm PMLE}$, we carefully select the domain $\Omega.$
                    The \textcolor{royalBlue}{blue line} is the objective of \pmle{} $\negLL_S(\cdot).$
                } 
                \label{fig:strongConvexity}
            \end{figure}
            
            \paragraph{B.~ Distance Between $\theta^\star$ And Minimizer of \pmle{}.}
                Let $\theta_{{\rm PMLE}}$ be the minimizer of $S$-\pmle{} over $\Omega.$
                Since $\negLL_S(\cdot)$ is convex, there is hope of finding $\theta_{{\rm PMLE}}$ via first-order methods.
                However, we want to find $\theta^\star$.
                Toward showing that $\theta^\star$ is close to $\theta_{{\rm PMLE}}$, we show that the gradient of $\negLL_S(\cdot)$ at $\theta^\star$ is small whenever $S\approx S^\star$.
                \begin{lemma}[Norm of Gradient at $\theta^\star$; see \cref{prop:intro:cor10cor11}]\label{prop:intro:cor10cor11:informal} 
                    For any $S\subseteq\R^d$ satisfying, $\frac{
                            \cE(S\triangle S^\star; \theta^\star)
                        }{
                            \cE(S^\star; \theta^\star)
                        }\leq \frac{3}{7}$,
                    \[
                        \norm{
                            \nabla \negLL_S(\theta)|_{\theta^\star}
                        }_2
                        \leq 
                        e^{\poly(1/\alpha)}
                        \cdot \sqrt{\frac{
                            \cE(S\triangle S^\star; \theta^\star)
                        }{
                            \cE(S^\star; \theta^\star)
                        }}
                        \,. 
                    \]
                \end{lemma}
                This shows that for any $S$ that is $\eps$-close to $S^\star$ in the following sense $\cE\inparen{S\triangle S^\star; \theta^\star}
                    \leq \eps^2\cdot \cE\inparen{S^\star; \theta^\star}$
                the gradient at $\theta^\star$ has a small $L_2$ norm:
                $\norm{
                            \nabla \negLL_S(\theta)|_{\theta^\star}
                }_2\leq  O(\eps)\cdot e^{\poly(1/\alpha)}$.
                The proof of \cref{prop:intro:cor10cor11} appears in \cref{sec:perturbed_mle_proof}.

                Consider any $S$ satisfying the above condition.
                Convexity of $\negLL_S(\cdot)$ combined with the bound on $\norm{
                            \nabla \negLL_S(\theta)|_{\theta^\star}
                }_2$ implies that $\theta^\star$ is $O(\eps)\cdot e^{\poly(1/\alpha)}$ close to $\theta_{{\rm PMLE}}$ in \textit{function value}.
                However, in general, $\theta^\star$ may still be far from $\theta_{{\rm PMLE}}$ -- see \cref{fig:convexFunction}.
                One way to show that $\theta^\star$ is close to $\theta_{{\rm PMLE}}$ is to show that $\negLL_S(\cdot)$ is strongly convex -- see \cref{fig:stronglyConvexFunction}.        
                Unfortunately, $\negLL_S(\theta)$ is not strongly convex for all $\theta\in \Theta$.
                Instead, as explained later in this section, we create a subset $\Omega\subseteq \Theta$ where $\negLL_S(\cdot)$ is $\alpha^{-\poly(1/\alpha)}$-strongly convex.
                Since $\theta_{{\rm PMLE}}$ is the minimizer of $\negLL_S(\cdot)$ subject to the \textit{constraint} that $\theta\in \Omega$, the bound on the gradient-norm with strong convexity implies that $\theta^\star$ is $O(\eps\sigma)$ close to $\theta_{{\rm PMLE}}$ -- see \cref{fig:stronglyConvexFunctionOnDomain}.
                This is why we introduce the constraint in the \pmle{} problem.
                
                Regarding efficiently solving \pmle{}, we can use projected stochastic gradient descent (PSGD) along the lines of prior works \cite{daskalakis2018efficient,lee2023learning}, and we present the details in \cref{sec:reduction_to_known_truncation:solving_pmle_psgd}.

            \paragraph{C.~ Domain where $\negLL_S(\cdot)$ is Smooth and Strongly Convex.}
                Note that $\nabla^2 \negLL_S(\cdot)$ is different from the Fisher Information matrix in \cref{asmp:cov} and, hence, at this point, we have neither shown that $\negLL_S(\cdot)$ is smooth nor strongly convex over $\Theta$. %
                Toward establishing smoothness and strong convexity, one can show that (see \cref{lem:pres_sc_smooth})
                \[
                    \Omega\inparen{\inparen{\frac{\cE(S; \theta)}{k}}^{k}}\cdot \lambda \cdot I
                    \preceq
                    \nabla^2 \negLL_S(\theta)  
                    \preceq 
                    \frac{\Lambda}{\cE(S; \theta)}\cdot I\,.
                    \yesnum\label{eq:intro:strongConvexity}
                \]
                Therefore, controlling the mass $\cE(S; \theta)$ is crucial in showing the strong convexity and smoothness of $\negLL_S(\cdot)$.
                If $S\approx S^\star$, then we know that $\cE(S; \theta^\star)\geq \Omega(\alpha)$, however, as one moves $\theta$ further from $\theta^\star$, this mass can become very small.
                To control $\cE(S; \theta)$ we show that one can \textit{efficiently} find a region $\Omega$ around \mbox{$\theta^\star$ over which $\cE(S; \theta)$ is bounded away from 0.}
            \begin{lemma}[{Ensuring Strong Convexity; see \cref{lem:findingTheta0Formal,lem:strongConvexityOnOmega}}]
                \label{lem:findingTheta0}
                Fix any $\delta\in (0,1/2)$.
                There is an algorithm that, 
                    given $n=O\sinparen{
                    m
                    \log^3\inparen{\sfrac{1}{\alpha}}
                    \log^2\inparen{\sfrac{1}{\delta}}
                    }$ independent samples $x_1,x_2,\dots,x_n$ from $\cE(\theta^\star, S^\star)$,
                    in $\poly(n)$ time, outputs a vector $\theta_0\in \R^m$ such that the following holds:
                    Define $\Omega$ as $\Theta$'s intersection with an $L_2$-ball centered at $\theta_0$, $\Omega\coloneqq B\inparen{\theta_0, \poly\inparen{\sfrac{1}{\alpha}}}\cap \Theta.$
                    For all $\theta\in \Omega$ and $T\subseteq\R^d$ 
                    \[  
                        \cE(T; \theta)\geq \cE(T; \theta^\star)^{\poly(1/\alpha)}\,.
                        \yesnum\label{eq:findingTheta0:mass}
                    \]
            \end{lemma}
            Suppose we can efficiently find $\theta_0$ as promised in the above result.
            Then \cref{eq:findingTheta0:mass} shows that if  $S\approx S^\star$ (and, hence, $\cE(S; \theta^\star)\geq \Omega(\alpha)$), then for any $\theta\in \Omega$, $\cE(S; \theta)\geq \alpha^{\poly(1/\alpha)}$.
            Therefore, by \cref{eq:intro:strongConvexity}, $\negLL_S(\cdot)$ is $\alpha^{\poly(1/\alpha)}$-strongly-convex and $\alpha^{-\poly(1/\alpha)}$-smooth over $\Omega.$
            To efficiently find $\theta_0$, we use \cref{infasmp:start} (or the formal versions: \cref{asmp:moment,asmp:proj}).
            \begin{figure}[t!]
                \centering 
                \subfigure[\label{fig:thetaZeroProperties:parameterDistance}]{}{
                    \includegraphics[]{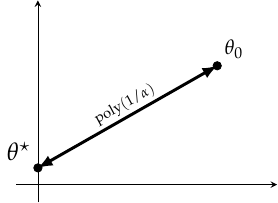} 
                }
                \subfigure[\label{fig:thetaZeroProperties:distribution}]{}{
                    \includegraphics[]{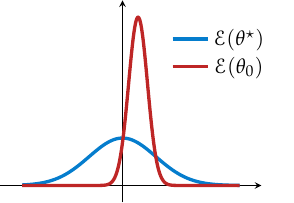} 
                }
                \subfigure[\label{fig:thetaZeroProperties:mass}]{}{
                    \includegraphics[]{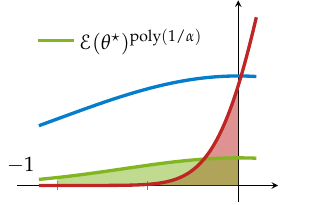} 
                }
                \vspace{-6mm}
                \caption{
                    Illustration of the properties of $\theta_0$ (see \cref{lem:findingTheta0}):
                        $\theta_0$ is at most $\poly(1/\alpha)$ far from $\theta^\star$ (\cref{fig:thetaZeroProperties:parameterDistance}),
                        $\cE(\theta_0)$ can have a constant TV distance with $\cE(\theta^\star)$ (\cref{fig:thetaZeroProperties:distribution}), and, yet, for any set $T$ (e.g., $T=[-1,0]$), $\cE(T;\theta_0)$ is lower bounded by $\cE(T;\theta^\star)^{\poly(\sfrac{1}{\alpha})}$ (\cref{fig:thetaZeroProperties:mass}).
                }
                \label{fig:thetaZeroProperties}
            \end{figure}

    \subsubsection{Obtaining Unlabeled Samples and Reduction to Agnostic Learning}\label{sec:intro:denis}
        So far we have shown that to learn $\theta^\star$ it is sufficient to learn $S^\star$ to a sufficiently high accuracy. %
        However, this itself is challenging as we only see samples falling within $S^\star$ (i.e., positive samples) and do not see any samples that fall outside $S^\star$ (i.e., negative samples.)

        \paragraph{Learning From Positive Samples.}
            \citet{natarajan1987learning} characterized the classes that are (inefficiently) PAC learnable from positive samples alone and, unfortunately, showed that only very restricted classes can be learned without negative examples.
            For instance, even the simple and practical class of two-dimensional halfspaces is not learnable from positive samples. 
            While we are admittedly in a more benign setting than \citet{natarajan1987learning} as we know that the underlying distribution $\cE(\theta^\star)$ belongs to an exponential family, it is unclear how to leverage this without learning the underlying distribution --- our goal in the first place. 
    
        \paragraph{Alternatives to Learning From Positive Samples.} 
            A standard way to overcome \citet{natarajan1987learning}'s characterization is to gain access to negative samples or unlabeled samples from $\cE(\theta^\star)$.
            With negative samples, we can learn $S^\star$ via empirical risk minimization (ERM) and, with unlabeled samples, we can use a reduction to agnostic learning by \citet{denis1990positive}.
            However, to classify a sample as negative, we need membership access to $S^\star$, which is what we are trying to learn. %
            Further, unlabeled samples correspond to samples from $\cE(\theta^\star)$, and with sample access to $\cE(\theta^\star)$ we can learn $\theta^\star$ directly via maximum likelihood estimation.

        \paragraph{Obtaining Unlabeled Samples with Covariate Shift.}
            Toward enabling learning, we show how to generate samples from a distribution close to $\cE(\thetaStar)$.
            \begin{lemma}[See \cref{thm:module:unlabeledSamples:measureGuarantees}]\label{thm:informal:module:unlabeledSamples:measureGuarantees}
                Suppose \cref{asmp:1:sufficientMass,asmp:1:polynomialStatistics,infasmp:cov,infasmp:int,infasmp:start} hold.
                For any $\theta_0\in \Theta(\eta)$ satisfying $\norm{\theta_0-\thetaStar}_2\leq \poly(1/\alpha)$, it holds that for all sets $T\subseteq \R^d$
                \[
                    e^{-C}\cdot \cE(T;\theta_0)^{C}
                    \leq 
                    \cE(T;\thetaStar)
                    \leq 
                    e^{C}\cdot \cE(T;\theta_0)^{1/C}\,,
                    \quadtext{where} C=\poly(1/\alpha)\,.
                    \yesnum\label{eq:guaranteeOnUnlabeledSamples}
                \]
                Moreover, under \cref{infasmp:start}, there is a $\poly(dm)$ time algorithm to find such a $\theta_0$.
            \end{lemma}
            Here, $\cE(\theta_0)$ is close to $\cE(\thetaStar)$ and we generate unlabeled samples from $\cE(\theta_0)$.
            When $C=2$, \cref{eq:guaranteeOnUnlabeledSamples} is equivalent to requiring that the $\chi^2$-divergence between $\cE(\theta^\star)$ and $\cE(\theta_0)$ (in both directions) is finite.
            However, for $C>2$ (which is the best we can ensure), this is weaker than bounding the $\chi^2$-divergence. 

        \paragraph{Reduction to Agnostic Learning Which Is Robust to Covariate Shift.}
            Next, we reduce the problem of learning from positive samples and unlabeled samples with covariate shift to an agnostic learning problem.
            Concretely, we consider agnostic learning with the following distribution.
        \begin{restatable}[]{definition}{distributionD}\label{def:distributionD}
            Given a distribution $\cP$ over $\R^d$ of positive samples, 
            a distribution $\cU$ over $\R^d$ of positive-and-negative samples (potentially with covariate shift), and 
            a constant $\rho\in [0,1]$,
            define $\cD_{\rho,\cP,\cU}$
            to be the following distribution over $\R^d\times \zo$:
            \[
                \cD_{\rho,\cP,\cU} \coloneqq \rho\cdot \cD_{\cU} + (1-\rho)\cdot \cD_{\cP}\,,
            \]
            where $(x,y)\sim \cD_{\cU}$ is generated by drawing $x\sim \cU$ and setting $y=0$, and $(x,y)\sim \cD_{\cP}$ is generated by drawing $x\sim \cP$ and setting $y=1$.
        \end{restatable}
        To gain some intuition about $\cD_{\rho,\cP,\cU}$, 
        \begin{figure}[t!] 
            \vspace{-2mm}
            \hspace{30mm}
            \includegraphics[]{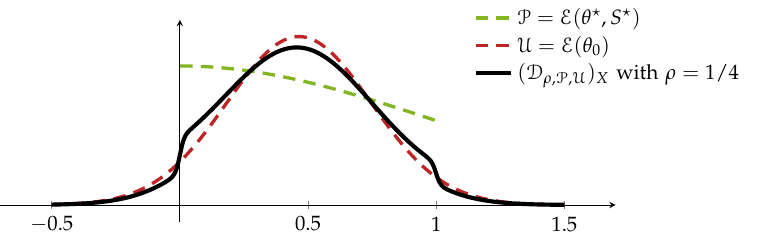}
            \vspace{-4mm}
            \caption{
                Illustration of the marginal distributions constructed in the reduction in \cref{thm:efficientLearning:denisReduction} with $\cQ=\cE(\thetaStar)$, $P=\Sstar=[0,1]$, $\cP=\cE(\thetaStar, \Sstar)$, and $\cU=\cE(\theta_0)$.
                Here, $\cE(\cdot)$ is the family of normal distributions and $\theta_0$ is the parameter promised in \cref{lem:findingTheta0}.
                Note that $\inparen{\cD_{\rho,\cP,\cU}}$ places mass outside of $S^\star$ unlike $\cE(\theta^\star, S^\star)$.
            }
            \label{fig:denisReduction}
        \end{figure}
        for each $x$, consider the marginal of $\cD_{\rho,\cP,\cU}$ over labels $\zo$ conditional on drawing $x$. 
        Let $P$ be the support of positive samples (which is $\Sstar$ in our setting).
        The labels of $\cD_{\rho,\cP,\cU}$ denote whether the point $x$ is in $P$, but have some noise:
        For points outside of $P$, the label is always $0$ and there is no noise.
        To see this, note that $\cP$ is supported over $P$ and, hence, any $x\not\in P$ must have been drawn from $\cU$ and, hence, is given a label 0 by construction.
        For points inside of $P$, there is noise and the label can be either 0 or 1.
        It is 0 if $x$ is drawn from $\cU$ and otherwise 1.
        In the ideal case, where there is no noise, one can show that the hypothesis with the minimum error with respect to $\cD_{\rho,\cP,\cU}$ is $P$. 
        Instead, we show that the amount of noise is ``manageable'' and minimizing the error with respect to these labels results in a set close to $P$. 
        This leads us to the following result.
        \begin{restatable}[Reduction from Positive-Unlabeled Learning to Agnostic Learning]{theorem}{learningReduction}\label{thm:efficientLearning:denisReduction}
            Fix any constant $C\geq 1$, set of positive examples $P$, and distribution $\cQ$ over $\R^d$.
            Let $\cP$ be the truncation of $\cQ$ to $P$.
            Let $\cU$ be any distribution with the same support as $\cQ$ satisfying:
            for any set $T\subseteq \R^d$,
            \[
                e^{-C}\cdot \cU(T)^{C}
                    \leq 
                    \cQ(T)
                    \leq 
                    e^{C}\cdot \cU(T)^{1/C}
                \,.
                \yesnum\label{eq:guaranteeOnCovariateShift}
            \]
            For any $0<\rho<1/(3C)$, the distribution $\cD=\cD_{\rho,\cP,\cU}$ from \cref{def:distributionD} satisfies:
            \[
                \cQ\inparen{S \triangle P} 
                \leq 
                O(\rho) +
                O\inparen{\frac{e^{C}}{\rho}}
                \cdot 
                \inparen{
                    \err_\cD(S) - \err_\cD(S_{\rm \opt})
                }^{\sfrac{1}{C}}
                \,.
            \]
            Where for any set $T$, the error $\Err_{\cD}(T)$ is defined as $\Err_{\cD}(T)\coloneqq \Pr_{(x,y)\sim \cD}\insquare{y\neq \mathds{1}_{T}(x)}$ and $S_{\opt}$ minimizes the error (i.e., $S_{\opt} \coloneqq \argmin_{T \subseteq \R^d} \Err_\cD(T)$).
        \end{restatable}
        This reduction generalizes \citet{denis1990positive}'s reduction, which only works when $\cQ=\cU$, and may be of independent interest.
        In our context, we consider the following choices 
        \[
            \cQ=\cE(\thetaStar)\,,\quad 
            P=\Sstar\,,\quad 
            \cP=\cE(\thetaStar,\Sstar)\,,\quad 
            \cU=\cE(\theta_0)\,,
        \]
        where $\theta_0$ is the parameter from \cref{thm:informal:module:unlabeledSamples:measureGuarantees}.
        $\cU=\cE(\theta_0)$ satisfies the requirements in \cref{thm:efficientLearning:denisReduction} with $C=\poly(1/\alpha)$ (\cref{thm:informal:module:unlabeledSamples:measureGuarantees}) and, hence, \cref{thm:efficientLearning:denisReduction} is applicable.
        It combined with standard analysis implies that ERM with respect to $\cD_{\poly(\alpha),\cE(\thetaStar,\Sstar),\cE(\theta_0)}$  (non-efficiently) learns $\Sstar$.
        This, in turn, allows us to learn $\theta^\star$ by solving the \pmle{} problem (\cref{sec:overview:pmle}).

        \paragraph{Overview of Proof of \cref{thm:efficientLearning:denisReduction}.}
        Recall from our earlier discussion that for each $(x,y)\sim \cD_{\rho,\cP,\cU}$, the label $y$ roughly denotes whether $x$ is in $P$ but has some noise.
        As mentioned before, in the ideal case, with no noise, one can show that the minimum-error hypothesis with respect to $\cD_{\rho,\cP,\cU}$ is $P$. 
        To prove \cref{thm:efficientLearning:denisReduction}, we show that the amount of noise is ``manageable'' and minimizing the error with respect to these noisy labels results in a set $O(\rho)$-close to $P$.
        In a bit more detail: when there is no covariate shift (i.e., $\cU=\cQ$), the noise rate can be uniformly bounded away from $\sfrac{1}{2}$.
        An important consequence of this is that the Bayes optimal classifier is $P$ and, hence, $P$ minimizes the error with respect to $\cD_{\rho,\cP,\cU}$.
        This is how \citet{denis1990positive} reduces learning with positive and unlabeled samples to agnostic learning.
        In our setting, however, there can be covariate shift and, due to this, we cannot uniformly bound the noise rate.
        In fact, the noise rate can be arbitrarily close to 1 at certain points.
        To control the error introduced by high noise rates, we show that the mass of points where the noise rate is above, e.g., $\sfrac{1}{3}$, is at most $O(\rho)$ (\cref{lem:module:efficientLearning:denisReduction:massBz}).
        This along with the guarantees on the amount of covariate shift (\cref{eq:guaranteeOnCovariateShift}) leads to \cref{thm:efficientLearning:denisReduction}.
        The full proof of \cref{thm:efficientLearning:denisReduction} appears in \cref{sec:module:reductionToLearning}.

        \begin{remark}[{Sampling From $\cE(\theta_0)$}]\label{rem:sampling}
            To utilize our reduction to agnostic learning, we need to be able to construct samples for the agnostic learning problem.
            This requires sampling from $\cE(\theta_0)$.
            For this, it suffices to have an oracle that, given $\theta$, outputs a sample $z\sim \cQ$ where $\cQ$ is $\delta$-close to $\cE(\theta_0)$ in TV-distance in $\poly(m/\delta)$ time.
            This is sufficient because using this oracle and standard analysis (\cref{lem:exactSampleAccess}), we can obtain $n$ \textit{exact} samples from $\cE(\theta_0)$ in $\poly(nm/\delta)$ time.
            Such sampling oracles are available for the Gaussian and product Exponential distribution:
            for Gaussians, relevant oracles operate, e.g., via the Box-Muller transform, and, for product Exponential distribution, one can sample \mbox{each coordinate independently using inverse CDF sampling \cite{Bishop:2006:PRM:1162264,vono2021highdimensional}.}
        \end{remark}

        \subsubsection{Efficiently Learning \texorpdfstring{$S^\star$}{S*} via \lreg{}}\label{sec:intro:efficientLearning}
                We have reduced our problem to solving an agnostic learning problem. %
                However, many computational hardness results are known for agnostic learning even with respect to simple hypothesis classes such as halfspaces \cite{guruswami2009hardness, feldman2006new, daniely2016complexity}. 
                Hence, most of the efficient agnostic learning algorithms make assumptions on the underlying distribution or the type of {noise} allowed in the labels.
                Next, we discuss one such algorithm -- \lreg{} -- and show that the learning problem in \cref{thm:efficientLearning:denisReduction} satisfies its requirements.
            \lreg{} comes with the following guarantee.

            \begin{theorem}[Guarantees of \lreg{} \cite{kalai2008agnostically}]\label{thm:l1Regression}
                    Fix any $\eps,\delta\in (0,1)$, distribution $\cD$ on $\R^d\times \zo$, and hypothesis class $\hyH$ satisfying, and $\ell\geq \poly(1/\eps)$ such that $\hyH$ is $\eps$-approximable with respect to $\cD$ in the $L_1$-norm (see \cref{sec:preliminaries} for a definition of approximability).
                    The \lreg{} algorithm, given $(\eps,\delta)$ and sample access to $\cD$, in time $\poly\inparen{\sfrac{d^\ell}{\eps},\log\inparen{\sfrac{1}{\delta}}}$, outputs a set $S\subseteq \R^d$ such that, with probability $1-\delta$, $\Err_{\cD}\inparen{S}\leq \OPT+\eps.$
                \end{theorem}
                Thus, as long as $S^\star$ is well-approximated by polynomials under distribution $\cD_X$ (of samples in the reduction to agnostic learning), we can use \lreg{} and \cref{thm:efficientLearning:denisReduction} (specifically, its version in \cref{thm:module:efficientLearning:denisReduction}) to learn $S^\star$.
                That said, almost all approximability results are proved for Gaussians or, if not Gaussian, at least log-concave distributions (e.g., \cite{kalai2008agnostically,kane2011gsa,kane2013learning}). 
                This is problematic as $\cD_X$ is a mixture of $\cE(\theta_0)$ and $\cE(\theta^\star, S^\star)$ (\cref{def:distributionD}), which is in general neither Gaussian nor log-concave.

                To use \lreg{} despite $\cD_X$ not being log-concave, we rely on the following assumption on the parameter space $\Theta$.
                As \cref{infasmp:2}, we show that this assumption is satisfied by Gaussian and Exponential distributions after simple pre-processing (\cref{sec:preprocess}).
                \begin{infassumption}[$\chi^2$-Bridge; see \cref{asmp:2}]\label{infasmp:bridge}
                    For any $\theta_1,\theta_2\in \Theta$, there exists $\theta\in \R^m$ such that $\chidiv{\cE(\theta_1)}{\cE(\theta)}, 
                    \chidiv{\cE(\theta_2)}{\cE(\theta)}
                    \leq e^{\poly(1/\alpha)}$.
                \end{infassumption}
                This assumption guarantees that for any pair of distributions $\cE(\theta_1)$ and $\cE(\theta_2)$ (for $\theta_1,\theta_2\in \Theta$), there is a ``bridge'' distribution that is $e^{\poly(1/\alpha)}$-close to both of $\cE(\theta_1)$ and $\cE(\theta_2)$ in $\chi^2$-divergence. %
                This is much weaker than requiring that $\cE(\theta_1)$ and $\cE(\theta_2)$ are close to each other.
                To gain some intuition, consider the example in \cref{fig:chi2bridge} where the bridge distribution $\cE(\theta)$ is close to the two distributions $\cE(\theta_1)$ and $\cE(\theta_2)$: $\chidiv{\cE(\theta_1)}{\cE(\theta)}\leq 10^2$ and 
                        $\chidiv{\cE(\theta_2)}{\cE(\theta)} \leq  10^2$.
                        And, yet, the two distributions are far from each other:  $\chidiv{\cE(\theta_1)}{\cE(\theta_2)}\geq 10^{86}$ and $\chidiv{\cE(\theta_2)}{\cE(\theta_1)}\geq10^{86}$.
                \begin{figure}[t!] 
                    \hspace{26.75mm}
                    \includegraphics[]{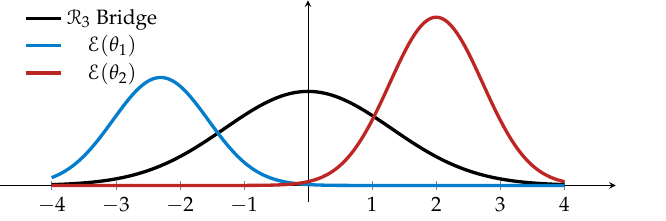}
                    \caption{An illustration of the $\chi^2$-Bridge promised to exist in \cref{infasmp:bridge} (or its formal version \cref{asmp:bridge}).
                    This bridge always exists for Gaussians and product exponential distributions after pre-processing described in  \cref{sec:preprocess}.
                    In this example, the two distributions $\cE(\theta_1)$ and $\cE(\theta_2)$ are far from each other in $\chi^2$-divergence ($\chidiv{\cE(\theta_1)}{\cE(\theta_2)}, \chidiv{\cE(\theta_2)}{\cE(\theta_1)}\geq 10^{86}$) and the bridge distribution $\cE(\theta)$ (denoted by the black-line) is close to both distributions in $\chi^2$-divergence ($\chidiv{\cE(\theta_1)}{\cE(\theta)}, 
                        \chidiv{\cE(\theta_2)}{\cE(\theta)}
                        \leq 10^2$).}
                    \label{fig:chi2bridge}
                \end{figure}
                Having a bridge distribution is useful as, we show that, if $\cE(\theta_3)$ is close to $\cE(\theta_1)$ and $\cE(\theta_2)$ in $\chi^2$-divergence then it is also close to any mixture of these distributions in $\chi^2$-divergence.
                In particular, this allows us to show that, for an appropriate $\theta\in \Theta$, $\cD_X$ is $e^{\poly(1/\alpha)}$ away from $\cE(\theta)$ in terms of $\chi^2$-divergence.
                Hence, even though $\cD_X$ is not log-concave itself, it is close to a log-concave distribution (namely, $\cE(\theta)$).
                This observation is sufficient to use \lreg{} with respect to $\cD_X$ and we prove this in the following result.

                \begin{restatable}[See \cref{lem:module:efficientLearning:upperBoundOnApproximability}]{lemma}{upperBoundonApproximability}\label{lem:module:efficientLearning:upperBoundOnApproximability:informal}
                    Let $\cD$ be the distribution in \cref{thm:module:efficientLearning:denisReduction}.
                    Suppose \cref{infasmp:bridge} holds.
                    Fix $\phi=\eps^2\cdot e^{-\poly(1/\alpha)}$.
                    Fix $\ell$ such that $\hyS$ is $\phi$-approximable by degree-$\ell$ polynomials in $L_2$-norm with respect to the family $\cE$ (where, typically, $\ell\geq \poly(1/\phi)$).
                    Then $\hyS$ is $\eps$-approximable by degree-$\ell$ polynomials in $L_1$-norm with respect to $\cD_X$.
                \end{restatable}
                This theorem, combined with \cref{thm:l1Regression}, allows us to conclude \cref{infthm:exponential}.
                For the special case where $\cE$ is the family of Gaussian distributions, we can upper bound $\ell$ in the above result (\cref{lem:module:efficientLearning:upperBoundOnApproximability}) in terms of the Gaussian surface area of $\hyS$ (defined below).
                This is useful as upper bounds on the Gaussian surface area are known for many hypothesis classes (see \cref{tab:gsa:GaussianSurfaceArea}).
                
                \begin{definition}[{Gaussian Surface Area}]\label{def:gsa:gaussianSurfaceArea}
                    Given a Borel set $S\subseteq\R^d$, $\delta > 0$, define $S_\delta$ to be the set of points $\delta$-close to $S$, i.e., $S_\delta \coloneqq \inbrace{x\in \R^d\colon {\rm dist}(x,S)\leq \delta}$.
                    The Gaussian Surface Area (GSA) of $S$ is 
                    $\Gamma(S)\coloneqq \operatornamewithlimits{\lim\inf}_{\delta\to 0} \sfrac{\cN\inparen{S_\delta\backslash S; 0, I}}{\delta}.$
                    Given a family of sets $\hyH$, define its GSA as $\Gamma(\hyH)\coloneqq \sup_{H\in \hyH}\Gamma(H)$.
                \end{definition}
                The upper bound on $\ell$ from \cref{lem:module:efficientLearning:upperBoundOnApproximability} is as follows.
                \begin{theorem}[GSA and Approximability \cite{klivans2008gaussian}]\label{thm:GSAupperbound}
                    For $\eps\in (0,1)$, family of sets $\hyH$, and $\ell\geq O\inparen{\sfrac{\Gamma(\hyH)^2}{\eps^4}}$,
                    it holds that $\hyH$ is $\eps$-approximable by degree-$\ell$ polynomials with respect to the family of Gaussian distributions.
                    
                \end{theorem}
                This theorem, combined with \cref{infthm:exponential} and the pre-processing algorithm to satisfy \cref{asmp:2} in \cref{sec:preprocess} implies \cref{infthm:gaussians}.

            {
        \begin{table}[ht!]
            \centering
                \caption{
                    Summary of known upper bounds on Gaussian Surface Area.
                    The last column gives the corresponding running time and sample complexity in \cref{infthm:gaussians}  when $\cE$ is the family of Gaussian distribution. %
                    When $\hyS$ is the family of halfspaces or axis-aligned rectangles, then we obtain $\poly(d/\eps)$ algorithms which are discussed in the next section (\cref{sec:intro:polytime}).
                }
                \vspace{2mm} 
                \small 
            \begin{tabular}{c c c}
             Set Family $\hyS$ & Gaussian Surface Area & Runtime and Sample Complexity\\
            \midrule
            Degree-$t$ Polynomial Threshold Functions  &
                $O(t)$ \cite{kane2011gsa} & 
                        $d^{O(t^2)\cdot \poly(1/\eps)}$\\
            Intersections of $t$ Halfspaces & 
                    $O(\sqrt{\log t})$ \cite{kalai2008agnostically} & 
                        $d^{O(\log{t})\cdot \poly(1/\eps)}$\\
            General Convex Sets &
                $O(d^{1/4})$ \cite{ball1993gsaConvexSets} & 
                    $d^{O(\sqrt{d})\cdot \poly(1/\eps)}$ \\
            \end{tabular}
            \label{tab:gsa:GaussianSurfaceArea}
        \end{table} 
    }

    \subsection{Detailed Overview of Proof of \texorpdfstring{\cref{infthm:polyTime}}{Informal Theorem 3}}\label{sec:overview:thm:polyTime}\label{sec:intro:polytime} 
        In this section, we give an overview of the proof of \cref{infthm:polyTime}, leaving details to \cref{sec:computationalEfficiency:polyTime}.

        The algorithm and proof of \cref{infthm:polyTime} has a two-part structure: it first learns $S\approx \Sstar$ and, then, finds $\thetaStar$ given membership access to $S.$
        The sub-routine to find $\thetaStar$ given membership access to $S\approx \Sstar$ is the same as the one discussed in \cref{sec:overview:pmle} (\cref{thm:module:reduction:informal}).
        To learn $S\approx \Sstar$, we use different sub-routines depending on whether $\Sstar$ is an axis-aligned rectangle or a halfspace.
        If $\Sstar$ is an axis-aligned rectangle, we use the folklore $\poly(d/\eps)$ time algorithm for learning axis-aligned rectangles up to $\eps$-error from just positive samples \cite{shalev2014understanding}.
        If $\Sstar$ is a half-space and $\cE(\thetaStar)$ is Gaussian, we develop a new algorithm that is described below.
        \begin{theorem}[Learning Halfspaces From Positive Examples Under Gaussian Marginals]\label{thm:halfspaceLearner}
            Consider any $\alpha > 0$, \emph{unknown} $d$-dimensional Gaussian $\cN(\muStar, \SigmaStar)$, and \emph{unknown} halfspace hypothesis $\Sstar$ such that $\cN(\Sstar;\muStar,\SigmaStar)\geq \alpha$.
            There is an algorithm that, given $\alpha>0$, $\eps,\delta\in (0,1/2)$, and $n=\poly(d/\eps)\cdot\polylog(1/\delta)$ positive samples (drawn independently from $\cN(\muStar,\SigmaStar,\Sstar)$) ,
            outputs a set $S$ such that with probability $1-\delta$, 
            $\Pr_{x\sim \cN(\muStar,\SigmaStar)}\inparen{\mathds{1}_S(x) \neq \mathds{1}_{\Sstar}(x)}\leq \eps$.
            The algorithm runs in $\poly(n)$ time.
        \end{theorem}
        Combined with the sub-routines described above \cref{thm:halfspaceLearner} completes the proof of \cref{infthm:polyTime}.
        In the remainder of this section, we give an overview of the proof of \cref{thm:halfspaceLearner}.
        Let $S^\star$ be the following halfspace
        \[
            S^\star = \inbrace{x\in \R^d\colon x^\top w^\star\geq \tau^\star}
            \qquadtext{where}
            \norm{\wStar}_2=1 \quadand \tauStar\in \R
            \,.
        \]
        Let $x_1,x_2,\dots,x_n$ be independent samples from the truncated distribution $\cN(\mu^\star,\SigmaStar,S^\star)$.
        The key result that enables learning $\Sstar$ is that 
        \[
            \Ex_{}\insquare{
                \inparen{x - \gamma}^{\otimes 3}
            }
            \propto 
            \wStar\otimes\wStar\otimes\wStar
            \qquadwhere
                x\sim \cN(\mu^\star,\SigmaStar,S^\star)
                \quadand
                \gamma \coloneqq \Ex_{}[x]\,.
            \yesnum\label{eq:polytime:proportionality}
        \]
        This implies that one can identify $\wStar$ by computing the third central moment of the truncated samples.
        Moreover, considering the third (or a higher) moment is crucial: this is because, while the first and second moments of the truncated samples are also related to $\wStar$ -- namely, $\Ex\insquare{x}=\muStar + c_1\wStar$ and $\cov\insquare{x}=\SigmaStar+c_2\wStar {\wStar}^\top$ for some $c_1,c_2>0$ -- identifying $\wStar$ from these moments requires knowledge of $\muStar$ or $\SigmaStar$ respectively.
    
        Returning to the third moment: it enables the identification of $\wStar$ but not of the threshold $\tauStar$, i.e., it allows the identification of $\Sstar$ up to the translation.
        This is useful because learning the threshold $\tauStar$ is a significantly easier problem because the family of halfspaces orthogonal to $\wStar$ is a minimally consistent class and can be learned from just positive samples \cite{natarajan1987learning}.
        Concretely, if $\wStar$ is exactly known, then $\Sstar$ can be learned by selecting the largest threshold $\tau$ that ensures that the resulting halfspace contains all positive samples.
    
        To formalize this, we need some additional notation. 
        Consider the linear transformation $z\mapsto {\SigmaStar}^{-1/2}(z-\muStar)$ that transforms samples from $\cN(\muStar,\SigmaStar)$ to samples from the standard gaussian distribution $\cN(0,I)$.
        Since $\cN(0, I)$ is invariant to rotation, we can further rotate the space to ensure that $\Sstar$ is an axis-aligned halfspace perpendicular to $e_1$ -- the first element of the canonical basis -- and $w^\star=e_1$. 
        Let $U$ be the unitary matrix describing this rotation. 
        Hence
        \[
            e_1 = U {\SigmaStar}^{-1/2} w\,.
            \yesnum\label{eq:polytime:rotation}
        \]
        Now, we are ready to discuss the main lemmas in the proof of \cref{thm:halfspaceLearner}.
        \begin{restatable}[]{lemma}{halfspaceLearnerMoment}\label{lem:halfspaceLearner:moment}
            Let $\tau$ be the unique value such that $\frac{1}{\sqrt{2\pi}}\int_{\tau}^\infty e^{-t^2/2}\d t = \cN(\Sstar; \muStar, \SigmaStar)$.
            Let $z$ be a standard Normal random variable truncated to $[\tau,\infty)$.
            As an operator over $\inbrace{v\otimes v\otimes v \colon v\in \R^d}$, it holds that 
            \[
                \Ex_{}\insquare{
                    \inparen{x - \gamma}^{\otimes 3}
                }
                =
                \Ex_{}\insquare{\inparen{z - \Ex[z]}^3}
                \cdot 
                \wStar\otimes\wStar\otimes\wStar
                \qquadtext{where}
                \gamma\coloneqq \Ex[x]
                \,.
            \]
        \end{restatable} 
        To estimate $\wStar$ from a finite number of samples, we need $\Ex_{}\insquare{\inparen{z - \Ex[z]}^3}$ to be bounded away from 0.
        The next lemma shows this is a general property of truncated Gaussian distributions. 
        \begin{restatable}[]{lemma}{halfspaceLearnerconstantLB}\label{lem:halfspaceLearner:constantLB}
            If $z$ is a standard Normal random variable truncated to $[\tau,\infty)$ for some $\tau\in \R$, then 
            \[
                \Ex_{}\insquare{\inparen{z - \Ex[z]}^3}
                \geq 
                \Omega\inparen{\min\inbrace{1, \tau^{-12}}}\cdot e^{-\tau^2/2}
                \,.
            \]
        \end{restatable}
        As long as $\cN(\Sstar,\muStar,\SigmaStar)$ is bounded away from 0 and 1, using the definition of $\tau$ in \cref{lem:halfspaceLearner:moment} and the above lemma, with some calculation, one can show that $\Ex_{}[\inparen{z - \Ex[z]}^3]$ is bounded away from 0.
        However, in the simple case of $\cN(\Sstar,\muStar,\SigmaStar)\approx 1$, i.e., there is no truncation, $\Ex_{}\insquare{\inparen{z - \Ex[z]}^3}\approx 0$ and we cannot hope to estimate $\wStar$ from finite-number of samples.
        That said, in this case, it is sufficient to set $S=\R^d$ as $\cN(\R^d\triangle \Sstar; \muStar,\SigmaStar)\approx 0$.
        The problem is that we need to identify when this case occurs, i.e., when $\cN(\Sstar,\muStar,\SigmaStar)\approx 1$.
        (Note that, since $\alpha$ is only a lower bound on $\cN(\Sstar,\muStar,\SigmaStar)$, knowing $\alpha$ is insufficient to identify this case.)
        For this, we use the following result.
        \begin{restatable}[]{lemma}{halfspaceLearnergoodEvent}\label{lem:halfspaceLearner:goodEvent}
            Consider the following event 
            \[
                \max_{1\leq i\leq d}~\abs{
                    \Ex_{}\insquare{
                        \inparen{x - \gamma}^{\otimes 3}
                    }_{iii}
                }
                \leq d^{-3/2}\cdot \poly(\eps)\,.
            \]
            If this event holds, then $\cN(\Sstar;\muStar,\SigmaStar)\geq 1-\poly(\eps)$ and, otherwise, $\Ex_{}\insquare{\inparen{z - \Ex[z]}^3}\geq d^{-3/2}\cdot \poly(\eps)$.
        \end{restatable} 
        Now we are ready to state the algorithm for learning $\Sstar$ in \cref{thm:halfspaceLearner}. (See \cref{alg:halfspace}.)
        \begin{algorithm}[ht!]
            \caption{Efficient Algorithm to Learn Halfspace Survival Sets When $\cE(\cdot)$ Is Gaussian}\label{alg:halfspace}
            \begin{algorithmic}
            \Require $n=\poly(d/\eps)\cdot 
                                \poly\log\inparen{\sfrac{1}{\delta}}$ independent samples ${x_1,x_2,\dots,x_n}$ from $\cN(\muStar,\SigmaStar,\Sstar)$ 
            \vspace{2mm}
            \State 
                1.~~Compute the empirical mean of the provided samples $\wh{\gamma}\coloneqq \frac{2}{n} \sum_{i=1}^{n/2} x_i$
            \State 2.~~For each $1\leq j\leq d$, compute the third moment of the samples: 
                    $M_j\coloneqq\frac{2}{n} \sum_{i=1}^{n/2} {
                                \inparen{\inparen{x_i}_j - \wh{\gamma}_j}^{\otimes 3}
                            }$
            \State \textit{~~~~~$\#$ Guarantee: With probability $1-\delta$, $\sabs{M-\Ex{\inparen{x-\Ex\insquare{x}}^{\otimes 3}}}\leq \poly(\eps/d)$ for $x\sim \cN(\muStar,\SigmaStar,\Sstar)$}
            \vspace{4mm}
            \State 3.~~\textbf{if} {$\max_{1\leq j\leq d} \abs{M_j}\leq d^{-3/2}\cdot \poly(\eps)$} \textbf{then}
                \State \indent ~~\textit{$\#$ Guarantee: $\cN(\Sstar;\muStar,\SigmaStar)\geq 1-\poly(\eps)$ with probability $1-\delta$ (\cref{lem:halfspaceLearner:goodEvent})}
                \State 4.~\indent \Return  Membership oracle to $S=\R^d$ \textit{~~$\#$ which is $\eps$-close to $\Sstar$ as $\cN(\Sstar;\muStar,\SigmaStar)\geq 1-\poly(\eps)$}
            \State 5.~~\textbf{end if}
            \vspace{4mm}
            \State 6.~~Compute the unit vector $w\propto\inparen{M_1, M_2,\dots,M_d}^{1/3}$ \textit{~~~~~$\#$ which satisfies $\norm{w-w^\star}_\infty \leq \poly(\eps/d)$} %
            \State 7.~~Compute ${\tau}=\max_{\sfrac{n}{2}< i\leq n} x_i^\top w$ \textit{~~~$\#$ maximum value, s.t., $\inbrace{x\in \R^d\colon x^\top w\geq {\tau}}$ contains all samples}
            \State 8.~~\Return Membership oracle to $S=\inbrace{x\in \R^d\colon x^\top w\geq {\tau}}$ \textit{$\#$ which is $\eps$-close to $\Sstar$ (\cref{lem:halfspaceLearner:symDiff})}
            \end{algorithmic}
        \end{algorithm}
         
        \noindent \cref{lem:halfspaceLearner:goodEvent} shows that the set $S$ output in Step 4 satisfies the guarantee in \cref{thm:halfspaceLearner}. %
        The final result in this section proves the halfspace $S$ output in Step 8 also satisfies the guarantee in \cref{thm:halfspaceLearner}. 
        \begin{restatable}[]{lemma}{halfspaceLearnersymDiff}\label{lem:halfspaceLearner:symDiff}
            Suppose $\max_{1\leq i\leq d}~\sabs{
                    \Ex_{}{
                        \inparen{x - \gamma}^{\otimes 3}
                    }_{iii}
                }
            \geq d^{-3/2}\cdot \poly(\eps)$ and hence the above algorithm exits in Step 8.
            Conditioned on the event that $\norm{w-\wStar}_\infty \leq \poly(\eps/d)$, it holds that $\cN\inparen{S\triangle \Sstar; \muStar,\SigmaStar}\leq \eps$.
        \end{restatable} 
        We present the proofs of the lemmas in this section (i.e., \cref{lem:halfspaceLearner:moment,lem:halfspaceLearner:constantLB,lem:halfspaceLearner:goodEvent,lem:halfspaceLearner:symDiff}) in \cref{sec:module:halfspaceLearner}.

    \subsection{Application to Linear Regression with Gaussian Marginals}
            In this section, we study truncated linear regression with Gaussian marginals.
            As is usual in linear regression, the dependent variable $y\in \R$ and independent variable $x\in \R^d$ are related as follows
            \[
                y=x^\top w^\star+b^\star+\xi\,,
            \]
            where $\xi\sim \cN(0,1)$ is noise that is independent of $x$ and $w^\star\in \R^d$ and $b^\star\in \R$ are unknown parameters that one wants to estimate.
            In the truncated setting, however, we only observe samples $(x,y)$ which fall inside an \textit{unknown} survival set $S^\star \subseteq \R^{d+1}$.
            We assume that $S^\star$ has a constant mass under the joint distribution of $(x,y)$. %
            Finally, Gaussian marginal means that the independent variable $x$ is drawn from an unknown Gaussian distribution, say, $\cN(\mu,\Sigma).$
    
            Truncated linear regression is a fundamental problem at the intersection of Econometrics and truncated statistics.
            Indeed, it is the main focus of the seminal works of \citet{tobin1958estimation,cox1972regression}, and has seen applications in influential studies \cite{feige1972investigation,haussman1976truncation}; see the survey by \citet{maddala1983limited} for an overview.
            Beyond Econometrics, it also appears in Astronomy \cite{woodroofe1985astronomyTruncation} and Causal Inference \cite{imbens2015causal,hernan2023causal}.
    
            Nevertheless, despite its importance, there are no efficient algorithms known for this problem when the survival set is unknown.
            In particular, the techniques of 
            \citet{Kontonis2019EfficientTS} do not apply to this setting as they require a specific structure on the covariance matrix $\cov[(x,y)]$ that does not hold in this setting. 
            In the easier case of known survival set, following \citet{daskalakis2018efficient}, several works study truncated linear regression \cite{daskalakis2019computationally,ilyas2020theoretical,trunc_regression_unknown_var} and give polynomial time algorithms without requiring the distribution of independent variables to be Gaussian.
            However, even these works can only handle truncation on the dependent variable $y$ (or in the logit-space in the case of logit-regression \cite{ilyas2020theoretical}), i.e., they only handle truncation sets of the form $S^\star=\R^d\times T$ for some known set $T$.

            To see how to use our framework to find the parameters $(w^\star,b^\star)$, observe that (before truncation) $(x,y)$ are distributed according to the following Gaussian distribution 
            \[
                    (x,y)~\sim~ \cN\inparen{
                        \overline{\mu} \coloneqq 
                            \begin{bmatrix}
                                \mu\\ \mu^\top w^\star + b^\star
                            \end{bmatrix},~
                        \overline{\Sigma} \coloneqq 
                            \begin{bmatrix}
                                \Sigma & \Sigma w^\star\\
                                (w^\star)^\top \Sigma^\top & (w^\star)^\top \Sigma w^\star + 1 
                            \end{bmatrix}
                    }\,.
            \]
            Given {truncated} samples $(x,y)$, we can use the framework from \cref{infthm:gaussians}, to find estimates $\wh{\mu}$ and $\wh{\Sigma}$ such that 
            \[
                \norm{
                    \wh{\Sigma}^{-1}\wh{\mu} -
                    {\overline{\Sigma}}^{-1}\overline{\mu}
                }_2
                +
                \norm{
                    \wh{\Sigma}^{-1} -  {\overline{\Sigma}}^{-1}
                }_F \leq \eps \,.
                \yesnum\label{eq:linearRegression:guarantee}
            \]
            This along with the following fact allows us to recover $w^\star$ and $b^\star$ up to $\eps$ error in $L_2$ norm.
            \begin{fact}\label{fact:sigmaInverse}
                $\overline{\Sigma}^{-1}
                = \begin{bmatrix}
                        \Sigma^{-1} + w^\star (w^\star)^\top
                        & -w^\star \\
                        -(w^\star)^\top 
                        & 1\\
                    \end{bmatrix}$
                ~~and~~ $\overline{\Sigma}^{-1}\overline{\mu}
                = \begin{bmatrix}
                        \Sigma^{-1}\mu - b^\star w^\star\\
                        b^\star
                \end{bmatrix}
                $.
            \end{fact}

            \noindent These observations immediately give the following corollary. %
    
            \begin{corollary}[Truncated Linear Regression with Gaussian Marginals with Unknown Truncation]\label{thm:trunc_linear_reg}
                Fix any parameters $w^\star\in \R^d$ and $b^\star\in \R$ of the linear regression model.
                Let $\cD$ be the joint distribution of $(x,y)$ where $x\sim \cN(\mu,\Sigma)$ and $y=x^\top w^\star +b^\star + \xi$ where $\xi\sim \cN(0,1)$ is independent of $x$.
                Let $S^\star\in \hyS$ be an (unknown) set with mass $\Omega(1)$ under $\cD$.
                Let $\ell=\poly(1/\eps)\cdot \Gamma(S^\star)^2$.
                
                There is an algorithm that, given $N=\poly\inparen{
                        \inparen{\sfrac{(d+m)^\ell}{\eps}}
                        \cdot 
                        \log\inparen{\sfrac{1}{\delta}}
                    }$ independent samples from $\cD$ truncated to $S^\star$,
                in $\poly(N)$ time, outputs estimates $\wh{w}$ and $\wh{b}$ that with probability $1-\delta$, satisfy 
                \[ \norm{\wh{w} - w^\star}_2 \leq \eps
                \qquadand \sabs{\wh{b} - b^\star} \leq \eps\,. 
                \]
                If $S^\star$ is known to be a halfspace or an axis-aligned rectangle, then the algorithm only needs $N$ samples and runs in $\poly(N)$ time for $N=\poly\inparen{
                        \inparen{\sfrac{(d+m)}{\eps}}
                        \cdot 
                        \log\inparen{\sfrac{1}{\delta}}
                    }$.
            \end{corollary}

    \subsection{Sample Efficient Algorithms That Make A Single ERM Query} \label{sec:outline:query}
        \cref{infthm:exponential} shows that $d^{\poly(\ell/\eps)}$ samples and time is sufficient to learn $\thetaStar$ up to $\eps$-accuracy where $\ell\geq 1$ is the degree of the polynomials required to $\poly(\eps)$-approximate $\Sstar$ with respect to $\cE(\thetaStar)$ in the $L_2$-norm.
        The bottleneck in this method is learning the set $\Sstar$ via \lreg{}.
        As an alternative, one can learn $\Sstar$ using empirical risk minimization (ERM) with $\eps^{-\poly(1/\alpha)}\cdot \vc{}(\hyS)$ samples.
        This, in particular, improves the sample complexity of learning $\thetaStar$ from $d^{\poly(\ell/\eps)}$ to $\eps^{-\poly(1/\alpha)}\cdot (\vc{}(\hyS) + m)$.
        Moreover, in practical contexts where efficient ERM oracles are available, this also reduces the running time to $\eps^{-\poly(1/\alpha)}\cdot \vc{}(\hyS)$ plus the time required to run the ERM oracle.

        \begin{remark}
            With exponentially many calls to the ERM oracle, the sample complexity can further be improved from $(\vc{}(\hyS)+m)\cdot \eps^{-\poly(1/\alpha)}$ to $O((\vc{}(\hyS)+m)\eps^{-2})$.
            For details, we refer the reader to \cref{sec:sampleEfficiency}.
        \end{remark}

\section{Related Works}

    \paragraph{Learning from Truncated Data.} 
        There is a vast body of sample-efficient methods in statistics and econometrics for learning from truncated samples. 
        However as mentioned earlier, they are not computationally efficient; thus a recent and growing body of work in computer science designs computationally efficient algorithms \cite{zampetakis2022analyzing}.
        These works consider various settings of structured errors -- from truncation (e.g., 
        \cite{daskalakis2018efficient}), {censorship} (e.g., \cite{ons_switch_grad}), to self-selection (e.g.,
        \cite{cherapanamjeri2023selfselection}) -- and a variety of learning tasks -- from learning statistics \cite{daskalakis2018efficient,Fotakis2020EfficientPE,pmlr-v117-nagarajan20a,truncated_sm,lee2023learning,nagarajan2020truncatedMixtureEM,fotakis2020booleanProductTruncated,nagarajan2023EMMixtureTruncation,tai2023mixtureCensored,Kontonis2019EfficientTS} to linear (and logistic) regression \cite{daskalakis2019computationally,ilyas2020theoretical,trunc_regression_unknown_var} to learning linear dynamical systems \cite{ons_switch_grad}. 
        Among these,
        \citet{Kontonis2019EfficientTS} is the only work studying unknown truncation in a general setting. There has also been a recent growing body of work related to detecting truncation in data \cite{de2024detecting}.
    
    \paragraph{Robust Statistics.} 
        More broadly, learning with truncated data falls under robust statistics, where there have been many works studying estimation and learning in high dimensions. For instance, several works estimate the parameters of a multivariate Gaussian distribution in the presence of corruptions to a small $\eps$ fraction of the samples by an adversary -- which may add, delete, or modify samples \cite{lai_agnostic_est_16,robust_high_dim_dist_learning_16,diakonikolas_robust_gaussian_16,diakonikolas_robust_highdim_pratical_17}. 
        In our setting, we only allow deletions but a large fraction of samples can be deleted (e.g., 99\%) and, yet, we are able to get an estimation error that goes to zero as the number of samples goes to infinity; in contrast, this is not possible with adversarial deletions.
    
     \paragraph{Learning hypothesis from given positive and unlabeled samples.} 
         Our method, which involves learning the unknown set $S$ with the information available, also has connections to learning with positive and unlabeled samples: the positive labels are the samples belonging to the survival set and the unlabeled ones are generated from a crude initial approximation of the data generating distribution. \citet{natarajan1987learning} initiated work on learning with positive-only samples and characterized the classes that are learnable in this model. However, unfortunately, the characterization shows that even simple classes such as halfspaces in 2-dimensions are not learnable with positive-only samples. Then in 1998, \citet{denis1990positive} initializes the study on learning from positive and unlabeled samples, showing that any class that is PAC-learnable is also PAC-learnable from positive and unlabeled samples (in the realizable setting). To show this, under the mild assumption that $\cD(S^\star)$ is bounded away from $0$, they reduced the (realizable) PAC-learnable from positive and unlabeled samples to learning in the CPCN model with noise bounded away from $\sfrac{1}{2}$. 
         \citet{de2015satisfying} and 
         \citet{canonne2020learning} showed that the latter is necessary for at least efficient learning: that, even for simple distributions (such as Gaussians) and simple hypothesis classes (such as degree two PTFs), one cannot efficiently learn with both positive and unlabeled samples. These works however are under the harder setting of distribution learning (whereas ours is for learning sets). Hence, their hardness results do not imply hardness results in our setting.

\section{Assumptions on the Parameter Set \texorpdfstring{$\Theta$}{Θ}}\label{sec:assumptions}
        In this section, we describe our assumptions on the subset of $\Theta$ of the natural parameter space.
        We explicitly construct a set $\Theta$ that satisfies these assumptions for Gaussian and product exponential distributions (see \cref{sec:preprocess}).
        For other exponential families, we assume that a suitable set $\Theta$ is provided to us.
        Recall that, in addition to this, we also make certain assumptions on the exponential family and survival set $\Sstar$ (\cref{asmp:1:sufficientMass,asmp:1:polynomialStatistics}).
        \begin{assumption}[{Parameter Space}]\label{asmp:2}
            There is a {convex} subset $\Theta\subseteq\R^m$ of the natural parameter space and constants $\Lambda\geq \lambda > 0$ and $\eta>0$ such that the following hold.
            \begin{asmpenum}
                \item\label{asmp:cov} \textbf{(Boundedness of Covariance)} For each $\theta\in \Theta$,
                $\lambda I \preceq \cov_{\cE(\theta)}[t(x)]\preceq \Lambda I$. 
                \item\label{asmp:int} \textbf{(Interiority)}           %
                    $\theta^\star$ is in the $\eta$-relative-interior $\Theta(\eta)$ of $\Theta$. 
                \item\label{asmp:moment} \textbf{(Solver for Moment Matching)} 
                        There is an algorithm that given $\eps>0$ and $v\in \R^m$ such that $\snorm{v-\E_{\cE(\theta^\star, S^\star)}[t(x)]}_2\leq \eps$, outputs a vector $\theta$ in $\poly(m/\eps)$ time such that 
                        {$\snorm{\E_{\cE(\theta)}[t(x)] -v}_2\leq \eps$.}
                \item\label{asmp:proj} \textbf{(Projection Oracles)} 
                Fix a convex set $O\subseteq \Theta(\eta)$ containing $\theta^\star$.
                There are polynomial time projection oracles to $\Theta$ and $O$.
                \item\label{asmp:bridge} \textbf{($\chi^2$-Bridge)} 
                    For each $\theta_1,\theta_2\in \Theta$ there is a $\theta$, s.t., $\chidiv{\cE(\theta_1)}{\cE(\theta)}, 
                        \chidiv{\cE(\theta_2)}{\cE(\theta)}
                        \leq  
                        {
                            e^{\poly\inparen{\sfrac{\Lambda}{\lambda}}}
                        }$. %
            \end{asmpenum}
        \end{assumption}
        We have already discussed the first two properties in \cref{sec:overview:assumptions}.
        The third property enables us to satisfy the starting point requirement in the informal version of this assumption (i.e.,  \cref{infasmp:start}).
        This is possible as we show that if $\Ex_{\cE(\theta)}[t(x)]\approx \Ex_{\cE(\thetaStar,\Sstar)}[t(x)]$, then $\norm{\theta-\thetaStar}_2\leq \poly(1/\alpha)$ (\cref{thm:module:unlabeledSamples:momentMatching}).
        Access to projection oracles is required to run PSGD and find $\theta_0$: 
            namely, projection oracle to $O\subseteq \Theta(\eta)$ is required as we can prove stronger guarantees for parameters in the relative-interior of $\Theta$ (\cref{thm:informal:module:unlabeledSamples:measureGuarantees}) 
            and 
            projection oracle to $\Theta$ is needed to run the projection step in PSGD (\cref{lem:findingTheta0}).
        Finally, the existence of the $\chi^2$-Bridge is required to use the \lreg{} algorithm, as discussed in \cref{sec:intro:efficientLearning}.

        \begin{remark}[Examples of Solvers for Moment Matching]\label{rem:momentMatching:solver}
            Efficient solvers satisfying \cref{asmp:moment} are available for several exponential families.
            For instance, for the Gaussian and Exponential families, there is a straightforward closed-form expression for $\theta$ that can be computed in $O(m)$ time.
            Concretely, for the Gaussian family, given $v$, it suffices to select $\theta=\inparen{\Sigma^{-1}\mu, \Sigma^{-1}}$ where $\mu$ is the first $d$-coordinates of $v$ and $\Sigma=v_{d\colon d^2+d}-\mu\mu^\top$ where $v_{d\colon d^2+d}$ is the last $d^2$-coordinates of $v$.
            Moreover, even when a closed-form solution is not available, one can use a projected stochastic gradient descent method (PSGD) to approximately compute $\theta$ and we provide one such method in \cref{sec:psgd:theta0}.
        
        \end{remark}

\section{Computationally Efficient Algorithm for Exponential Family}\label{sec:computationalEfficiency}\label{sec:computationalEfficiency:exponential}     
    In this section, we prove \cref{infthm:exponential} whose formal statement is as follows.
    \begin{theorem}[{Efficient Estimation For General $S^\star$}]\label{thm:main}
        Let \cref{asmp:1:sufficientMass,asmp:1:polynomialStatistics,asmp:2} hold.
        Fix any $\eps,\delta\in (0,\nfrac12)$, $\ell$ such that degree-$\ell$ polynomials $\zeta$-approximate $\hyS$ with respect to $\cE$ in $L_2$-norm for $\zeta\leq \eps^{\poly\inparen{\frac{\Lambda k}{\lambda\eta\alpha}}}$, and 
        \[
                n = 
                    \poly\inparen{
                        {\frac{{(d+m)}^{\ell}}{\eps}}
                        \cdot 
                        \log{\frac{1}{\delta}}
                    }\,.
        \]
        There exists an algorithm that, given $n$ independent samples from $\cE(\theta^\star, S^\star)$ and
        constants $(\lambda,\Lambda,k,\alpha,\ell)$,
        in $\poly(n)$ time, outputs an estimate $\theta$, such that, with probability at least $1-\delta$, 
        \[
            \tv{\cE(\theta)}{\cE(\theta^\star)} \leq \eps\,.
        \]
    \end{theorem}
    Consider the special case of the Gaussian distribution.
    In this case, $\sfrac{(\Lambda k)}{(\lambda \eta\alpha)}=\poly(1/\alpha)$ (as mentioned in \cref{rem:preprocessing}) and, since, $\alpha$ is a constant, \cref{thm:GSAupperbound} implies that $\ell\leq \Gamma(\hyS)^2\cdot \poly(1/\eps)$.
    Substituting this in \cref{thm:main}, we get a $d^{\Gamma(\hyS)^2\cdot \poly\inparen{\sfrac{1}{\eps}}}$ time algorithm.
    This matches the running time of \citet{Kontonis2019EfficientTS}'s algorithm but is more general as it works for any Gaussian distribution and also for exponential family distributions that satisfy certain assumptions.
    Moreover, this dependence on $d$ and $\eps$ is necessary for any SQ algorithm due to the recent lower bound by \citet{diakonikolas2024statistical}.
    \begin{remark}[Bounds on $\ell$]
        \cref{thm:main} requires the family $\hyS$ containing the survival set $\Sstar$ to be approximable by polynomials of a finite degree $\ell$ and its running time increases with $\ell$.
        However, most explicit bounds on the degree of polynomials required to approximate set families have been computed when the underlying measure is Gaussian (see \cref{thm:GSAupperbound,tab:gsa:GaussianSurfaceArea} for examples of known bounds). 
        Nevertheless, for any log-concave exponential family and reasonably well-behaved set $S^\star$, a finite $\ell$ satisfies \cref{thm:main}'s requirement.\footnote{To formalize this, given $\rho>0$, define $\partial_\rho S^\star$ as the set of points strictly closer than $\rho$ from $S^\star$'s boundary. 
        Since $1_{\Sstar}$ is a $C^\infty$-function on $\R^d\backslash \partial_\rho S^\star$ and $\partial_\rho S^\star$ is an open set, Whitney's Extension Theorem \cite{whitney1934extension} guarantees that there is a $C^\infty$-function $f\colon \R^d\to\R$ that extends $1_{\Sstar}$ to $\R^d$. %
        Now, by the Stone-Weierstrass theorem \cite{pinkus20001weierstrass}, for any $R,\zeta>0$, there is a finite $\ell$ such that degree-$\ell$ polynomials $\zeta$-approximate $f$ in $L_\infty$-norm (and, hence, also in $L_2$-norm) with respect to $\cE(\thetaStar)$ within the ball $B(0,R)$ of radius $R$ and centered the origin. 
        Hence, degree-$\ell$ polynomials also $\zeta$-approximate $1_{\Sstar}$ on $\R^d\backslash B(0,R)\backslash \partial_\rho S^\star$.
        This is sufficient to run our algorithms if $\cE(\theta^\star)$ places less than $\poly(\eps)$ on $B(0, R)\cup \partial_\rho S^\star$; the mass on $B(0, R)$ can be bounded using tail bounds for log-concave distributions and the mass on $\partial_\rho S^\star$ goes to $0$ as $\zeta\to0$.}
        That said, this finite value of $\ell$ can be extremely large, and relaxing the requirements in \cref{thm:main} is an interesting open problem.
    \end{remark}
    The algorithm in \cref{thm:main} is summarized as \cref{alg:main}.  
    In the remainder of this section, we prove \cref{thm:main} following the outline in \cref{sec:overview}.
        \begin{algorithm}[ht!]
        \caption{Estimation From Samples Truncated to Unknown Survival Set (\cref{thm:main})}\label{alg:main}
        \begin{algorithmic}
        \Require 
            Constants $\alpha,\eps,\delta,\lambda,\Lambda$ and
            independent samples $x_1, x_2, \dots, x_n\in \R^d$ from $\cE(\theta^\star, S^\star)$
        \Ensure \cref{asmp:1:sufficientMass,asmp:1:polynomialStatistics,asmp:2} hold with constants $\alpha,k,\eta,\lambda,\Lambda$ and domain $\Theta\subseteq\R^m$
        \Oracle Exact sampling oracle for $\cE(\theta)$ and oracles promised in \cref{asmp:moment,asmp:proj}
        \vspace{2mm}
        \State\textbf{(Subroutine A) Find a parameter $\theta_0$ at constant distance to $\thetaStar$} 
        \begin{enumerate}
            \item Calculate $\overline{t} \leftarrow \frac{3}{n}\sum_{i=1}^{\nfrac{n}{3}} t(x_i)$ 
            \item Find  $\wh{\theta}$ such that $\snorm{\E_{\cE(\wh{\theta})}\insquare{t(x)}-\overline{t}}_2 \leq \eps$ using the oracle in \cref{asmp:moment} %
            \item Define $\theta_0$ to be the projection of $\smash{\wh{\theta}}$ to ${\Theta(\eta)}$ \hfill 
            \textit{$\#$ $\Theta(\eta)$ is the $\eta$-relative interior of $\Theta$}
        \end{enumerate}
        \State \hspace{2.5mm} \textit{$\#$~~~Guarantee: With probability $1-\delta$, $\theta_0$ is in $\Theta(\eta)$ and at Euclidean distance $O(\Lambda/(\eta\alpha\lambda))$ from $\thetaStar$}
        \State \hspace{2.5mm} \textit{$\#$~~~Run time: $O(n)+\poly(dm/\eps)$}
        \vspace{4mm}
        \State\textbf{(Subroutine B) Find a set $S$ such that $\cE(S\triangle S^\star;\thetaStar)\leq \poly(\alpha\eps)$}
        \begin{enumerate}
            \item[4.] Obtain a membership oracle $\cM_S$ from  \cref{alg:pu}~($\theta_0,\eps, \delta, x_{\inparen{\nfrac{n}{3}}+1:\inparen{\nfrac{2n}{3}}}$) 
        \end{enumerate}
        \State \hspace{2.5mm} \textit{$\#$~~~~Guarantee: $\cM_S$ is a membership oracle for $S$ that, with probability $1-\delta$, satisfies $\cE(S\triangle S^\star;\thetaStar)\leq \poly(\alpha\eps)$}
        \State \hspace{2.5mm} \textit{$\#$~~~Run time: $\poly(n)$}
        \vspace{4mm}
        \State\textbf{(Subroutine C) Find $\theta\approx\theta^\star$}
        \begin{enumerate}[]
            \item[5.] Find $\theta$ by solving the \pmle{} problem: \cref{alg:psgd} $(\theta_0,\cM_S, \alpha, \eta, \lambda, \Lambda, x_{\inparen{\nfrac{2n}{3}}+1:n})$
        \end{enumerate}
        \State \hspace{2.5mm} \textit{$\#$~~~Guarantee: With probability $1-\delta$, $\theta$ is at distance at most $\eps$ from $\thetaStar$}
        \State \hspace{2.5mm} \textit{$\#$~~~Run time: $\poly(dm/\eps)\cdot M_S$ where $M_S$ running time of $\cM_S$}
        \vspace{2mm}
        \State \Return estimate $\theta$
        \end{algorithmic}
        \end{algorithm}

    \subsection{Finding \texorpdfstring{$\theta^\star$}{θ*} with an Approximation of \texorpdfstring{$S^\star$}{S*} via \pmle{}}\label{sec:module:reductionKnownTruncation}
        As discussed in \cref{sec:overview}, our approach to learning $\theta^\star$ steps away from prior works by considering an optimization problem whose objective is \textit{not} the maximum likelihood: 
        we solve the $S$-\pmle{} problem which is defined as follows (see \cref{def:pmle}):
        \[
            \min_{\theta\in \Omega}\quad  
                {
                    \negLL_S\sinparen{\theta}
                    \coloneqq
                    -\Ex_{x\sim \cE(\theta^\star,S^\star\cap {S})}{\log{\cE(x; \theta, S)}}
                }\,.
        \]
        Where $\Omega$ is a set over which $\negLL_S(\cdot)$ is strongly convex.
    This section proves the following result. %
    \begin{restatable}[{Learning $\thetaStar$ Given $S\approx \Sstar$}]{theorem}{thmModuleReduction}\label{thm:module:reduction}
                Fix any $\theta^\star$ and $S^\star$ such that \cref{asmp:1:sufficientMass,asmp:1:polynomialStatistics,asmp:cov,asmp:int,asmp:moment,asmp:proj} hold.
                Fix any $\eps,\delta\in (0,1/2)$. 
                There is an algorithm that, given the following inputs:
                \begin{itemize}
                    \item query access to a membership oracle to $S$ satisfying $\cE(S\triangle S^\star; \theta^\star)\leq  \alpha^3\eps^2 $ with running time $M_S$;
                    \item $n=\poly(dm/\eps)\cdot {\exp\sinparen{\wt{O}\inparen{\sfrac{k\Lambda^2}{(\alpha\eta\lambda)^2}}}}$ independent samples from $\cE(\theta^\star, S^\star)$; 
                    \item constants $(\lambda, \Lambda, \alpha)$; %
                \end{itemize}
                outputs a vector $\theta$, in time $\poly(n)\times M_S$, such that, with probability at least $1-\delta$, $\tv{\cE(\theta)}{\cE(\theta^\star)} \leq \eps.$
    \end{restatable}
    The proof of \cref{thm:module:reduction} is organized as follows: \cref{sec:reduction_to_known_truncation:initialize_psgd} constructs the domain $\Omega$,
    \cref{sec:reduction_to_known_truncation:optimum_of_pmle} proves that $\theta^\star$ is close to the \textit{constrained} minimizer of $\negLL_S(\cdot)$ over $\Omega$, and 
    \cref{sec:reduction_to_known_truncation:solving_pmle_psgd} gives the projected SGD algorithm to minimize $\negLL_S(\cdot)$ over $\Omega.$

    \subsubsection{Domain \texorpdfstring{$\Omega$}{Ω} where \texorpdfstring{$\negLL_S(\cdot)$}{L\_S} is Strongly Convex}\label{sec:reduction_to_known_truncation:initialize_psgd}
        From \cref{eq:intro:strongConvexity}, recall that for each $\theta$, 
        \[
            \nabla^2 \negLL_S(\theta)  
            \succeq
            \Omega\inparen{\inparen{\frac{\cE(S; \theta)}{k}}^{k}}\cdot \lambda \cdot  I\,.
            \yesnum\label{eq:proof:strongConvexity}
        \]
        Hence, to show that $\negLL_S(\cdot)$ is strongly convex over $\Omega$, it suffices to lower bound $\cE(S;\theta)$ over $\Omega$.
        To lower bound $\cE(S;\theta)$, we first find  a parameter $\theta_0$ within $O(1/(\eta\lambda\alpha))$ distance to $\thetaStar$.
        \begin{lemma}[{A Parameter Close to $\thetaStar$}]
                \label{thm:module:unlabeledSamples:momentMatching}
                \label{lem:findingTheta0Formal}
                Suppose \cref{asmp:1:sufficientMass,asmp:1:polynomialStatistics,asmp:cov,asmp:int,asmp:moment,asmp:proj} hold.
                Fix any $\delta\in (0,1)$.
                Given $n=O\sinparen{
                    (\sfrac{m\Lambda^2}{\eta^4})
                    \log^2\inparen{\sfrac{1}{\alpha}}
                    \log^2\inparen{\sfrac{1}{\delta}}
                }$ independent samples $x_1,x_2,\dots,x_n$ from $\cE(\theta^\star, S^\star)$, define 
                \begin{enumerate}
                    \item $\wh{\theta}\in \R^m$ as any vector satisfying $\norm{\Ex_{\cE(\wh{\theta})}[t(x)] - \inparen{\sfrac{1}{n}} \sum_{i=1}^n t(x_i)}_2\leq \eta
                    $; and 
                    \item $\theta_0\in \R^m$ as the projection of $\wh{\theta}$ to the convex set $ O \subseteq \Theta(\eta)$ defined in \cref{asmp:proj}. 
                \end{enumerate}
                With probability $1-\delta$, $\wh{\theta}$ exists and $\norm{\theta_0-\theta^\star}_2\leq \sfrac{3\Lambda}{(\eta\lambda\alpha)}$. 
                {Moreover, there is an algorithm that, given $x_1,\dots,x_n$, outputs $\theta_0$ in $\poly(n)$ time.}
            \end{lemma}
        Now, if $\Omega$ is a neighborhood of $\theta_0$, then $\min_{\theta\in \Omega}\cE(S;\theta)$ can be lower bounded using the next result. 
        
        \begin{restatable}[]{lemma}{measureBounds}\label{thm:module:unlabeledSamples:measureGuarantees}
                Fix any $r, R>0$.
                Define $C=\max\inbrace{\Lambda R(R+r), 1+(R/r)}$. %
                For any $\theta_1,\theta_2\in \Theta(r)$ such that $\norm{\theta_1-\theta_2}_2\leq R$, it holds that for all sets $T\subseteq\R^d$
                \[
                    e^{-C}\cdot {\cE(T; \theta_2)^{C}}
                    \leq \cE(T; \theta_1)
                    \leq 
                    e^C\cdot {\cE(T; \theta_2)^{\sfrac{1}{C}}}\,.
                \]
            \end{restatable}
        Combining \cref{lem:findingTheta0Formal,thm:module:unlabeledSamples:measureGuarantees} implies the following result.
        \begin{lemma}[Construction of $\Omega$ and Strong Convexity]\label{lem:strongConvexityOnOmega}\label{prop:necessary_cond_s_tilde:mass_lower_bound}
            Suppose \cref{asmp:1:sufficientMass,asmp:1:polynomialStatistics,asmp:cov,asmp:int,asmp:moment,asmp:proj} hold.
            Fix the same $\delta,n,\wh{\theta},\theta_0$ as \cref{lem:findingTheta0Formal}.
            Let $B(z,r)$ be the $L_2$-ball of radius $r$ centered at $z$.
            Define $\Omega\coloneqq 
                B\inparen{\theta_0, 
                    \sfrac{3\Lambda}{(\eta\lambda\alpha)}
                }\cap \Theta.$
            Conditioned on $\wh{\theta}$ existing and $\norm{\theta_0-\theta^\star}_2\leq \sfrac{3\Lambda}{(\eta\lambda\alpha)}$, 
            \[
                \text{for all $\theta\in \Omega$ and $T\subseteq\R^d$}\,,
                \qquad 
                \cE(T; \theta)
                \geq 
                \cE(T; \theta^\star)
                ^{{
                    {\frac{16\Lambda^2}{(\alpha\eta\lambda)^2}}
                }}\,.
            \]
            Hence, by \cref{eq:proof:strongConvexity}, for all $T\subseteq\R^d$, $\negLL_T(\cdot)$ is $\lambda \cdot \Omega\inparen{
                \inparen{\cE(T; \theta^\star)/k}
                ^{{
                    {\frac{16k\Lambda^2}{(\alpha\eta\lambda)^2}}
                }}}$ strongly convex on $\Omega$.
        \end{lemma}
        \begin{proof}
            This follows from \cref{thm:module:unlabeledSamples:measureGuarantees} by using that $\norm{\theta_0-\theta^\star}_2\leq \sfrac{3\Lambda}{(\eta\lambda\alpha)}$ and $\theta_0, \thetaStar\in \Theta(\eta)$.
        \end{proof}
        In particular, we get the following corollary.
        \begin{corollary}\label{coro:strongConvexity}
            Consider the setup in \cref{lem:strongConvexityOnOmega}.
            Consider any set $S\subseteq\R^d$ satisfying $\cE\inparen{S\triangle S^\star; \theta^\star}
                    \leq \eps\cdot \cE\inparen{S^\star; \theta^\star}$.
            With probability at least $1-\delta$, $\negLL_S(\cdot)$ is $\Omega\inparen{e^{
                    {-\frac{17k\Lambda^2}{(\alpha\eta\lambda)^2}\cdot \log{\frac{k}{\alpha}}}
                }}$ strongly convex on $\Omega$.
        \end{corollary}
        In the remainder of this section, we prove  \cref{lem:findingTheta0Formal,thm:module:unlabeledSamples:measureGuarantees}.

        \subsubsection*{Proof of \cref{lem:findingTheta0Formal}}
        We begin by making a few observations.
            \begin{description}
                \item[\textbf{Computing $\wh{\theta}$.}] 
                    Using standard analysis, $\sum_{i=1}^n t(x_i)/n$ concentrates around its mean $\Ex_{\cE(\theta^\star, S^\star)}[t(x)]$.
                    In particular, the two are within a distance $\eta$ in the $L_2$-norm with probability $1-\delta$ (see \cref{prop:perturbed_mle_alt_proof:trunc_concentration} for a proof).
                    This along with  \cref{asmp:moment} ensures the existence of $\wh{\theta}$ with probability $1-\delta.$
                    \cref{asmp:moment} also guarantees a $\poly(m/\eps)$ time method to compute $\wh{\theta}$.
                    In \cref{rem:momentMatching:solver}, we discuss algorithms for computing $\wh{\theta}$ for specific distribution families. 
                \item[\textbf{Computing ${\theta_0}$.}]
                    $\theta_0$ is $\wh{\theta}$'s projection to $O$, and is computable in $\poly(m/\eps)$ time by \cref{asmp:proj}. %
            \end{description}
            \mbox{To upper bound $\norm{\theta_0-\theta^\star}$ observe that, as $O$ is convex and contains $\thetaStar$ and $\theta_0$ is $\wh{\theta}$'s projection to $O$}
            \[
                \norm{\theta_0-\theta^\star}\leq 
                \snorm{\wh{\theta}-\theta^\star}\,.
            \]
            To upper bound $\snorm{\wh{\theta}-\theta^\star}$, define the untruncated maximum likelihood function 
            \[
                {\negLL}^{\text{untr}}(\theta)= \E_{\cE(\theta^\star)}\insquare{\log \cE(x; \theta)}\,.
            \]
            Its gradient and hessian have the following expressions
            \[ 
                \nabla {\negLL}^{\text{untr}}(\theta) = \E_{\cE(\theta)}[t(x)]-\E_{\cE(\theta^\star)}[t(x)]
                \qquadand
                \nabla^2 {\negLL}^{\text{untr}}(\theta) = \cov_{\cE(\theta)}[t(x)]\,. 
                \yesnum\label{eq:mainProof:distance:eq1}
            \]
            It follows that $\nabla^2 {\negLL}^{\text{untr}}(\theta)$ is positive semi-definite and, hence, ${\negLL}^{\text{untr}}(\cdot)$ is convex.
            Since ${\negLL}^{\text{untr}}(\cdot)$ is convex and, as is easy to check, its gradient is 0 at $\thetaStar$, $\thetaStar$ is a minimizer of ${\negLL}^{\text{untr}}(\cdot)$.
            In fact, it is the unique minimizer as \cref{asmp:cov} guarantees that ${\negLL}^{\text{untr}}(\cdot)$ is $\lambda$-strongly convex on $\Theta$ (which contains $\thetaStar$).
            Another consequence of strong convexity is that to upper bound $\snorm{\wh{\theta}-\theta^\star}_2$, it suffices to upper bound $\snorm{\nabla {\negLL}^{\text{untr}}(\wh{\theta}) }_2$.
            Toward upper bounding $\snorm{\nabla {\negLL}^{\text{untr}}(\wh{\theta}) }_2$, observe that 
            \begin{align*}
                \norm{\nabla {\negLL}^{\text{untr}}(\wh{\theta}) }_2 \quad\Stackrel{\eqref{eq:mainProof:distance:eq1}}{=}\quad
                \norm{\E_{\cE(\wh{\theta})}[t(x)] - \E_{\cE(\theta^\star)}[t(x)]}_2 
                \qquad \Stackrel{{\rm Definition~of~}\wh{\theta}}{\leq}\qquad \norm{\frac{1}{n}\sum_{i \in [n]} t(x_i) - \E_{\cE(\theta^\star)}[t(x)] }_2 + \eta\,. 
            \end{align*}
            Where $x_1,x_2,\dots,x_n$ are independent samples from $\cE(\theta^\star, S^\star)$, which were used to define $\wh{\theta}$ in \cref{lem:findingTheta0Formal}.
            To upper bound the RHS we use the following result, which is proved in \cref{proofof:changeInMomentSuffTrunc}.
            \begin{restatable}[{Upper Bound on Change in Moments Due to Truncation}]{lemma}{changeInMomentSuffTrunc}\label{prop:perturbed_mle_alt_proof:trunc_nontru_mean_suff}
                Suppose that \cref{asmp:1:sufficientMass,asmp:1:polynomialStatistics,asmp:cov,asmp:int} hold. %
                Fix any constants $\zeta,\delta\in (0,1)$ and set $T\subseteq\R^d$ with $\cE(T;\theta^\star)\geq\alpha$.
                Fix $n = \wt{\Omega}\sinparen{\inparen{m\Lambda^2/\inparen{\zeta^2\eta^2}} \log^2(1/\alpha)\log^2(1/\delta)}$.
                Given independent samples ${x_1,x_2,\dots,x_n}$ from $\cE(\theta^\star, T)$, with probability at least $1 - \delta$, 
                \[ 
                    \norm{ \frac{1}{n}\sum_{i \in [n]} t(x_i) - \E_{z \sim \cE(\theta^\star)}[t(z)] }_2 
                    \leq 
                    \max\inbrace{
                        \eta \Lambda\,,
                        \frac{2}{\eta}\cdot \log{\frac{1}{\cE(T; \theta^\star)}} + \zeta
                    }
                    \,.
                \]
            \end{restatable} 
            \noindent In particular, for $\zeta=\eta$,  
            \[
                \norm{\nabla {\negLL}^{\text{untr}}(\wh{\theta}) }_2 
                \leq 
                        \max\inbrace{
                            \eta \Lambda\,,
                            \frac{2}{\eta}\cdot \log{\frac{1}{\cE(T; \theta^\star)}} + \eta
                        }
                        \,.
            \]
        Finally, the $\lambda$-strong convexity of ${\negLL}^{\text{untr}}(\cdot)$, implies that %
        \[
            \snorm{\wh{\theta} - \theta^\star}_2 
            \leq 
                \max\inbrace{
                    \frac{\eta \Lambda}{\lambda}\,,
                    \frac{2}{\eta\lambda}\cdot \log{\frac{1}{\cE(S^\star; \theta^\star)}}
                    + \frac{\eta}{\lambda}
                }
            \quad \Stackrel{0<\eta,\alpha\leq 1}{\leq}\quad 
                \frac{3\Lambda}{\eta\lambda\alpha}
            \,. %
        \]

        \noindent We close this section by noting that $\wh{\theta}$ can be computed using the solvers in \cref{asmp:moment}, which  are available for several exponential families as explained in \cref{rem:momentMatching:solver}. %

        \subsubsection*{Proof of \cref{thm:module:unlabeledSamples:measureGuarantees}}
            
            \begin{proof}[Proof]
                It suffices to consider $\theta_1\in \Theta(\eta)$ as the other case follows by symmetry.
                Fix any set $T$.
                Consider $u,v\geq 1$ satisfying $(\sfrac{1}{u})+(\sfrac{1}{v})=1$ to be fixed later.
                Observe that 
                \[
                    \cE(T; \theta_1)
                    = 
                    \Ex_{x\sim \cE(\theta_2)}\insquare{
                        \frac{
                            \cE(x; \theta_1)
                        }{
                            \cE(x; \theta_2)
                        }
                        \cdot \mathds{1}_S(x)
                    }
                    \quad\qquad
                        \Stackrel{\text{Hölder's inequality}}{\leq}
                    \qquad\quad
                    \Ex_{x\sim \cE(\theta_2)}\insquare{
                         \mathds{1}_S(x)
                    }^{1/u}
                    \cdot 
                    \Ex_{x\sim \cE(\theta_2)}\insquare{
                        \inparen{
                            \frac{
                                \cE(x; \theta_1)
                            }{
                                \cE(x; \theta_2)
                            }
                        }^v
                    }^{1/v}
                    \,.
                \]
                The first term is exactly $\cE(T; \theta_2)$.
                Toward upper bounding the second term, observe 
                \begin{align*}
                    \Ex_{x\sim \cE(\theta_2)}\insquare{
                        \inparen{
                            \frac{
                                \cE(x; \theta_1)
                            }{
                                \cE(x; \theta_2)
                            }
                        }^v
                    }
                    &=
                    \int  e^{v(\theta_1-\theta_2)^\top t(x) + (v-1)A(\theta_2)-vA(\theta_1)} \cdot \cE(x; \theta_2)\d x\\
                    &=
                    \int  e^{
                            (v-1)A(\theta_2)+ A\inparen{\theta_1+(v-1)(\theta_1 - \theta_2)} - vA(\theta_1) 
                        }
                        \cdot \cE(x; \theta_2)\d x
                     \,.
                \end{align*}
                Fix 
                \[
                    v=1 + \frac{r}{R}
                    \quadand
                    u = 1 +\frac{R}{r}\,.
                \]
                Since $\theta_1$ is in the $r$-relative interior of $\Theta$ and $\theta_1-\theta_2$ lies in the affine subspace of $\Theta$ with $\norm{\theta_1-\theta_2}\leq R$, for $v=1+(\sfrac{r}{R})$, $\theta_1+(v-1)(\theta_1 - \theta_2)\in \Theta$.
                Since $A(\cdot)$ is $\Lambda$-smooth over $\Theta$, it follows that 
                \[
                    (v-1)A(\theta_2)+ A\inparen{\theta_1+(v-1)(\theta_1 - \theta_2)} - vA(\theta_1) \leq \Lambda \norm{\theta_1-\theta_2}_2^2
                    + \Lambda(v-1)^2\norm{\theta_1-\theta_2}_2^2\,.
                \]
                Using that $\norm{\theta_1-\theta_2}\leq R$, it follows that $\Ex_{x\sim \cE(\theta_2)}\insquare{
                        \inparen{
                            \sfrac{
                                \cE(x; \theta_1)
                            }{
                                \cE(x; \theta_2)
                            }
                        }^v
                    }
                    \leq 
                    \exp\inparen{\Lambda R^2\inparen{1 + (\sfrac{r^2}{R^2})}}$ and, hence, 
                \[
                    \cE(T;\theta_1)
                    \leq 
                    \cE(T;\theta_2)^{\frac{1}{1+(R/r)}}
                    \cdot 
                    \exp\inparen{\Lambda\cdot \frac{R^2 + r^2}{1+(\sfrac{r}{R})} }
                    \leq 
                    \cE(T;\theta_2)^{\frac{1}{1+(R/r)}}
                    \cdot 
                    \exp\inparen{\Lambda R(R+r)}\,.
                \]
            \end{proof}

    \subsubsection{\texorpdfstring{$\theta^\star$}{θ*} Is Close to the Minimizer of \texorpdfstring{$\negLL_S(\cdot)$}{L\_S} Over \texorpdfstring{$\Omega$}{Ω}}  
    \label{sec:reduction_to_known_truncation:optimum_of_pmle}
        In this section, we prove the following result.
    \begin{lemma}[{$\theta^\star$ Is Close to the Solution of Perturbed MLE}]\label{prop:distOfNoisyMinimizer}
        Suppose that \cref{asmp:1:sufficientMass,asmp:1:polynomialStatistics,asmp:cov,asmp:int} hold. %
        For any $\eps,\delta \in (0, \sfrac{1}{2})$, 
        any set ${S}$ satisfying $\cE\sinparen{S\triangle S^\star; \theta^\star}\leq  \alpha^3\eps^2 $, 
        the set $\Omega$ defined in \cref{lem:findingTheta0Formal}, and 
        any $\eps$-minimizer $\wt{\theta}$ of the $S$-Perturbed MLE problem over $\Omega$, it holds that 
        \[
            \snorm{\wt{\theta}-\theta^\star}_2 
            \leq 
            \eps\cdot \frac{\Lambda}{\eta}
            \cdot 
            O\inparen{
                e^{
                    \frac{17k\Lambda^2}{(\alpha\eta\lambda)^2}\cdot \log{\frac{k}{\alpha}}
                }
            }
            \,. 
        \] 
    \end{lemma}
    \begin{proof}
        Let $\thetapmle$ be the minimizer of $S$-Perturbed MLE over $\Omega$.
        The triangle inequality implies 
        \[
            \snorm{\wt{\theta}-\theta^\star}_2 
            \leq 
            \snorm{\wt{\theta}-\thetapmle}_2 
            + 
            \snorm{\thetapmle-\theta^\star}_2 \,.
        \]
        Since $\wt{\theta}$ is an $\eps$-minimizer of $S$-Perturbed MLE problem over $\Omega$, it holds that $\snorm{\wt{\theta} - \thetapmle}_2\leq \eps$.
        It remains to upper bound the second term; namely, $\snorm{\thetapmle-\theta^\star}_2$.
        \cref{coro:strongConvexity} shows that with probability $1-\delta$, $\negLL{}_S(\cdot)$ is $\sigma$-strongly-convex over $\Omega$
        for 
        \[
            \sigma \coloneqq 
            \Omega\inparen{
                e^{
                    -\frac{17k\Lambda^2}{(\alpha\eta\lambda)^2}\cdot \log{\frac{k}{\alpha}}
                }
            }\,.
        \]
        Standard properties of strongly convex function implies 
        \[
            \inangle{
                \nabla \negLL{}_S(\thetaStar) - \nabla \negLL{}_S(\thetapmle),
                \thetaStar - \thetapmle
            }
            \geq \sigma\cdot \snorm{\thetaStar - \thetapmle}_2^2\,.
            \yesnum\label{eq:proof:usingStrongConvexity1}
        \]
        Since $\thetapmle$ is the minimizer of $\negLL{}_S(\cdot)$ over a convex set $\Omega$ that contains $\thetaStar$, it follows that 
        \[
            \inangle{
                \nabla \negLL{}_S(\thetapmle),
                \thetaStar - \thetapmle
            }
            \geq 
            0\,.
            \yesnum\label{eq:proof:usingStrongConvexity2}
        \]
        As otherwise, moving by an infinitesimal amount in the direction $\thetaStar - \thetapmle$ from $\thetapmle$ would reduce the value of $\negLL{}_S(\cdot)$ and violate the optimality of $\thetapmle.$
        Therefore, we get that 
        \[
            \snorm{\thetaStar - \thetapmle}_2^2
            \quad \Stackrel{\eqref{eq:proof:usingStrongConvexity1},\eqref{eq:proof:usingStrongConvexity2}}{\leq}\quad  
            \frac{1}{\sigma}\cdot \inangle{
                \nabla \negLL{}_S(\thetaStar),
                \thetaStar - \thetapmle
            }
            \leq 
            \frac{1}{\sigma}\cdot 
                \norm{\nabla \negLL{}_S(\thetaStar)}_2\cdot 
                \snorm{\thetaStar - \thetapmle}\,.
        \]
        To upper bound $\snorm{\thetaStar-\thetapmle}_2$, it remains to bound $\norm{\nabla \negLL_S(\thetaStar)}_2$, which can be done as follows
        \begin{align*}
            \norm{\nabla_\theta \negLL_S({\theta^\star})}_2
            \quad\Stackrel{\rm \cref{fact:gradOfNoisyMLE}}{=}\quad
            \norm{\Ex_{\cE(\theta^\star,S^\star\cap S)}[t(x)] - \Ex_{\cE(\theta^\star,S)}[t(x)]}
            ~~\qquad\qquad \Stackrel{{\rm\cref{prop:intro:cor10cor11}}\,,~\cE(S\triangle \Sstar;\thetaStar)\leq \alpha^3\eps^2 }{\leq} ~~\qquad\qquad 
            \frac{24\Lambda e^{2\Lambda\eta^2}}{\eta} \cdot \eps\,. 
        \end{align*}
    \end{proof}

    \subsubsection{Solving Perturbed MLE via PSGD over \texorpdfstring{$\Omega$}{Ω}}\label{sec:reduction_to_known_truncation:solving_pmle_psgd}

    To solve the \pmle{} problem over the domain $\Omega$, we use projected stochastic gradient descent (PSGD). %
    There are various analyses of PSGD, we use the following. %
    
    \begin{restatable}[Theorem 5.7 of \citet{garrigos2023handbook})]{theorem}{thmSGD}\label{thm:sgd}
        Let $f$ be a $\sigma$-strongly convex function. 
        Given a convex set $K$, let $\bar{\theta} \in \arg\min_{\theta \in K} f(\theta)$. 
        Consider the sequence $\{\theta_t \}_{t=1}^N$ generated by PSGD given the sequence of random vectors $\{v_t \}_{t=1}^N$  satisfying $\E[v_t \mid \theta_t] = \nabla f(\theta_t)$ and $\E[\|v_t\|^2 \mid \theta_t] < \rho^2$ for all $t$, with a constant step size satisfying $0<\gamma< \sfrac{1}{\sigma}$ and projection set $K$.
        For $t \geq 1$, 
        \[ 
            \E\|\theta_t - \bar{\theta}\|^2 \leq (1-2\gamma \sigma )^t \cdot \snorm{\theta_0 - \bar{\theta}}^2 + \frac{\gamma\rho^2}{\sigma}\,. 
        \]
    \end{restatable}
    We prove the following result in this section.
    \begin{lemma}\label{lem:psgd_complexity_result}
        Suppose \cref{asmp:1:sufficientMass}, \ref{asmp:1:polynomialStatistics}, \ref{asmp:cov}, \ref{asmp:int}, \ref{asmp:moment}, and \ref{asmp:proj} hold. 
        There is an algorithm which, given 
        \begin{itemize}[itemsep=2pt]
            \item an initial point $\theta_0$ defined in \cref{lem:findingTheta0Formal};
            \item independent samples from $\cE(\theta^\star, S^\star)$;
            \item access to a membership oracle, with run time $M_S$, for a set ${S}$ satisfying $\cE\inparen{S\triangle S^\star; \theta^\star}
                    \leq \eps\cdot \cE\inparen{S^\star; \theta^\star}$
        \end{itemize}
        outputs $\theta$ that, with probability $1-\delta$, is $\eps$-close to the optimizer of $S$-Perturbed MLE over $\Omega$ (defined in \cref{lem:findingTheta0Formal}).
        It runs in $\poly\inparen{\frac{md}{\eps},\log\frac{1}{\delta}, M_S, 
            {e^{\wt{O}\inparen{\frac{k\Lambda^2}{(\alpha\eta\lambda)^2}}}}}$ time and uses $\tilde{O}\inparen{{\frac{m}{\eps^{2}}}\log^2{\frac{1}{\delta}}}
            \cdot 
            {e^{\wt{O}\inparen{\frac{k\Lambda^2}{(\alpha\eta\lambda)^2}}}}$ samples.
    \end{lemma}
    \noindent The algorithm in \cref{lem:psgd_complexity_result} is a projected stochastic gradient descent (see \cref{alg:psgd}).
    \begin{algorithm}[ht!]
    \caption{Projected Stochastic Gradient Descent (algorithm in \cref{lem:psgd_complexity_result}}
        \label{alg:psgd}
    \begin{algorithmic}
        \Require Initial point $\theta_0 \in \Theta$, membership oracle $\cM_S(\cdot)$ to set $S$, constants $\alpha,k,\eta,\lambda,\Lambda$, and \\
        \hspace{8mm} independent samples $x_1, x_2, \dots \in \R^d$ from $\cE(\theta^\star, S^\star)$
        \Ensure \cref{asmp:1:sufficientMass,asmp:1:polynomialStatistics,asmp:2} hold and $S$ satisfies $\cE(S \triangle S^\star; \theta^\star) \leq \alpha \eps$
        \Oracle Sampling oracle for the family $\cE(\cdot)$ and projection oracle promised in \cref{asmp:proj}
        \vspace{4mm}
        \State Construct parameter space $\Omega \leftarrow B(\theta_0, 3\Lambda/(\eta \lambda \alpha)) \cap \Theta$
        \State \textit{$\#$~~With high probability, $\Omega$ contains $\theta^\star$ and $\negLL_S(\cdot)$ is strongly convex over it (\cref{lem:findingTheta0Formal,lem:strongConvexityOnOmega})} 
        \vspace{2mm}
        \State Fix the step size $\gamma \leftarrow e^{-\wt{O}\inparen{\sfrac{k\Lambda^2}{(\alpha\eta\lambda)^2}}}$ and number of iterations $N \leftarrow \tilde{O}\sinparen{\inparen{\sfrac{m}{(\gamma \eps^2)}}\cdot \log^2(1/\delta)}$
        \vspace{2mm}
        \For{$t = 1, 2, \dots, N$}
            \State Sample $v_t \leftarrow $ \cref{alg:grad} $(x_j, \theta_{t-1}, \cM_S(\cdot))$\textit{~~$\#$~~which satisfies $\E[v_t\mid\theta_{t-1}] = \nabla_\theta \negLL(\theta_{t-1})$}
            \State Let $\wt{\theta}_t \leftarrow \theta_{t-1} - \gamma \cdot v_t$
            \State Project $\wt{\theta}_t$ to $\Omega$ to get $\theta_j$~~\textit{$\#$~~using the Ellipsoid method \cite{grotschel1988geometric} and oracle in \cref{asmp:proj}}
        \EndFor
        \vspace{2mm}
        \State \Return $\theta_N$
    \end{algorithmic}
    \end{algorithm}

    \begin{algorithm}[ht!]
    \caption{Gradient Sampler}
        \label{alg:grad}
    \begin{algorithmic}
        \Require A point $x \in \R^d$, parameter $\theta \in \Theta$, and membership oracle $\cM_S(\cdot)$ to some set $S$
        \Oracle Exact sampling access to the distribution $\cE(\theta)$ 
        \State Draw $z \sim \cE(\theta)$ and repeat till $z\in S$ (as assessed by $\cM_S(\cdot)$)
        \State \Return $t(z)-t(x)$ %
        \textit{$\#$~~which satisfies that $\E[t(z) - t(x)] = \nabla_\theta \negLL_S{}(\theta)$}
    \end{algorithmic}
    \end{algorithm}
    \noindent To infer \cref{lem:psgd_complexity_result} from \cref{thm:sgd}, we need to prove a few statistical and algorithmic results.
    The required statistical results are as follows:
        \begin{enumerate}[leftmargin=10pt]
            \item {\bf Strong Convexity:} 
            \mbox{\cref{coro:strongConvexity} shows that $\negLL_S(\cdot)$ is $\sigma = \Omega\inparen{e^{
                    {-\frac{17k\Lambda^2}{(\alpha\eta\lambda)^2}\cdot \log{\frac{k}{\alpha}}}
                }}$-strongly convex over $\Omega$.}
            \item {\bf Upper Bound on the Second Moment of Stochastic Gradients:} 
            The following result upper bounds the second moment of the stochastic gradients.
            Its proof is left to \cref{sec:sgd_analysis:proofof:bounded_var_step}.
        \begin{restatable}[\textbf{Upper Bound on Second Moment of Stochastic Gradients}]{lemma}{upperBoundVar}\label{prop:solving_pmle_psgd:bounded_var_step}
            Fix $S$ and $\Omega$ as in \cref{lem:psgd_complexity_result}.
            Fix any $\theta \in \Omega$.
            Given independent samples $x$ and $z$ from $\cE(\theta^\star, S^\star \cap S)$ and $\cE(\theta, S)$ respectively, the vector $v = t(x) - t(z)$ satisfies that 
            $\E\left[\norm{v}_2^2 \mid \theta \right] \leq 
                e^{
                        {\wt{O}\inparen{
                            \frac{\Lambda^2}{(\alpha\eta\lambda)^2}   
                        }}
                }
            \cdot m$.
        \end{restatable}
        \end{enumerate}
        The required algorithmic results are as follows.
        \begin{enumerate}
            \item {\bf Starting Point:} 
            We use $\theta_0$ from \cref{lem:findingTheta0Formal} as the starting point for PSGD.
            \cref{lem:findingTheta0Formal} shows that $\theta_0$ is at most $O(1/(\eta\alpha\lambda))$ distance away from the optimal.
            \item {\bf Unbiased Gradient Estimation:} 
            Recall that 
            \[
                \nabla~ \negLL_S(\theta) = \Ex_{\cE(\theta^\star,S^\star\cap S)}[t(x)] - \Ex_{\cE(\theta,S)}[t(x)]\,.
            \]
            Hence, to obtain an unbiased estimate of $\nabla \negLL_S(\theta)$, 
            it suffices to have sample access to $\cE(\theta^\star,S^\star\cap S)$ and $\cE(\theta,S)$.
            Since we have sample access to $\cE(\theta^\star,S^\star)$, we can use rejection sampling to sample from $\cE(\theta^\star,S^\star\cap S)$.
            With probability $1-\delta$, this only increases the running time by a multiplicative factor of $O(\log(1/\delta))$ as the rejection probability is at most $\sfrac{
                \cE(S^\star\triangle S; \theta^\star)
            }{
                \cE(S^\star; \theta^\star)
            }\leq \eps\leq \sfrac{1}{2}$.
            Further, to sample from $\cE(\theta,S)$ again, it suffices use the sampling oracle for $\cE(\theta)$ with rejection sampling as the rejection probability is bounded away from 1; to see this note that the rejection probability is $1-\cE(S;\theta)\leq 1-e^{{
                    -{\frac{16\Lambda^2}{(\alpha\eta\lambda)^2}}
                    \cdot 
                    \inparen{1+\log\frac{1}{\alpha}}
                }}$ (due to \cref{lem:strongConvexityOnOmega}).
            
            \item {\bf Projection to $\Omega$:} 
                At each step of the PSGD procedure, we need to project the current iterate to $\Omega$.
                Recall that $\Omega$ is defined as $\Omega=\Theta \cap B(\theta_0, R)$ for an appropriate choice of $R$.
                Using a separation oracle for the $L_2$-ball $B(\theta_0,R)$, the $\poly(m/\eps)$ time projection oracle to $\Omega$ guaranteed in \cref{asmp:proj}, and the Ellipsoid method \cite{grotschel1988geometric}, one can construct a $\poly(m/\eps)$ time projection oracle to the intersection $\Omega=\Theta \cap B(\theta_0, R)$.
        \end{enumerate}

        \paragraph{PSGD Running Time.} 
        Substituting the aforementioned bounds on $\sigma$ and $\rho^2$, we can solve the recurrence in \cref{thm:sgd} (e.g., using Lemma A.3 of \citet{garrigos2023handbook}),
        to conclude that $N$ iterations of PSGD suffice to ensure $\E\snorm{\theta_t - \wh{\theta}} \leq \eps$, where 
        \[ 
            N = \tilde{O}\inparen{\frac{m}{\eps^2}}
            \cdot 
            e^{
                        {\wt{O}\inparen{
                            \frac{\Lambda^2}{(\alpha\eta\lambda)^2}   
                        }}
                }
                \cdot 
            O\inparen{e^{
                    {\frac{34k\Lambda^2}{(\alpha\eta\lambda)^2}\cdot \log{\frac{k}{\alpha}}}
            }}
            = 
            \tilde{O}\inparen{\frac{m}{\eps^2}}
            \cdot 
            {e^{\wt{O}\inparen{\frac{k\Lambda^2}{(\alpha\eta\lambda)^2}}}}
            \,.
        \]
        Now we can convert this to a guarantee that holds with probability $1-\delta$ by a standard boosting procedure as in \citet{daskalakis2018efficient}, which repeats the procedure $\log(1/\delta)$ times.
        Further, each step of the procedure can itself use $O(\log(1/\delta))$ samples due to rejection sampling subroutines.
        This brings the combined number of samples used to $O(\log^2(1/\delta))\cdot N$. %
        The total sample complexity is this number plus the number of samples needed for initializing $\theta_0$, which is $O\sinparen{
                    \inparen{m
                    /\eta^{2}}
                    \log\inparen{\sfrac{1}{\alpha}}
                    \log^2\inparen{\sfrac{1}{\delta}}
                }$ by \cref{lem:findingTheta0Formal}
        resulting in a total sample complexity of 
        \[
        O\inparen{
                    \frac{m}{\eta^{2}}
                    \log{\frac{1}{\alpha}}
                    \log^2{\frac{1}{\delta}}
                }
        ~~+~~ 
        {
            \tilde{O}\inparen{\frac{m}{\eps^2}}
            \cdot 
            {e^{\wt{O}\inparen{\frac{k\Lambda^2}{(\alpha\eta\lambda)^2}}}}
            \cdot \log^2{\frac{1}{\delta}}
        }
        ~~~~~=~~~ {
            \tilde{O}\inparen{\frac{m}{\eps^2}}
            \cdot 
            {e^{\wt{O}\inparen{\frac{k\Lambda^2}{(\alpha\eta\lambda)^2}}}}
            \cdot \log^2{\frac{1}{\delta}}
        }
        \,.
        \]
        Therefore, the algorithmic results mentioned earlier, imply a running time of
        \[
            \underbrace{\poly\inparen{\frac{dm}{\eps}}}_{\text{Finding $\theta_0$}} ~~+~~    
            \tilde{O}\inparen{\frac{m}{\eps^2}\log^2{\frac{1}{\delta}}}
            \cdot 
            {e^{\wt{O}\inparen{\frac{k\Lambda^2}{(\alpha\eta\lambda)^2}}}}
            ~\times~ \underbrace{\inparen{\poly\inparen{\frac{dm}{\eps}}
            + M_S}}_{{\text{Sampling and Projection}}}
            \,. \]
        This completes the proof of \cref{lem:psgd_complexity_result}.
        
    \subsection{Reduction from Positive-Unlabeled Learning to Agnostic Learning}\label{sec:module:reductionToLearning}
        In this section, we prove the following result. 
        \learningReduction*
        \noindent This combined with the fact that $\norm{\theta_0- \thetaStar}_2\leq O(1/(\eta\alpha\lambda))$ and \cref{thm:module:unlabeledSamples:measureGuarantees} implies the following.
        \begin{corollary}\label{thm:module:efficientLearning:denisReduction:symdiff}\label{thm:module:efficientLearning:denisReduction}
            Fix $\theta_0$ from \cref{lem:findingTheta0Formal} and $C=\max\inbrace{\sfrac{12\Lambda^2}{(\eta\alpha\lambda)^2}, \sfrac{4\Lambda}{(\eta^2\alpha \lambda)}}$. %
            For any $0<\rho\leq \eps/6$, 
            the distribution $\cD=\cD_{\rho,\cE(\thetaStar, \Sstar),\cE(\theta_0)}$ from \cref{def:distributionD} satisfies that: for any $S \subseteq \R^d$,
            \begin{align*}
                \cE\inparen{S \triangle S^\star;\theta^\star} &\leq 
                O(\rho)
                + 
                O\inparen{\frac{e^C}{\rho}}
                \cdot 
                \inparen{\err_\cD(S) - \err_\cD(S_{\rm \opt})}^{\sfrac{1}{C}}\,.
                 \yesnum\label{eq:module:efficientLearning:denisReduction:agnosticLearning}
            \end{align*}
        \end{corollary}
        Thus, if $\rho=\eps/6$, then any $S$ whose excess error with respect to $\cD$ is at most $O(\eps^{2C})$ satisfies
        \[
            \cE(S\triangle S^\star; \theta^\star)
            \leq
            O(\eps)\,.
        \]
        Therefore, to learn $S^\star$ with respect to $\cE(\thetaStar)$ it is sufficient to agnostically learn with respect to $\cD$ (which has both positive and negative samples) efficiently.

        \paragraph{Discussion of the proof of \cref{thm:efficientLearning:denisReduction}.}
        Recall the distributions $\cQ$, $\cP$, $\cU$, and $\cD=\cD_{\rho,\cP,\cU}$ from \cref{def:distributionD}. 
        For any $x \in \R^d$, let $y^\star(x)\coloneqq \mathds{1}_{P}(x)$ be the indicator whether $x\in P$--these are the true labels we want to learn. 
        Given samples $(x, y) \sim \cD$, define
            \[ 
                \gamma(x) = \Pr_{(x,y) \sim \cD}[y \neq y^\star(x)]\,. 
                \tagnum{Noise Rate}{eq:denisReduction:noiseRate}
            \]
        By our construction of $\cD$, it holds that
            \begin{align*}
                \gamma(x) &= 
                            \begin{cases}
                                \dfrac{\rho \cdot \cU(x)}{\rho \cdot \cU(x) + (1-\rho)\cP(x)} & \text{if } x \in P\,,\\
                                \hfil 0 & \text{if } x \notin P\,.
                            \end{cases}
                            \yesnum\label{eq:module:efficientLearning:denisReduction:gamma}
            \end{align*}
        Intuitively, $\gamma(\cdot)$ can be thought of as the noise in the labels in $\cD$: for each $x$, it is the probability that the observed label $y$ is different from the true label $y^\star(x)$.
        If $\sup_x\gamma(x) < \sfrac{1}{2}$, in the above construction, then it is straightforward to show that $P$ is the Bayes optimal classifier.\footnote{The Bayes optimal classifier is the one which achieves $\inf_{h} \Err_{(x,y) \sim \cD}[\mathds{1}(h(x) \neq y)]$, which is known to be the decision rule $h(x) = 1$ when $\Pr_\cD[y = 1] > 1/2$ and 0 otherwise.}
        Further, if $P$ is the Bayes optimal classifier and $\sup_x\gamma(x)$ is bounded away from $\sfrac{1}{2}$, standard agnostic learning analysis guarantees that for any empirical risk minimizer $S$, $S\triangle P$ has $\eps$ mass given sufficiently large number of samples.
        Unfortunately, this does not hold and, in fact, the noise rate $\gamma(x)$ can approach 1 for certain parts of the domain.
        We, instead, show that the part of the domain where the noise rate is larger than, e.g., $\inparen{1/2}-\Omega(1)$, is small.
        
        \paragraph{Proof of \cref{thm:efficientLearning:denisReduction}.} 
        We divide the proof of \cref{thm:efficientLearning:denisReduction} into three lemmas (\cref{lem:module:efficientLearning:denisReduction:massBz,lem:denisReduction:symDiff,lem:denisReduction:excessError}).
        For each $0\leq z\leq 1$, let $B_z$ be the set of points where the noise rate $\gamma(x)\geq z$, i.e., 
        \[
            B_z\coloneqq \inbrace{x\in\R^d\colon\gamma(x)\geq z}\,.
            \yesnum\label{def:bz}
        \]
        Observe first that $S_{\opt}$ can be written in terms of $P$ and $B_z$ for $z = 1/2$.
        \begin{restatable}[]{proposition}{sopt}\label{lem:module:efficientLearning:denisReduction:agnosticLearning:sopt} \label{lem:module:efficientLearning:denisReduction:sopt}
            $S_{\opt} = P \backslash B_{1/2}$.
        \end{restatable}
        \begin{proof}
            Since $S_{\opt}$ is the Bayes optimal hypothesis, 
            \begin{align*}
                S_{\opt} &= \inbrace{ x \in \R^d \colon \Pr\nolimits_\cD\inparen{y = 1 \mid x} > 1/2 } \\
                &= \inbrace{x \in \R^d \colon y^\star(x) = 1 \text{ and } \gamma(x) < 1/2, \text{ or } y^\star(x) = 0 \text{ and } \gamma(x) \geq 1/2} \\
                &= \inparen{(\R^d \backslash B_{1/2}) \cap \inbrace{x \in \R^d \colon y^\star(x) = 1}} \cup \inparen{ B_{1/2} \cap \inbrace{x \in \R^d \colon y^\star(x) = 0} } \\
                &= \inparen{(\R^d \backslash B_{1/2}) \cap P} \cup \inparen{B_{1/2} \cap (\R^d \backslash P)} \\
                &= P \backslash B_{1/2}\,,
            \end{align*}
            where in the last equality we use $B_{1/2} \subseteq P$ (which follows from \cref{eq:module:efficientLearning:denisReduction:gamma,def:bz}).
        \end{proof}
        The first lemma bounds the mass of $B_z$.
        \begin{restatable}[Upper Bound on the Mass of $B_z$]{lemma}{massBoundBz}
            \label{lem:module:efficientLearning:denisReduction:massBz}
            For any $0\leq z\leq 1$, %
            \[
                \cQ(B_z) \leq \frac{\rho}{1-\rho}\cdot \frac{1-z}{z}\,.
            \]
            Hence, if $\rho\leq \eps z / 2$, then $\cQ(B_z) \leq \eps$.
        \end{restatable}
        Note that as a corollary of \cref{lem:module:efficientLearning:denisReduction:agnosticLearning:sopt,lem:module:efficientLearning:denisReduction:massBz}, we have that $\cQ(S_{\opt} \triangle P) \leq \eps$ if $\rho < \eps / 4$.
        The second lemma upper bounds $\cQ(S\triangle \Sstar)$ in terms of quantities that we can relate to the excess error of $S$.
        \begin{lemma}\label{lem:denisReduction:symDiff}
            Consider the setting in \cref{thm:efficientLearning:denisReduction}.
            It holds that 
            \[
                \cQ(S\triangle P)\leq 
                \cQ(S\backslash P)
                +\cQ(S_{\opt}\backslash\inparen{S\cup B_{\sfrac{1}{3}}})
                +\cQ(B_{\sfrac{1}{3}})
                +\cQ(B_{\sfrac{1}{2}})\,.
            \]
        \end{lemma}
        The next lemma relates $\Err_\cD(S)-\Err_\cD(S_{\opt})$ to these quantities.
        \begin{lemma}\label{lem:denisReduction:excessError}
            Consider the setting in \cref{thm:efficientLearning:denisReduction}.
            It holds that 
            \[
                \Err_\cD(S)-\Err_\cD(S_{\opt})
                \geq 
                \rho\cdot e^{-C}\cdot \inparen{
                    \frac{
                        \cQ(S \backslash P) + \cQ(S_{\opt} \backslash S \backslash B_{1/3})
                    }{
                        2
                    }
                }^C\,.
            \]
        \end{lemma}
        \cref{thm:efficientLearning:denisReduction} follows by chaining \cref{lem:module:efficientLearning:denisReduction:massBz,lem:denisReduction:symDiff,lem:denisReduction:excessError}, whose proofs appear below.
        \subsubsection*{Proof of \cref{lem:module:efficientLearning:denisReduction:massBz}}

        On the one hand, from the definition of $B_z$, for all $x\in B_z$,
        \[ \E_{x \sim \cD_X}\insquare{\gamma(x) \mid x \in B_z} \geq z\,. \]
        On the other hand,
        \begin{align*}
            \E_{x \sim \cD_X}\insquare{\gamma(x) \mid x \in B_z} 
            \qquad\Stackrel{\rm\cref{eq:module:efficientLearning:denisReduction:gamma}}{=}\qquad 
            \frac{\rho \cdot \cU(B_z)}{\rho \cdot \cU(B_z) + (1-\rho) \cP(B_z)} \,.
        \end{align*}
        Substituting this in the lower bound and rearranging implies that 
        \begin{align*}
            \cU(x) \,&\geq\, \frac{1-\rho}{\rho}\cdot \frac{z}{1-z} \cdot \cP(x)\,.
        \end{align*}
        Since $\cU(B_z)\leq 1$, it follows that 
        \[ 
            \cP(B_z) \leq \frac{\rho}{1-\rho}\cdot \frac{1-z}{z}\,. 
        \]
        Since for any set $T$,  $\cQ(T) \leq \cP(T)$, the result follows.
            
        \subsubsection*{Proof of \cref{lem:denisReduction:symDiff}}
            For any $S \subseteq \R^d$, 
            \[
                  S \triangle P = S \backslash P + P \backslash S\,.
            \]
            We can further express the second term on the right side as follows
            \[
                  P \backslash S = S_{\opt} \backslash S + (P \backslash S) \backslash S_{\opt}
                ~~\qquad\Stackrel{\rm \cref{lem:module:efficientLearning:denisReduction:sopt}}{=}\qquad~~
                S_{\opt} \backslash S + (P \backslash S) \cap B_{1/2}
                \,.
            \]
            Chaining these equalities implies that 
            \begin{equation}
                \cQ(S \triangle P) 
                \leq 
                \cQ(S \backslash P) + 
                \cQ(S_{\opt} \backslash S) + 
                \cQ(B_{1/2})\,. 
                \label{eq:module:efficientLearning:denisReduction:mainThmIneq}
            \end{equation}
            The second term on the right side of the above inequality can expressed as $\cQ(S_{\opt} \backslash S) = \cQ( (S_{\opt} \backslash S) \cap B_{1/3}) + \cQ(S_{\opt} \backslash S \backslash B_{1/3})$, thus we can also write
            \[ 
                \cQ(S \triangle P) \leq 
                \cQ(S \backslash P) + 
                \cQ(B_{1/3}) + 
                \cQ(S_{\opt} \backslash S \backslash B_{1/3}) + 
                \cQ(B_{1/2})\,. 
            \]

        \subsubsection*{Proof of \cref{lem:denisReduction:excessError}}
            For any $S$ we can write the error as
            \begin{align*}
                \err_\cD(S) &= \int_{z \in S \triangle P} (1-\gamma(z)) \cD_X(z) \d z + \int_{z \notin S \triangle P} \gamma(z) \cD_X(z) \d z \\
                &= \cD_X(S \backslash P) + \int_{z \in P \backslash S} (1-\gamma(z))\cD_X(z) \d z + \int_{z \in S \cap P} \gamma(z) \cD_X(z) \d z\,,
            \end{align*}
            where $\cD_X$ is the marginal distribution of $x$ under $\cD$ (from \cref{def:distributionD}). The last equality uses that $\gamma(x) = 0$ for $x \notin P$ (by \cref{eq:module:efficientLearning:denisReduction:gamma}). Where for $S_{\opt} = P \backslash B_{1/2}$ (by \cref{lem:module:efficientLearning:denisReduction:sopt}) we have
            \[ \err_\cD(S_{\opt}) = \int_{z \in B_{1/2}} (1-\gamma(z))\cD_X(z)\d z + \int_{z \in S_{\opt}}\gamma(z) \cD_X(z) \d z\,. \]
            We can then write the difference as 
            \begin{align*}
                \err_\cD(S) - \err_\cD(S_{\opt}) &= \cD_X(S \backslash P) + \int_{S_{\opt} \backslash S}(1-2\gamma(z))\cD_X(z) \d z + \int_{B_{1/2} \cap S}(2\gamma(z) - 1) \cD_X(z) \d z \\
                &\geq \cD_X(S \backslash P) + \int_{S_{\opt} \backslash S}(1-2\gamma(z))\cD_X(z) \d z \tag{since $\gamma(z) \geq 1/2$ for $z \in B_{1/2}$} \\
                &\geq \cD_X(S \backslash P) + \int_{S_{\opt} \backslash S \backslash B_{1/3}} (1-2\gamma(z))\cD_X(z) \d z \tag{since $\sfrac{1}{2} \leq \gamma(z) \leq \sfrac{1}{3}$ for $z \in S_{\opt} \backslash S \backslash B_{1/3}$}\\
                &\geq \cD_X(S \backslash P) + \frac{1}{3}\cD_X(S_{\opt} \backslash S \backslash B_{1/3})\,.
            \end{align*}
            Further, \cref{thm:module:unlabeledSamples:measureGuarantees} implies the following two inequalities
            \begin{align*}
                \cD_X(S \backslash P) 
                &~\geq~ \rho \cdot \cU(S \backslash P) 
                ~\geq~ \rho\cdot e^{-C}\cdot \cQ(S\backslash P)^C\,, \\
                \cD_X(S_{\opt} \backslash S \backslash B_{1/3}) 
                &~\geq~ (1-\rho)\cdot \cP(S_{\opt} \backslash S \backslash B_{1/3}) 
                ~\geq~ (1-\rho) \cdot \cQ(S_{\opt} \backslash S \backslash B_{1/3})\,.
            \end{align*}
            Finally, since $2^k(a^k + b^k) \geq (a+b)^k$ for any $k\geq 1$ and $a,b\geq 0$, and $1-\rho \geq \sfrac{\rho}{C}$, we have
            \begin{equation}
                \err_\cD(S) - \err_\cD(S_{\opt})
                \geq 
                \rho\cdot e^{-C}
                \inparen{
                    \frac{
                        \cQ(S \backslash P) + 
                        \cQ(S_{\opt} \backslash S \backslash B_{1/3})
                    }{
                        2
                    }
                }^C\,. 
                \label{eq:module:efficientLearning:denisReduction:mainThmSecTerm}
            \end{equation}

    \subsection{Efficiently Learning \texorpdfstring{$S^\star$}{S*} via \lreg{}}\label{sec:module:efficientLearning}
            In the last section, we constructed a distribution $\cD$ such that any set $S$ that has agnostic error $\poly(\eps)$-close to the optimal with respect to $\cD$ satisfies $\cE\inparen{S\triangle S^\star;\theta^\star}=O(\eps)$ as desired by \cref{thm:module:reduction}.
            In this section, we prove \cref{lem:module:efficientLearning:upperBoundOnApproximability:informal} that that enables us to use \lreg{} to learn $S^\star$ as explained in \cref{sec:intro:efficientLearning}.
            The formal statement of \cref{lem:module:efficientLearning:upperBoundOnApproximability:informal} is as follows.
            
                \begin{restatable}[]{lemma}{upperBoundonApproximabilityFormal}\label{lem:module:efficientLearning:upperBoundOnApproximability}
                    Let $\cD$ be the distribution in \cref{thm:module:efficientLearning:denisReduction}.
                    Suppose \cref{asmp:bridge} holds.
                    Fix 
                    \[
                        \phi=\eps^2\cdot e^{-\poly\inparen{\sfrac{\Lambda}{\lambda}} / \alpha^2}\,.
                    \]
                    Fix $\ell$ such that degree-$\ell$ polynomials $\phi$-approximate $\hyS$ in $L_2$-norm with respect to the family $\cE$  (where, typically, $\ell\geq \poly(1/\phi)$).
                    Then degree-$\ell$ polynomials is $\eps$-approximate $\hyS$ in $L_1$-norm with respect to $\cD_X$.
                \end{restatable}
                \noindent We will use the following lemma to prove the above result.
                        \begin{lemma}[{$\cD$ Is Close to Log-Concave}]\label{lem:module:efficientLearning:DisCloseToLogConcave}
                            Let $\cE$ be any exponential family satisfying \cref{asmp:1:sufficientMass,asmp:1:polynomialStatistics,asmp:2}.
                            For the distribution $\cD$ in \cref{thm:module:efficientLearning:denisReduction}, 
                            there exists $\theta$ such that
                            $\renyi{2}{\cD_X}{\cE(\theta)}\leq {
                                    {
                                        \alpha^{-2}\cdot \poly\inparen{\sfrac{\Lambda}{\lambda}}
                                    }
                                }.$
                        \end{lemma}
                        \begin{proof}[Proof of \cref{lem:module:efficientLearning:DisCloseToLogConcave}]
                            For $\theta_1=\theta^\star$ and $\theta_2=\theta_0$, let $\cE(\theta)$ be the distribution promised in \cref{asmp:bridge}.
                            Recall that $\cD_X$ is a mixture of $\cE(\theta^\star,S^\star)$ and $\cE(\theta_0)$ and, hence, %
                            \begin{align*}
                                \frac{e^{\renyi{2}{\cD_X}{\cE(\theta)}}}{2}
                                &=\frac{1}{2}\int_x{{
                                    \frac{{\d \cD_X(x)}^2}{\d \cE(x;\theta)}
                                }}\d x\\
                                &\leq 
                                    {
                                        \int_x
                                            \frac{{\d \cE(x; \theta^\star, S^\star)}^2}{\d \cE(x;\theta)}
                                        \d x
                                        + 
                                        \int_x \frac{{\d \cE(x; \theta_0)}^2}{\d \cE(x;\theta)}\d x
                                    }
                                    \tag{as $(a+b)^2\leq 2(a^2+b^2)$ for all $a,b\in \R$}
                                    \\
                                &= {e^{\renyi{2}{\cE(\theta^\star, S^\star)}{\cE(\theta)}} + e^{\renyi{2}{\cE(\theta_0)}{\cE(\theta)}}}   
                                \,.
                            \end{align*} 
                            Toward upper bounding the first term, observe that 
                            \[
                                e^{\renyi{2}{\cE(\theta^\star, S^\star)}{\cE(\theta)}}
                                = 
                                \int_x{{
                                    \frac{{\d \cE(\theta^\star, S^\star)}^2}{\d \cE(x;\theta)}
                                }}\d x
                                \leq \frac{1}{\cE(S^\star; \theta^\star)^2}
                                \int_x{{
                                    \frac{{\d \cE(\theta^\star)}^2}{\d \cE(x;\theta)}
                                }}\d x
                                \leq \frac{1}{\alpha^2}
                                \int_x{{
                                    \frac{{\d \cE(\theta^\star)}^2}{\d \cE(x;\theta)}
                                }}\d x\,.
                            \]
                            Now, due to \cref{asmp:2} 
                            $\chidiv{\cE(\theta^\star, S^\star)}{\cE(\theta)}, \chidiv{\cE(\theta_0)}{\cE(\theta)}\leq e^{\poly(\sfrac{\Lambda}{\lambda})}-1$ and, hence, by the standard relation between $\chi^2$-divergence and the second-order Rényi divergence (\cref{sec:preliminaries}) 
                            \[
                                  \renyi{2}{\cE(\theta^\star, S^\star)}{\cE(\theta)}, \renyi{2}{\cE(\theta_0)}{\cE(\theta)}\leq \poly\inparen{\frac{\Lambda}{\lambda}}
                                  \,.  
                            \] 
                            Therefore
                            \[
                                {e^{\renyi{2}{\cD_X}{\cE(\theta)}}}
                                \leq 4\cdot e^{-\poly\inparen{\sfrac{\Lambda}{\lambda}} / \alpha^2}\,.
                            \]
                        \end{proof}
                     
                \noindent Now we are ready to prove \cref{lem:module:efficientLearning:upperBoundOnApproximability}.
                \begin{proof}[Proof of \cref{lem:module:efficientLearning:upperBoundOnApproximability}]
                    Fix the distribution $\cE(\theta)$ in \cref{lem:module:efficientLearning:DisCloseToLogConcave}.
                    Hence, it holds that, for all $S\in \hyS$
                    \[
                        \min_{{\rm deg}(p)\leq \ell}
                        ~~
                        \Ex_{x\sim \cE(\theta)}\insquare{\abs{\mathds{1}_S(x)-p(x)}^2}
                        ~~\leq~~
                        \eps^2 \cdot e^{-\poly\inparen{\sfrac{\Lambda}{\lambda}} / \alpha^2}
                        \,.
                        \yesnum\label{eq:module:efficientLearning:squareBound}
                    \]
                    It suffices to show that for any $S\in \hyS$, there exists a degree-$\ell$ polynomial $p\colon\R^d\to \R$ such that $\Ex_{x\sim \cD_X}\insquare{\abs{\mathds{1}_S(x)-p(x)}}\leq \eps.$
                    To see this observe that Cauchy-Schwarz Inequality implies that for all $p\colon \R^d\to \R$
                    \[
                        \Ex_{x\sim \cD}\insquare{\abs{\mathds{1}_S(x)-g(x)}}
                        \leq 
                            \sqrt{\int_x\frac{\d\cD(x)^2}{\d \cE(x;\theta)}\d x}\cdot 
                            \sqrt{
                                \Ex_{x\sim \cE(\theta)}\insquare{\abs{\mathds{1}_S(x)-p(x)}^2}
                            }
                        \,.
                    \]
                    Now, \cref{eq:module:efficientLearning:squareBound} implies that 
                    \[
                        \Ex_{x\sim \cD}\insquare{\abs{\mathds{1}_S(x)-g(x)}}
                        \leq 
                            \sqrt{\int_x\frac{\d\cD(x)^2}{\d \cE(x;\theta)}\d x}\cdot 
                            \eps\cdot 
                            e^{-\poly\inparen{\sfrac{\Lambda}{\lambda}} / \alpha^2}
                        \qquad\Stackrel{\rm \cref{lem:module:efficientLearning:DisCloseToLogConcave}}{\leq}\qquad 
                        \eps\,.
                    \]
                    
                \end{proof}

    \begin{algorithm}[t!]
        \caption{
            Algorithm  to Learn the Survival Set
        }
        \label{alg:pu}
        \begin{algorithmic}
        \Require A parameter $\theta_0 \in \Theta$, constants $\eps, \delta$, independent samples $x_1,x_2, \dots$ from $\cE(\theta^\star, S^\star)$, and\\
        \hspace{8mm} \textit{optionally}, a constant $\ell$ such that $\cS$ is $\eps$-approximable in the $L_2$ distance by degree-$\ell$\\
        \hspace{8mm} polynomials with respect to distribution $\cE(\thetaStar)$
        \Ensure $\norm{\theta_0-\thetaStar}_2\leq O(\Lambda/(\alpha\eta\Lambda))$ 
        \vspace{2mm}
        \State Draw $\poly\inparen{
                    \eps^{-1}\cdot {(d+m)}^{\ell}
                    \cdot  \log\inparen{\sfrac{1}{\delta}}
                }$ samples from $\cD_{\eps/6, \cE(\thetaStar, \Sstar), \cE(\theta_0)}$ defined in \cref{def:distributionD}
        \State \Return the membership oracle $\cM_S(\cdot)$ found by the $\lreg{}$ algorithm \cite{kalai2008agnostically}
        \end{algorithmic}
        \end{algorithm}

\section{Polynomial Time Algorithms with Halfspace and Axis-Aligned Boxes}\label{sec:computationalEfficiency:polyTime}\label{sec:module:halfspaceLearner} 
    In this section, we prove our polynomial-time estimation results.
    The first result is \cref{infthm:polyTime} whose formal statement is as follows.
    \begin{theorem}[{Polynomial Time Estimation for Gaussians}]\label{thm:mainPolynomial}
        Suppose $S^\star\subseteq\R^d$ is (a) a halfspace or (b) an axis-aligned rectangle with at least $\alpha>0$ mass under the $d$-dimensional Gaussian distribution $\cN(\muStar, \SigmaStar)$.
        Fix any constants $\eps,\delta\in (0,1/2)$ and  
        \[
                n =  
                \poly\inparen{
                        \frac{d}{\eps}
                        \cdot \log{\frac{1}{\delta}}
                    }
                \cdot
                e^{\poly\inparen{\sfrac{1}{\alpha}}}
                \,. 
        \]
        There exists an algorithm that, given $n$ independent samples from $\cN(\muStar, \SigmaStar, S^\star)$ and the ``type'' of $S^\star$ (i.e., (a) or (b)), 
        outputs estimates $\mu$ and $\Sigma$, in $\poly(n)$ time, such that, with probability at least $1-\delta$, 
        \[
            \tv{\cE(\theta)}{\cE(\theta^\star)} \leq \eps\,.
        \]
    \end{theorem}
    This is the first $\poly(d/\eps)$ algorithm for statistical estimation under truncation to unknown survival sets.
    We stress that, even in the special case of survival sets that are halfspaces or axis-aligned rectangles, \citet{Kontonis2019EfficientTS}'s algorithm has a running time $d^{\poly(1/\eps)}$ and, moreover, requires the unknown covariance matrix $\SigmaStar$ to be close to diagonal.
    Furthermore, for a general $\Sstar$ (with constant VC dimension and Gaussian surface area), this dependence on $d$ and $\eps$ is necessary for any algorithm in SQ (a powerful class of learning algorithms) due to a recent lower bound by \citet{diakonikolas2024statistical}.
    Hence, one has to restrict attention to specific survival sets to obtain $\poly(d/\eps)$ algorithms even for estimating truncated Gaussians.

    Our next result is a polynomial-time algorithm for estimating truncated exponential distributions when $\Sstar$ is an axis-aligned rectangle (which was stated as \cref{infthm:polyTime:exponential} earlier).
    \begin{theorem}[Polynomial-Time Estimation for Exponential Families]\label{thm:polyTime:exponential}
        Suppose $S^\star\subseteq\R^d$ is an axis-aligned rectangle.
        Let \cref{asmp:1:sufficientMass}, \ref{asmp:1:polynomialStatistics}, and \ref{asmp:cov} to \ref{asmp:proj} hold.
        Fix any constants $\eps,\delta\in (0,1/2)$, and 
        \[
                n = 
                    \poly\inparen{
                        \frac{dm}{\eps}
                        \cdot 
                        \log{\frac{1}{\delta}}
                    }
                    \cdot 
                    e^{\wt{O}\inparen{\frac{k\Lambda^2}{(\alpha\eta\lambda)^2}}}
                    \,.
        \]
        There exists an algorithm that, given $n$ independent samples from $\cE(\theta^\star, S^\star)$ and
        constants $(\lambda,\Lambda,k,\alpha,\ell)$,
        in $\poly(n)$ time, outputs an estimate $\theta$, such that, with probability at least $1-\delta$, 
        \[
            \tv{\cE(\theta)}{\cE(\theta^\star)} \leq \eps\,.
        \]
    \end{theorem}
    \begin{proof}[Proof of \cref{thm:polyTime:exponential}]
        The algorithm in the above result has two parts.
        First, it learns a set $S$ such that $\cE(S\triangle \Sstar; \thetaStar)\leq \alpha^3\eps^2$ with $\poly(d/(\alpha\eps))$ samples from $\cE(\thetaStar, \Sstar)$ using a folklore algorithm for learning axis-aligned rectangles from positive samples \cite{shalev2014understanding}.
        Second, it learns $\thetaStar$ using the algorithm in \cref{thm:module:reduction}.
        The running time bound and correctness of this algorithm follow as corollaries of \cref{thm:module:reduction} and the running time and correctness of the aforementioned folklore algorithm.
    \end{proof}

    \noindent The remainder of this section proves \cref{thm:mainPolynomial}.
    The algorithm in \cref{thm:mainPolynomial} is \cref{alg:mainPolynomial}.
    Recall that in \cref{sec:intro:polytime}, we showed that to prove \cref{thm:mainPolynomial}, it is sufficient to prove \cref{thm:halfspaceLearner} and further divided the proof of \cref{thm:halfspaceLearner} into \cref{lem:halfspaceLearner:moment,lem:halfspaceLearner:constantLB,lem:halfspaceLearner:goodEvent,lem:halfspaceLearner:symDiff}.
    The proofs of these lemmas are presented below.
    \begin{algorithm}[ht!]
    \caption{Polynomial Time Gaussian Estimation with ``Simple'' Survival Sets (\cref{thm:mainPolynomial})}\label{alg:mainPolynomial}
    \begin{algorithmic}
    \Require 
        Independent samples $x_1, x_2, \dots, x_n\in \R^d$ from $\cE(\theta^\star, S^\star)$ where $\thetaStar=\inparen{{\SigmaStar}^{-1}\muStar, \frac{1}{2}{\SigmaStar}^{-1}}$
    \Ensure Set $\Sstar$ is a halfspace or axis-aligned rectangle with mass $\cE(\Sstar;\thetaStar)\geq \alpha > 0$
    \vspace{2mm}
    \State\textbf{(Subroutine A) Pre-process the samples and construct $\Theta$} 
    \begin{enumerate}
        \item Obtain a linear transformation $L\colon\R^d\to\R^d$ and set $\Theta$ from subroutine \cref{alg:preprocess}$(x_{1~\colon \nfrac{n}{3}})$
        \item Apply the linear transformation to remaining samples: $x_i\gets L(x_i)$ for each $n/3<i\leq n$
    \end{enumerate}
    \State \hspace{2.5mm} \textit{$\#$~~~Guarantee: \cref{asmp:2} holds with above $\Theta$ and $\eta,\lambda,(1/\Lambda)=\poly(\alpha)$, with probability $1-\delta$}
    \State \hspace{2.5mm} \textit{$\#$~~~Run time: $O(n^2)$}
    \vspace{4mm}
    \State\textbf{(Subroutine B) Find a set $S$ such that $\cE(S\triangle S^\star;\thetaStar)\leq \poly(\alpha\eps)$}
    \begin{enumerate}
        \item[3.] \textbf{if} $\Sstar$ is an axis-aligned rectangle \textbf{then:}
        \item[3.1\hspace{-2mm}] \hspace{6mm} For each $1\leq j\leq d$, compute $L_j=\min_{\nfrac{n}{3}<i\leq \nfrac{2n}{3}} x_{ij}$ and $R_j=\max_{\nfrac{n}{3}<i\leq \nfrac{2n}{3}}x_{ij}$ 
        \item[3.2\hspace{-2mm}] \hspace{6mm} Let $\cM_S$ be the membership oracle to $S=\bigtimes_{j=1}^d~ [L_j,R_j]$ 
        \item[4.] \textbf{else if} $\Sstar$ is an halfspace\textbf{:}
        \item[4.1\hspace{-2mm}] \hspace{6mm} Obtain a membership oracle $\cM_S$ from  \cref{alg:halfspace}~($\eps, \delta, x_{\inparen{\nfrac{n}{3}} + 1~:~\nfrac{2n}{3}}$) 
    \end{enumerate}
    \State \hspace{2.5mm} \textit{$\#$~~~~\mbox{Guarantee: $\cM_S$ is a membership oracle for $S$ that, with probability $1-\delta$, satisfies $\cE(S\triangle S^\star;\thetaStar)\leq \poly(\alpha\eps)$}}
    \State \hspace{2.5mm} \textit{$\#$~~~Run time: $\poly(dm/\eps)$}
    \vspace{4mm}
    \State\textbf{(Subroutine C) Find $\theta\approx\theta^\star$}
    \begin{enumerate}[]
        \item[5.] Define $\theta_0\gets (0, I/2)$, $\eta\gets \poly(\alpha)$, $\lambda\gets \poly(\alpha)$, and $\Lambda\gets \poly(1/\alpha)$
        \item[6.] Find $\theta$ by solving the \pmle{} problem: \cref{alg:psgd} $(\theta_0,\cM_S, \alpha, \eta, \lambda, \Lambda, x_{\inparen{\nfrac{2n}{3}}+1:n})$
    \end{enumerate}
    \State \hspace{2.5mm} \textit{$\#$~~~Guarantee: With probability $1-\delta$, $\theta$ is at a distance at most $\eps$ from $\thetaStar$}
    \State \hspace{2.5mm} \textit{$\#$~~~Run time: $\poly(dm/\eps)\cdot M_S$ where $M_S$ running time of $\cM_S$}
    \vspace{2mm}
    \State \Return estimate $\theta$
    \end{algorithmic}
    \end{algorithm}

    \subsection*{Proof of \cref{lem:halfspaceLearner:moment}}     
        First, we prove \cref{lem:halfspaceLearner:moment}, which we restate below for the reader's convenience.
        \halfspaceLearnerMoment*
        \noindent Let $r$ be the transformation of the truncated sample $x\sim \cN(\muStar,\SigmaStar,\Sstar)$ after applying the aforementioned linear transformation, i.e., $r=U{\SigmaStar}^{-1/2}(x-\muStar)$.
        Due to the mixed product property of the Tensor product, it holds that 
        \begin{align*}
           \Ex\insquare{\inparen{x-\gamma}^{\otimes 3}}
           &= \Ex\insquare{\inparen{x-\Ex[x]}^{\otimes 3}} 
            \tag{since $\gamma=\Ex[x]$}\\
           &= \Ex\insquare{\inparen{
                    \inparen{x-\muStar}
                    -\Ex\insquare{\inparen{x-\muStar}}
                }^{\otimes 3}}\\
           &= \inparen{
                \inparen{{\SigmaStar}^{1/2} U^\top }^{\otimes 3}
            }
            \Ex\insquare{\inparen{
                    r
                    -\Ex\insquare{r}
                }^{\otimes 3}}\,.
            \tagnum{since $r=U{\SigmaStar}^{-1/2}(x-\muStar)$}{eq:halfspaceLearner:mixedProduct}
        \end{align*}
        Further, due to the chosen linear transformation, $r_2,r_3,\dots,r_d$ are independent draws from the standard normal distribution and $r_1$ is a draw from the standard normal distribution truncated to $[\tau,\infty)$ for $\tau$ such that $\frac{1}{\sqrt{2\pi}}\int_{\tau}^\infty e^{-t^2/2}\d t=\cN(\Sstar; \muStar,\SigmaStar)$.
        Hence, for any unit vector $v\in \R^d$, the mixed-product property of the tensor product implies that 
        \begin{align*}
            \inparen{v^{\otimes 3}}^\top \Ex\insquare{\inparen{r-\Ex[r]}^{\otimes 3}}
            &= \Ex\insquare{\inangle{v,r - \Ex[r]}^3}\,.
        \end{align*}
        The RHS expands to the following 
        \begin{align*}
            \Ex\insquare{v_1^3 \cdot  \inparen{r_1 - \Ex[r_1]}^3}
            +   \Ex\insquare{v_1 \cdot \inparen{r_1 - \Ex[r_1]}} \cdot {
                \Ex\insquare{
                     \inangle{v_{-1}, r_{-1} - \Ex[r_{-1}]}^2
                }
            }
                + 
                \Ex\insquare{
                  \inangle{v_{-1}, r_{-1} - \Ex[r_{-1}]}^3
                }
            \,.
        \end{align*}
        Where for a vector $u$, $u_{-1}$ denotes the vector obtained by dropping the first coordinate of $u$.
        The first expectation in the second term is zero due to the linearity of expectation.
        Further, since $r_2,\dots, r_d\sim \cN(0,1)$, Isserlis's theorem implies that the last term is 0.
        Therefore, it follows that 
        \begin{align*}
            \inparen{v^{\otimes 3}}^\top \Ex\insquare{\inparen{r-\Ex[r]}^{\otimes 3}}
            =  v_1^3 \cdot \Ex\insquare{\inparen{r_1 - \Ex[r_1]}^3}\,.
        \end{align*}
        Hence, as an operator over $\inbrace{v^{\otimes 3}\colon v\in \R^d}$
        \[
            \Ex\insquare{\inparen{r-\Ex[r]}^{\otimes 3}}
            =
            \Ex\insquare{\inparen{r_1-\Ex[r_1]}^3} \cdot e_1\otimes e_1 \otimes e_1\,.
        \]
        Substituting this in Equation~\eqref{eq:halfspaceLearner:mixedProduct} and using the mixed-product property of tensor product implies 
        \begin{align*}
            \Ex\insquare{\inparen{x-\gamma}^{\otimes 3}}
            ~~&=~~
            \Ex\insquare{\inparen{r_1-\Ex[r_1]}^3} \cdot 
                \inparen{U{\SigmaStar}^{-1/2}e_1} \otimes 
                \inparen{U{\SigmaStar}^{-1/2}e_1} \otimes
                \inparen{U{\SigmaStar}^{-1/2}e_1}\\
            &\Stackrel{\eqref{eq:polytime:rotation}}{=}~~
            \Ex\insquare{\inparen{r_1-\Ex[r_1]}^3}\cdot w^{\otimes 3}\,.
        \end{align*}
        This completes the proof.
    
    \subsection*{Proof of \cref{lem:halfspaceLearner:constantLB}} 
        Next, we prove \cref{lem:halfspaceLearner:constantLB}, which we restate below for the reader's convenience.
        \halfspaceLearnerconstantLB*
        \noindent Recall that $z$ is a standard normal variable truncated to $[\tau,\infty)$, where $\tau$ satisfies 
        \[
            \frac{1}{\sqrt{2\pi}}\int_{\tau}^\infty e^{-t^2/2}\d t = \cN(\Sstar; \muStar, \SigmaStar)\,.
        \]
        Our goal is to show that 
        \[
            \Ex_{}\insquare{\inparen{z - \Ex[z]}^3}
            \geq 
            \Omega\inparen{\min\inbrace{1, \tau^{-12}}}\cdot e^{-\tau^2/2}
            \,.
        \]
        Linearity of expectation implies that 
        \[
            \Ex_{}\insquare{\inparen{z - \Ex[z]}^3}
            = 
            \Ex[z^3] + 2 \Ex[z^2] - 3 \Ex[z]\Ex[z^2]\,.
        \]
        Standard expressions of the moments of the truncated normal distributions imply that 
        \[
            \Ex[z] = \frac{\phi(\tau)}{1-\Phi(\tau)}\,,
            \quad 
            \Ex[z^2] = 1 + \frac{\tau \phi(\tau)}{1-\Phi(\tau)}\,,
            \quad 
            \Ex[z^3] = \inparen{2+\beta^2} \cdot \frac{\phi(\tau)}{1-\Phi(\tau)}\,.
        \]
        Where $\phi(\cdot)$ and $\Phi(\cdot)$ are the standard normal distribution's density and cumulative density functions. 
        Here, $\frac{\phi(\cdot )}{1-\Phi(\cdot )}$ is the Hazard rate, which is a fundamental quantity associated with a distribution.
        It arises in diverse fields from survival analysis to auction theory and is known by many names, including failure rate and the inverse of the Mills' ratio \cite{cox1972regression,mechanismDesign2013hartline}.
        Since it shows up often in the subsequent analysis, we define 
        \[
            H(t) \coloneqq \frac{\phi(t)}{1-\Phi(t)}\,.
        \]
        So far, we have shown that 
        \[
            \Ex_{}\insquare{\inparen{z - \Ex[z]}^3}
            = 
            H(\tau)\cdot 
            \inparen{
                \tau^2 - 1 + 2 H^2(\tau) - 3 \tau H(\tau)
            }
            \,.
        \]
        The first term in the above product can be lower bounded straightforwardly.
        \begin{lemma}
            For any $t\in \R$, it holds that $H(t)\geq \frac{1}{\sqrt{2\pi}}\cdot e^{-t^2/2}$.
        \end{lemma}
        \begin{proof}
            Since $1-\Phi(t)\leq 1$ and $\phi(t)\geq 0$, $H(t)\geq \phi(t) = \frac{1}{\sqrt{2\pi}}\cdot e^{-t^2/2}.$
        \end{proof}
        The second term is harder to lower bound and the rest of the proof is devoted to it.
        \begin{lemma}\label{lem:halfspaceLearner:constantLB:secondTerm}
            For any $t\in \R$, it holds that 
            \[
                t^2 - 1 + 2 H^2(t) - 3 t H(t)
                \geq 
                \Omega\inparen{\min\inbrace{1, \tau^{-12}}}\,.
            \]
        \end{lemma}
        Combining these two lemmas implies \cref{lem:halfspaceLearner:constantLB}.
        The proof of \cref{lem:halfspaceLearner:constantLB:secondTerm} is rather technical and is deferred to \cref{sec:proofof:lem:halfspaceLearner:constantLB:secondTerm}.

    \subsection*{Proof of \cref{lem:halfspaceLearner:goodEvent}} 
        Next, we prove \cref{lem:halfspaceLearner:goodEvent}, which we restate below for the reader's convenience.
        \halfspaceLearnergoodEvent*
        \noindent 
        We divide the proof into two cases depending on whether the following event holds
        \[
            \max_{1\leq i\leq d}~\abs{
                \Ex_{}\insquare{
                    \inparen{x - \gamma}^{\otimes 3}
                }_{iii}
            }
            \leq d^{-3/2}\cdot \poly(\eps)\,.
            \yesnum\label{eq:halfspaceLearner:goodEvent}
        \]
        \paragraph{Case A (Event~\eqref{eq:halfspaceLearner:goodEvent} holds):}
            Since $\norm{w^\star}_2=1$, $\wStar_j\geq d^{-1/2}$ for some $1\leq j\leq d$.
            Further %
            \[
                \frac{
                    \Ex\insquare{\inparen{z - \Ex[z]}^3}
                }{d^{3/2}} 
                \leq 
                {\wStar_j}^3 \cdot \Ex\insquare{\inparen{z - \Ex[z]}^3}
                \qquad\Stackrel{\rm\cref{lem:halfspaceLearner:moment}}{\leq} \qquad
                \max_{1\leq i\leq d}~\abs{
                    \Ex_{}\insquare{
                        \inparen{x - \gamma}^{\otimes 3}
                    }_{iii}
                }
                \leq 
                    \frac{
                        \poly(\eps)
                    }{
                        d^{3/2}
                    }
                    \,.
            \]
            Hence, $\Ex[(z-\Ex[z])^3]\leq \poly(\eps)$.
            Therefore, \cref{lem:halfspaceLearner:constantLB} implies that 
            \[
                \Omega\inparen{
                    \min\inbrace{1, \tau^{-12}}
                }
                    \cdot e^{-\tau^2/2}
                \leq \poly(\eps)\,.
            \]
            Since $\Omega\inparen{e^{-r^2}}\leq {{\min\inbrace{1, r^{-12}}}}\cdot e^{-r^2/2}$ for all $r\in \R$, it follows that 
            \[
                \tau^2 \geq \Omega\inparen{\log\frac{1}{\eps}}\,.
            \]
            If $\tau\geq O(\sqrt{\log(1/\eps)})$, then 
            \[
                \cN(\Sstar;\muStar, \SigmaStar)
                = \frac{1}{\sqrt{2\pi}}\int_{\tau}^\infty e^{-t^2/2}\d t
                \leq \frac{1}{\tau+\sqrt{\tau^2+\frac{8}{\pi}}} \cdot\sqrt{\frac{2}{\pi}}\cdot e^{-\tau^2/2}
                \leq \poly(\eps)\,.
            \]
            Where the first equality is by the definition of $\tau$ (see \cref{lem:halfspaceLearner:moment}) and the second inequality is a standard tail bound for the Gaussian distribution \cite{abramowitz1965handbook}.
            However, this is a contradiction, since $\cN(\Sstar; \muStar, \SigmaStar)$ and, hence, $\tau\leq -O(\sqrt{\log(1/\eps)})$.
            This implies that 
            \[
                \cN(\Sstar;\muStar, \SigmaStar)
                = 1 - \frac{1}{\sqrt{2\pi}}\int_{-\infty}^\tau e^{-t^2/2}\d t
                \geq 1 - \frac{1}{\tau+\sqrt{\tau^2+\frac{8}{\pi}}} \cdot\sqrt{\frac{2}{\pi}}\cdot e^{-\tau^2/2}
                \geq 1 - \poly(\eps)\,.
            \]
            Where, again, the first equality is by the definition of $\tau$ (see \cref{lem:halfspaceLearner:moment}) and the second inequality is a standard tail bound for the Gaussian distribution \cite{abramowitz1965handbook}.

        \paragraph{Case B (Event~\eqref{eq:halfspaceLearner:goodEvent} does not hold):}
            Since $\norm{w^\star}_2=1$, $\wStar_j\leq 1$ for all $1\leq j\leq d$.
            Further 
            \[
                \Ex\insquare{\inparen{z - \Ex[z]}^3}
                \geq 
                \Ex\insquare{\inparen{z - \Ex[z]}^3} \cdot \max_{1\leq i\leq d} {\wStar_i}^3
                \qquad\Stackrel{\rm\cref{lem:halfspaceLearner:moment}}{\geq} \qquad
                \max_{1\leq i\leq d}~\abs{
                    \Ex_{}\insquare{
                        \inparen{x - \gamma}^{\otimes 3}
                    }_{iii}
                }
                \geq 
                d^{-3/2}\cdot \poly(\eps)\,.
            \]

        \newcommand{\sshifted}{\ensuremath{S^\star_{t}~}}
    \subsection*{Proof of \cref{lem:halfspaceLearner:symDiff}} 
        Next, we prove \cref{lem:halfspaceLearner:symDiff}, which we restate below for the reader's convenience.
        \halfspaceLearnersymDiff*
        \noindent 
        Let $S$ be the halfspace in Step 7 of the algorithm.
        In other words, $S$ is of the form $\inbrace{x\in \R^d\colon \inangle{w, x}\geq {\tau}}$ where ${\tau}$  is selected to be the maximum value subject to the constraint that $S$ contains all samples $x_1,\dots,x_n$.
        Define $\sshifted$ to be the halfspace constructed by translating $S^\star$ in the direction of $\wStar$ by the maximum amount such that the halfspace continues to contain all samples $x_1,\dots,x_n$.
        From standard VC analysis (e.g., used to learn halfspaces with a known normal vector), it follows that 
        \[
            \cN\inparen{\sshifted\triangle \Sstar; \muStar,\SigmaStar}\leq \poly\inparen{\frac{\eps}{d}}\,.
            \yesnum\label{eq:halfspaceLearner:symDiff:1}
        \]
        (A bit more concretely, one can, e.g., use the results from \citet{natarajan1987learning} by observing that halfspaces with a fixed normal vector are minimally consistent.)
        If we can prove that 
        \[
            \cN\sinparen{S\triangle \sshifted; \muStar,\SigmaStar}\leq \poly\inparen{\frac{\eps}{\alpha d}}\,,
            \yesnum\label{eq:halfspaceLearner:symDiff:0}
        \]
        then this combined with \cref{eq:halfspaceLearner:symDiff:1} implies that 
        \[
            \cN\sinparen{S\triangle \Sstar; \muStar,\SigmaStar}
            \leq 
            \cN\sinparen{S\triangle \sshifted; \muStar,\SigmaStar}
            + 
            \cN\inparen{\sshifted\triangle\Sstar ; \muStar,\SigmaStar}
            \leq \poly\inparen{\frac{\eps}{\alpha d}}\,,
        \]
        as required. 
        In the rest of the proof, we prove \cref{eq:halfspaceLearner:symDiff:0}.

        Since linear transformation do not change the probability mass of sets, without loss of generality, we transform the underlying space so that $\muStar=0$, $\SigmaStar=I$, and ${\rm span}(w,\wStar)\subseteq{\rm span}(e_1,e_2)$ (with strict containment only when $w=\wStar$). %
        Using well-known tail bounds for the normal distribution, it follows that $1-\poly(\eps/d)$ fraction of the total mass of $\cN(\muStar,\SigmaStar)$ is located within a distance $\sqrt{d}+\polylog(d/\eps)$ of the origin.
        Hence, if $B$ is the $L_2$-ball of radius $\sqrt{d}+\polylog(d/\eps)$ centered at the origin, then 
        \[
            \abs{
                \cN\inparen{S\triangle \sshifted; \muStar,\SigmaStar}
                -
                \cN\inparen{
                    \inparen{S\triangle \sshifted}\cap B; \muStar,\SigmaStar
                }
            }
            \leq \poly\inparen{\frac{\eps}{\alpha d}}\,. 
            \yesnum\label{eq:halfspaceLearner:triangleInEq}
        \]
        Therefore, to prove \cref{eq:halfspaceLearner:symDiff:0}, it suffices to prove that  
        \[
            \cN\inparen{
                    \inparen{S\triangle \sshifted}\cap B; \muStar,\SigmaStar
                }
            \leq \poly\inparen{\frac{\eps}{\alpha d}}\,.
        \]
        To bound $\cN\inparen{
                    \inparen{S\triangle \sshifted}\cap B; \muStar,\SigmaStar
                }$, we claim that $\inparen{S\triangle \sshifted}\cap B$ is contained within the infinite cuboid $T\times (-\infty, \infty)^{d-1}$ where $T$ is an interval of size $\poly\inparen{\eps/(\alpha d)}$.
        If this is true, we get the desired upper bound on $\cN\inparen{
                    \inparen{S\triangle \sshifted}\cap B; \muStar,\SigmaStar
                }$:
            \begin{align*}
                \cN\inparen{
                    \inparen{S\triangle \sshifted}\cap B; \muStar,\SigmaStar
                }
                \leq 
                \int_{T\times (-\infty, \infty)^{d-1}} \frac{e^{-\norm{x}_2^2/2}}{\inparen{2\pi}^{\frac{d}{2}}} \d x
                \leq \frac{\inparen{2\pi}^{\frac{d-1}{2}}}{\inparen{2\pi}^{\frac{d}{2}}}\cdot \int_T e^{-x_1^2/2}\d x_1
                \leq \poly\inparen{\frac{\eps}{\alpha d}}\,.
            \end{align*}
        It remains to show that $\inparen{S\triangle \sshifted}\cap B\subseteq T\times (-\infty, \infty)^{d-1}$.
        To prove this, we first prove the following
        \begin{enumerate}
            \item The angle between the boundaries of $S$ and $\sshifted$, i.e., $\angle(w,\wStar)$, is $\poly(\eps/(\alpha d))$.
            \item With probability at least $1-\delta$, the intersection of the boundaries of $S$ and $\sshifted$ is at a distance at most ${d~\polylog\inparen{
                            \sfrac{d}{(\alpha\eps \delta)}
                        } }$ from the origin.
        \end{enumerate}
        Recall that after the transformation at the start of this proof, the normal vectors of $S$ and $\sshifted$ lie in ${\rm span}(e_1,e_2)$. %
        After projecting to ${\rm span}(e_1,e_2)$, $S\triangle \sshifted$ is the area between two lines with an angle of $\angle(w,\wStar)$ between them.
        Define 
        \[
            \text{$v$ as the point of intersection of $S\triangle \sshifted$ in the ${\rm span}(e_1,e_2)$\,.}
        \]
        Due to (2) above, $\norm{v}_2\leq {d~\polylog\inparen{
                            \sfrac{d}{(\alpha\eps \delta)}
                        } } \eqqcolon r$.
        Since $\norm{v}_2\leq r$ and the radius of $B$ is at most $r$, the maximum width of $\inparen{S\cap \sshifted}\cap B$ in the direction perpendicular to $\sfrac{(w+\wStar)}{\sqrt{2}}$ is at most 
        $\sin(\angle(w,\wStar))\cdot O(r)$, which by (1) and (2) above is at most $\poly\inparen{\sfrac{\eps}{(\alpha d)}}\cdot \polylog(1/\delta)$.

                    \medskip

                    \noindent \textit{Upper bounding $\angle(w,\wStar)$.}\indent 
                        Observe the following in the {\em untransformed} space,
                        \begin{align*}
                            \inangle{w, \wStar}
                                =\sum_i \wStar_i\inparen{w_i -\wStar_i + \wStar_i}
                                \in \norm{\wStar}_2 \pm \norm{\wStar}_1 \cdot \norm{w - \wStar}_\infty
                                \geq 1 - \sqrt{d}\cdot \poly(\eps/d)\,.
                        \end{align*}
                        Since $\cos(z)\leq 1-\inparen{\sfrac{z^2}{3}}$ for $-\pi/2\leq z\leq \pi/2$ and 
                        $\cos(\angle(w,\wStar))=\inangle{w, \wStar}$ (as both $w$ an $\wStar$ are unit norm), it follows that $\angle\inparen{w,\wStar}=\poly(\eps/d)$ in the original (untransformed) space.
                        Further since, $\norm{\SigmaStar}_2=\poly(1/\alpha)$ after the initial pre-processing, after applying the aforementioned transformation which converts $\muStar$ to 0 and $\SigmaStar$ to $I$, $\angle\inparen{w,\wStar}=\poly(\eps/(\alpha d)).$

                        \medskip 

                \begin{figure}[t!]
                    \centering
                    \includegraphics[clip,trim=3cm 2cm 6cm 0.75cm,width=0.5\linewidth]{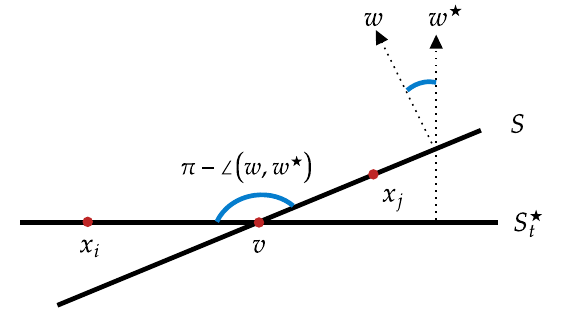}
                    \caption{
                        A construction appearing in the proof of \cref{lem:halfspaceLearner:symDiff}.
                        This is used to bound $\norm{v}_2$, where $v$ is the intersection of $S$ and $\sshifted$ in ${\rm span}(w,\wStar)$,  in terms of $\norm{x_i-x_j}_2$, $\norm{x_i}_2$, and $\norm{x_j}_2$.
                    }
                    \label{fig:halfspaceProof}
                \end{figure}
                
                \noindent \textit{Proof of (2).}\indent 
                Define $R$ as the maximum $L_2$ norm of any sample $x_1,\dots,x_n$.
                Recall that both $S$ and $\sshifted$ were selected to have the largest threshold subject to including all samples $x_1,\dots,x_n$ and, hence, the boundary of each $S$ and $\sshifted$ passes through (possibly) different points $x_i$ and $x_j$.
                Project $x_i$ and $x_j$ to ${\rm span}(e_1,e_2)$ and, with some abuse of notation, denote these projects by $x_i$ and $x_j$.
                    Straightforward geometry (see \cref{fig:halfspaceProof}) implies that 
                    \begin{align*}
                        \norm{x_i-x_j}_2^2 
                        &~~=~~
                        \norm{v - x_i}_2^2
                        +
                        \norm{v - x_j}_2^2
                        + 2\cos\inparen{\angle(w,\wStar)}\cdot \norm{v - x_i}_2\norm{v - x_j}_2\\
                        &~~\geq~~
                        \norm{v - x_i}_2^2\,.
                        \tag{as $\angle(w,\wStar) =\poly\inparen{\frac{\eps}{d}}$ and, hence, $\cos\inparen{\angle(w,\wStar)}\geq 0$}
                    \end{align*}
                    Since the norm of $\norm{x_i}_2\leq R$ and $\norm{x_j}_2\leq R$, 
                    \[
                        \norm{v} 
                        \leq \norm{x_i}_2+\norm{v-x_i}_2
                        \leq \norm{x_i}_2 + \norm{x_i-x_j}_2
                        \leq R+2R
                        \,.
                    \]
                    Finally, the $n$ samples $x_1,x_2,\dots,x_n$ from $\cN(0,I,\Sstar)$ can be seen as $O(n/\alpha)$ samples from $\cN(0,I)$ where we only keep those that fall inside $\Sstar$.
                    Well-known results about the concentration of Gaussian distributions imply that, with probability $1-\delta$, for each sample $x_i$, $\norm{x_i}_\infty\leq O(\log\inparen{nd/(\alpha\delta)})$.
                    Hence, $R\leq O(d\cdot \log\inparen{nd/(\alpha\delta)}).$
                    Substituting this into the bound on $\norm{v}$ implies $\norm{v}\leq O\inparen{d \log\inparen{\sfrac{nd}{\alpha\delta}}}$.
                    Since $n=\poly\inparen{(\sfrac{d}{(\alpha\eps)})\cdot \log(\sfrac{1}{\delta})}$
                    \[
                        \norm{v}\leq {d~\polylog\inparen{
                            \frac{d}{\alpha\eps \delta}
                        } }\,.
                    \]

\section{Sample Efficient Algorithms}\label{sec:sampleEfficiency}
    
    In this section, we summarize our sample complexity results.
    
    \begin{theorem}[
    ]\label{thm:sampleComplexity}
        Fix any $\theta^\star$ and $S^\star\in \hyS$ satisfying \cref{asmp:1:sufficientMass,asmp:1:polynomialStatistics,asmp:2} , any $\eps,\delta\in (0,1/4)$, and %

        {
        \[
            N_1 \coloneqq 
                \poly\inparen{
                    \frac{\Lambda k^k}{\eta\lambda\alpha^k}
                }\cdot 
                \wt{O}\inparen{
                \frac{
                   \vc{(\hyS)}}
                {\eps}
                +
                \frac{m}{\eps^2} \log^2\frac{1}{\delta}
            }
            \quadand
            N_2 \coloneqq \wt{O}\inparen{
                \frac{
                    \vc{(\hyS)} 
                    + m{\log^2\inparen{\sfrac{1}{\delta}}}
                }{
                    {\eps^{
                        \poly(\frac{k\Lambda}   {\eta\alpha\lambda})
                        }
                    }     
                }
            }\,.
        \]
        }
        The following hold:
        \begin{enumerate}
            \item[$\triangleright$] \textbf{(Exponential Optimization Calls)}~
            There is an algorithm that, given $N_1$ independent samples from $\cE(\theta^\star, S^\star)$, outputs an estimate $\wh{\theta}$, such that, with probability $1-\delta$,~ $\tv{\cE(\wh{\theta})}{~\cE(\theta^\star)} \leq \eps$;
            \item[$\triangleright$] \textbf{(Single Optimization Call)}~
            There is an algorithm that, given access to an ERM oracle\footnote{
                An ERM oracle to a hypothesis class $\hyS$ is a primitive that given labeled samples $\sinbrace{(x_i,y_i)\in \R^d\times \zo \colon 1\leq i\leq n}$ outputs $\argmin_{S\in \hyS}~\sum_i \mathds{I}\insquare{\mathds{1}_S(x_i)\neq y_i}$.
            }  for the class $\hyS$
            and $N_2$ independent samples from $\cE(\theta^\star, S^\star)$,
                makes one call to this oracle, does $\poly\inparen{md/\eps}$ computation,
            and outputs an estimate $\wh{\theta}$ such that, with probability $1-\delta$,~ $\tv{\cE(\wh{\theta})}{~\cE(\theta^\star)} \leq \eps$.
        \end{enumerate}
    \end{theorem}
    The first part of \cref{thm:sampleComplexity} bounds the sample complexity for learning $\cE(\theta^\star)$ under \cref{asmp:1:sufficientMass,asmp:1:polynomialStatistics,asmp:2}. 
    In the special case of Gaussian distributions, where $m=O(d^2)$, this matches the sample complexity obtained by \citet{Kontonis2019EfficientTS}.
    If $\vc{}(\hyH)=O(d)$ and $\delta\geq e^{-\eps}$, then this also matches the sample complexity of learning a Gaussian distribution without any truncation.
    This result, however, is based on a covering argument and is inherently inefficient: it requires solving $\Omega(m/\eps)^{m}$ optimization problems.
    
    The second part of \cref{thm:sampleComplexity} shows that if we are willing to allow a polynomial blowup in the sample complexity, then we can learn $\cE(\theta^\star)$ in by solving a \textit{single} optimization problem and doing $\poly(md/\eps)$ additional computation.
    This may enable practical implementations for classes $\hyS$ for which efficient heuristic optimization oracles are available.
    We note that, prior to our work, such a result was not known even in the special case of Gaussian distributions.

    The proof of the first part appears in \cref{sec:additionalProofs:sec:sampleEfficiency}.
    The proof of the second part is a corollary of our analysis to prove \cref{thm:main}.
    To see this, note that, due to \cref{thm:module:efficientLearning:denisReduction}, we can learn a set $S$ that is $O(\eps)$-close to $S^\star$ by solving an agnostic learning problem to $\eps^{1/\poly(\eta\alpha\lambda)}$ accuracy.
    Standard sample complexity bounds for agnostic learning imply that this can be done with 
    \[ \tilde{O}\inparen{\frac{\vc(\hyS) + \log\inparen{1/\delta}}{
        \eps^{\poly(\frac{k\Lambda}   {\eta\alpha\lambda})}
        }} \quadtext{samples.} \]
    Using this set $S$, we can use the PSGD framework in \cref{lem:psgd_complexity_result} to estimate $\theta^\star$ to $\eps$ accuracy in $\poly(md/\eps)$ time with
    \[
                \wt{O}_{\lambda,\Lambda,\alpha,k,\eta}\inparen{
                    \frac{m}{\eps^2}
                }
                \cdot \log^2\inparen{\frac{1}{\delta}}\quadtext{additional samples.}
        \]

\section{Pre-processing Algorithms to Construct \texorpdfstring{$\Theta$}{Θ} Satisfying \texorpdfstring{\cref{asmp:2}}{Assumption 3}}\label{sec:preprocess} %
    In this section, we focus on the families of Gaussian and exponential distributions.
    We present the definition of $\Theta$ from \cref{asmp:2} for both families and present pre-processing routines for ensuring that requirements in \cref{asmp:2} are satisfied.
    We note that both of these families are log-concave with polynomial sufficient statistics and have efficient samplers and, hence, satisfy \cref{asmp:1:polynomialStatistics}.

\subsection{Gaussian Distributions} %
\label{sec:gaussianPreprocess}
    \begin{algorithm}[b!]
    \caption{Pre-processing Routine for Gaussian Distributions}
        \label{alg:preprocess}
    \begin{algorithmic}
        \Require Independent samples $x_1,\dots,x_n$ from $\cN(\muStar,\SigmaStar, S^\star)$ and constant $\alpha>0$ from \cref{asmp:1:sufficientMass}

        \vspace{2mm}
        
        \State Calculate the sample mean $\wh{\mu}_{\theta^\star, S^\star}  \leftarrow \frac{1}{n} \sum_{i \in [n]} x_i$ and covariance $\wh{\Sigma}_{\theta^\star, S^\star}  \leftarrow \frac{1}{n}\sum_{i \in [n]} (x_i - \wh{\mu})(x_i - \wh{\mu})^\top$

        \State Define the linear transformation $L\colon\R^d\to\R^d$ as the map $z\mapsto \wh{\Sigma}_{\theta^\star, S^\star}^{-1/2}\inparen{z - \wh{\mu}_{\theta^\star, S^\star}}$

        \State \Return $L(\cdot)$ and a description of the set $\Theta$ in \cref{def:parameterSet} 
         
    \end{algorithmic}
    \end{algorithm}
    In this section, we give a pre-processing routine that satisfies \cref{asmp:2} when the underlying distribution family is Gaussian, the subroutine and construction of $\Theta$ are in \cref{alg:preprocess}.
    
     \paragraph{Preprocessing Routine.}
        For each $\theta$, let $\mu_\theta$ and $\Sigma_\theta$ be the mean and covariance corresponding to $\theta$.
        Let $\muStarTheta$ and $\sigmaStarTheta$ be $\mu_{\thetaStar}$ and $\Sigma_{\thetaStar}$, i.e., the parameters corresponding to $\theta^\star$.
        The pre-processing routine is simple:
        Given $\alpha > 0$ such that $\cE(\Sstar;\thetaStar) \geq \alpha$ and 
        \[
            n=\wt{O}\inparen{
                \frac{d^3}{\alpha^{10}}
                \cdot \log^2{\frac{1}{\delta}}
            }
        \]
        independent samples from $\cE(\theta^\star, S^\star)$, compute the sample mean $\wh{\mu}_{\theta^\star, S^\star}$ and the sample covariance $\wh{\Sigma}_{\theta^\star, S^\star}$.
        Then linearly transform the space to ensure that 
        \[
            \wh{\mu}_{\theta^\star, S^\star} = 0
            \quadand
            \wh{\Sigma}_{\theta^\star, S^\star} = I\,.
        \]
    
    \paragraph{Family of Parameter Sets.}
        As we will show, this pre-processing ensures that $(\muStarTheta,\sigmaStarTheta)$ are not too far from $(0, I)$ and enables us to construct a set $\Theta$ containing $\theta^\star$ while satisfying the requirements in \cref{asmp:2}.
        Formally, Part 2 of \cref{thm:resultsOfdaskalakis2018efficient} proves that, after this pre-processing, with probability $1-\delta$
        \[
            \norm{\muStarTheta}_2
            \leq O\inparen{\frac{1}{\alpha^2}\sqrt{\log\frac{1}{\alpha}}}\,,\quad 
            \sigmaStarTheta \preceq O\inparen{\frac{1}{\alpha^2}} \cdot I\,,\quadand
            \norm{I - \inparen{\sigmaStarTheta}^{-1}}\leq O\inparen{\log\frac{1}{\alpha}}\,.
            \yesnum\label{eq:gaussianPreprocess:propertiesofThetaStar}
        \]
       This motivates us to consider the following family of parameter sets.
       \begin{definition}[{Family of Parameter Sets}]\label{def:familyParameterSet}
            Given $r>0$, define 
            \[
                \Theta_r \coloneqq \inbrace{
                    \theta \coloneqq \inparen{
                            \sigmaTheta^{-1}\muTheta\,, \frac{1}{2}\sigmaTheta^{-1}
                        }
                        \in \R^d\times \mathscr{S}^d_+\colon\quad  
                    \norm{\sigmaTheta^{-1}\muTheta}_2\leq r\,,\quad
                    \norm{\sigmaTheta}_2\leq r\,,\quad 
                    \norm{I-\sigmaTheta^{-1}}_F\leq r
                }\,.
            \]
            Where $\mathscr{S}^d_+$ is the set of $d\times d$ positive semi-definite matrices.
        \end{definition}
        To see why this is a reasonable choice, observe that \cref{eq:gaussianPreprocess:propertiesofThetaStar} implies that for $r=O(\log(1/\alpha)/\alpha^2)$, $\theta^\star\in \Theta_r$ with high probability.
        Naturally, we define $\Theta$ as follows.
        \begin{definition}[{Parameter Set Satisfying \cref{asmp:2}}]\label{def:parameterSet}
            Fix $b=O(\log(1/\alpha)/\alpha^2)$.
            Define $\Theta \coloneqq \Theta_{b}$.
        \end{definition}
    We need to verify that $\Theta$ satisfies the following properties listed in \cref{asmp:2}.
    \begin{enumerate}
        \item \textbf{Convexity:} $\Theta$ is convex.
        \item \textbf{Bounds on covariance of sufficient statistic:} For all $\theta\in \Theta$, $\lambda I \preceq \cov_{\cN(\mu_\theta, \Sigma_\theta)}[t(x)]\preceq \Lambda I$.
        \item \textbf{Interiority of optimal:} $\theta^\star \in \Theta(\eta)$
        \item \textbf{Existence of a $\chi^2$-bridge:}
        For any $\theta_1,\theta_2\in \Theta$, there exists a $\theta$ such that 
        \[
            \chidiv{\cN(\mu_{\theta_1},\Sigma_{\theta_1})}{\cN(\mu_\theta,\Sigma_\theta)}\,,~~\chidiv{\cN(\mu_{\theta_2},\Sigma_{\theta_2})}{\cN(\mu_\theta,\Sigma_\theta)}~~\leq~~ e^{\poly\inparen{\frac{\Lambda}{\lambda}}}
            \,.
        \]
        \item \textbf{Projection oracles:}
            There are polynomial time projection oracles to 
            $\Theta$ and $O$ 
            where $O$ is a (fixed) convex subset of $\Theta(\eta)$.
        \item \textbf{{Feasibility of \cref{asmp:moment}:}} 
            For any vector $v\in \R^m$ such that $v\approx \E_{\cN(\mu_{\theta^\star}, \SigmaStar, S^\star)}[t(x)]$, there exists $\theta \in \Theta$, such that $\E_{\cN(\mu_\theta,\Sigma_\theta)}[t(x)] = v$. 
    \end{enumerate}

    \subsubsection*{Convexity of $\Theta$}
        Clearly, the constraints $\snorm{\sigmaTheta^{-1}\muTheta}_2\leq r$ and $\snorm{I-\sigmaTheta^{-1}}_F\leq r$ are convex in the the first and second components of $\theta$ respectively (namely, $\sigmaTheta^{-1}\muTheta$ and $(\sfrac{1}{2})\sigmaTheta^{-1}$).
        Finally, the constraint $\norm{\sigmaTheta}_2\leq r$ is an intersection of infinitely many constraints of the form $x^\top \sigmaTheta^{-1} x\geq r$ for $x\in \R^d$.
        Since each of these constraints is convex in the entries of $\sigmaTheta^{-1}$, their intersection is also convex in the entries of $\Theta$.

    \subsubsection*{Bounds on Covariance of Sufficient Statistics}
        Observe that for any $\theta\in \Theta_r$, $\snorm{I-\sigmaTheta^{-1}}_2 \leq \snorm{I-\sigmaTheta^{-1}}_F\leq r$ and, hence, $\sigmaTheta\succeq \inparen{\sfrac{1}{(1+r)}}\cdot I$.
        This combined with the fact that $\norm{\sigmaTheta}_2\leq r$, implies that 
        \begin{align*}
            \forall\ \theta\in \Theta_r\,,\quad 
            \frac{1}{1+r} \cdot I \preceq \Sigma_{\theta} \preceq r \cdot I\,.
        \end{align*}
        Combining this with the bound on $\sigmaTheta$ and $\muTheta$ in the definition of $\Theta$ implies
        \begin{align*}
            \Omega\inparen{\frac{\alpha^2}{\log(1/\alpha)}}
            \preceq 
            \Sigma_\theta 
            \preceq O\inparen{\frac{\log(1/\alpha)}{\alpha^2}}
            \quadand
            \norm{\mu_\theta} 
            \leq O\inparen{\frac{\log^2(1/\alpha)}{\alpha^4}}\,.
            \yesnum\label{eq:gaussianPreprocess:covarianceMeanBounds}
        \end{align*}
        The bounds on $\lambda$ and $\Lambda$ from \cref{asmp:2} follows from \cref{lem:gaussianPreprocess:boundsOnCovariangeForGaussian}.
        To state \cref{lem:gaussianPreprocess:boundsOnCovariangeForGaussian}, we need the following notation. 
        \begin{definition}\label{def:upperAndLowerBoundsonfourthEv}
            Given a $d\times d$ matrix $\Sigma\succ 0$ with eigenvalues $\lambda_1, \dots , \lambda_d$, define %
            \[
                \fourthevmin{\Sigma}\coloneqq \min\inbrace{\min_{1\leq i\leq d}\lambda_i, \min_{1\leq i,j\leq d} \lambda_i\lambda_j}
                \quadand
                \fourthevmax{\Sigma}\coloneqq \max\inbrace{\min_{1\leq i\leq d}\lambda_i, \max_{1\leq i,j\leq d} \lambda_i\lambda_j}\,.
            \]
            Further, given a vector $\mu\in \R^d$, define 
            \[
                \fourthevlower{\mu}{\Sigma}\coloneqq 
                \min\inbrace{
                    \frac{\fourthevmin{\Sigma}}{4}\,, 
                    \frac{\fourthevmin{\Sigma}}{16\norm{\mu}_2^2 + \sqrt{\fourthevmin{\Sigma}}}
                }   
                \quadand
                \fourthevupper{\mu}{\Sigma}\coloneqq 
                7\cdot \inparen{1 + \norm{\mu}_2^2} \cdot \fourthevmax{\Sigma}
                \,.
            \]
        \end{definition}
        \begin{lemma}\label{lem:gaussianPreprocess:boundsOnCovariangeForGaussian}
            For any $d\times d$ matrix $\Sigma\succ 0$ and vector $\mu\in \R^d$, the distribution $\cN(\mu,\Sigma)$ satisfies that 
            \[
                \fourthevlower{\mu}{\Sigma}\cdot I
                ~~\preceq~~
                \cov_{x\sim \cN(\mu,\Sigma)}\insquare{t(x)}
                ~~\preceq~~
                \fourthevupper{\mu}{\Sigma}\cdot I \,.
            \]
            where $t(z)=\insquare{\begin{matrix}
                z & -zz^\top
            \end{matrix}}$ is the sufficient statistic of the normal distribution.
        \end{lemma}
        Thus, due to \cref{eq:gaussianPreprocess:covarianceMeanBounds,lem:gaussianPreprocess:boundsOnCovariangeForGaussian}, we have the following corollary, which proves the desired bound on the covariance of sufficient statistics
        \begin{corollary}[\textbf{Covariance of Sufficient Statistics}]\label{lem:gaussianPreprocess:covarianceBoundsForParameterSet}
            Consider the $\Theta$ in \cref{def:parameterSet}.
            For all $\theta\in \Theta$ %
            \[
                \wt{\Omega}\inparen{
                    \alpha^6 %
                }
                \preceq
                    \cov_{x\sim \cN(\muTheta,\sigmaTheta)}\insquare{t(x)}
                \preceq
                \wt{O}\inparen{
                    \alpha^{-6}
                }\,.
                \tagnum{Covariance Bounds}{eq:gaussianPreprocess:covarianceBounds}
            \]
        \end{corollary}
        \begin{proof}[Proof of \cref{lem:gaussianPreprocess:boundsOnCovariangeForGaussian}]
            The lower bound is proved in Lemma 3 of \citet{daskalakis2018efficient}.
            Toward proving the upper bound define the following matrices.
            \begin{align*}
                Q(\mu,\Sigma)
                    &\coloneqq 
                    \frac{1}{2}\Sigma\otimes \Sigma
                    +\frac{1}{4}\inparen{
                        \sinparen{\mu\mu^\top}\otimes \Sigma +
                        \mu\otimes\Sigma\otimes\mu^\top + 
                        \mu^\top\otimes\Sigma\otimes\mu +
                        \Sigma \otimes \sinparen{\mu\mu^\top}
                    }\,,\\
                R(\mu,\Sigma)&\coloneqq
                    \frac{1}{2}\inparen{\mu\otimes\Sigma + \Sigma\otimes\mu}\,,\\
                B(\mu,\Sigma)&\coloneqq 
                    {
                        \sinparen{\mu\mu^\top}\otimes \Sigma +
                        \mu\otimes\Sigma\otimes\mu^\top + 
                        \mu^\top\otimes\Sigma\otimes\mu +
                        \Sigma \otimes \sinparen{\mu\mu^\top}
                    }\,.
            \end{align*}
            We will use the following fact that follows from Equation A.26 in \citet{daskalakis2018efficient} after observing that we parameterize Gaussians with sufficient statistics $\insquare{\begin{matrix}
                z & -zz^\top
            \end{matrix}}$ which is slightly different from the parameterization in \citet{daskalakis2018efficient}.
            \begin{fact}
                It holds that 
                \[
                    \cov_{x\sim \cN(\mu,\Sigma)}\insquare{t(x)} = 
                    \begin{bmatrix}
                        \Sigma & -R^
                        \top(\mu,\Sigma)\\
                        -R(\mu,\Sigma) & 
                        Q(\mu,\Sigma)
                    \end{bmatrix}\,.
                \]
            \end{fact}
            To upper bound $\cov_{x\sim \cN(\mu,\Sigma)}\insquare{t(x)}$ we show that the following matrix is positive semi-definite 
            \[
                \begin{bmatrix}
                    2\Sigma & 0\\
                    0 & Q(\mu,\Sigma) + B(\mu,\Sigma)
                \end{bmatrix}
                - 
                \begin{bmatrix}
                    \Sigma & -R^\top(\mu,\Sigma)\\
                    -R(\mu,\Sigma) & Q(\mu,\Sigma)
                \end{bmatrix}
                = 
                \begin{bmatrix}
                    \Sigma & R^\top(\mu,\Sigma)\\
                    R(\mu,\Sigma) & B(\mu,\Sigma)
                \end{bmatrix}
                \eqqcolon
                P\,, 
            \]
            and, then, upper bound the eigenvalues of the first block-matrix in the above expression for $P$.
                \begin{fact}\label{fact:PSD}
                    $P\succeq 0$.
                \end{fact}
                \begin{proof}[Proof of \cref{fact:PSD}]
                    To see that $P$ is positive semi-definite, recall that Schur's complement theorem for block-symmetric matrices shows that $P\succeq 0$ implied by the following 
                    \[
                        2\Sigma\succ 0
                        \quadand
                        B(\mu,\Sigma) - R(\mu,\Sigma) \cdot \Sigma^{-1}\cdot  R^\top(\mu,\Sigma)\succeq 0\,.
                    \]
                    The first condition holds by the assumption that $\Sigma\succ 0$ and the second condition is implied by
                    \begin{align*}
                        \sinparen{\mu\mu^\top}\otimes \Sigma +
                                \mu\otimes\Sigma\otimes\mu^\top + 
                                \mu^\top\otimes\Sigma\otimes\mu +
                                \Sigma \otimes \sinparen{\mu\mu^\top}
                        -
                        \inparen{
                            \mu\otimes\Sigma + \Sigma\otimes\mu
                        } 
                        \Sigma^{-1}
                        \inparen{\mu\otimes\Sigma + \Sigma\otimes\mu}^\top
                        \succeq 0\,.
                    \end{align*}
                    We can simplify the second term using $(a\otimes b)^\top = a^\top \otimes b^\top$, $(\mu\otimes \Sigma)\cdot \Sigma^{-1} = \mu \otimes I_{d\times d}$, and $(\Sigma \otimes \mu)\Sigma^{-1} = I_{d\times d}\otimes \mu$ to get 
                    \[
                        \inparen{
                            \mu\otimes\Sigma + \Sigma\otimes\mu
                        } 
                        \Sigma^{-1}
                        \inparen{\mu\otimes\Sigma + \Sigma\otimes\mu}^\top
                        = 
                        \inparen{
                            \mu\otimes I
                        }
                        \cdot \inparen{\mu^\top\otimes\Sigma + \Sigma\otimes\mu^\top}
                        +
                        \inparen{
                            I\otimes \mu
                        }
                        \cdot \inparen{\mu^\top\otimes\Sigma + \Sigma\otimes\mu^\top}
                        \,.
                    \]
                    Now, we can simplify the expression by using the following equalities
                    \begin{align*}
                        \mu\otimes \Sigma\otimes\mu^\top &= (\mu\otimes I)\cdot (\Sigma \otimes \mu^\top)\,,\qquad\qquad
                        \mu^\top\otimes \Sigma\otimes\mu = (I\otimes \mu)\cdot (\mu^\top \otimes \Sigma)\,,\\
                        (\mu\otimes I)\cdot(\mu^\top\otimes \Sigma) &= \inparen{\mu\mu^\top}\otimes \Sigma\,,\qquad\quad~~ 
                        (I\otimes \mu)\cdot \inparen{\Sigma\otimes\mu^\top} = \Sigma\otimes \inparen{\mu\mu^\top}\,.
                    \end{align*}
                    to get 
                    \begin{align*}
                        B(\mu,\Sigma) - R(\mu,\Sigma) \cdot \Sigma^{-1}\cdot  R^\top(\mu,\Sigma) = 0\succeq 0\,.
                    \end{align*} 
                \end{proof}
            It remains to show that 
            \[
                \begin{bmatrix}
                    2\Sigma & 0\\
                    0 & Q(\mu,\Sigma) + B(\mu,\Sigma)
                \end{bmatrix}
                \preceq \fourthevupper{\mu}{\Sigma}\cdot I
                \,.
            \]
            For this, it suffices to show that 
            \[
                Q(\mu,\Sigma) + B(\mu,\Sigma) \preceq \fourthevupper{\mu}{\Sigma}\cdot I
                \quadand
                2\Sigma \preceq \fourthevupper{\mu}{\Sigma}\cdot I\,.
            \]
            The latter holds as $\fourthevupper{\mu}{\Sigma}$ is larger than twice the largest eigenvalue of $\Sigma$ by definition (\cref{def:upperAndLowerBoundsonfourthEv}).
            To see the former, observe that 
            \[
               Q(\mu,\Sigma) + B(\mu,\Sigma)  
               = \frac{1}{2}\Sigma\otimes \Sigma
                + \frac{5}{4}\inparen{
                        \sinparen{\mu\mu^\top}\otimes \Sigma +
                        \mu\otimes\Sigma\otimes\mu^\top + 
                        \mu^\top\otimes\Sigma\otimes\mu +
                        \Sigma \otimes \sinparen{\mu\mu^\top}
                    }\,.
            \]
            Hence, to prove $Q(\mu,\Sigma) + B(\mu,\Sigma) \preceq \fourthevupper{\mu}{\Sigma}\cdot I$, it suffices to prove the following set of inequalities 
            \begin{align*}
                \Sigma\otimes \Sigma  
                    &\preceq \frac{2}{5}\cdot \fourthevupper{\mu}{\Sigma}\cdot I\,,\yesnum\label{eq:gaussianPreprocess:upperbound:1}\\
                \Sigma\otimes \sinparen{\mu\mu^\top}\,,\ \sinparen{\mu\mu^\top}\otimes \Sigma
                    & \preceq \frac{1}{5}\cdot \frac{4}{5}\cdot \fourthevupper{\mu}{\Sigma}\cdot I\,,\yesnum\label{eq:gaussianPreprocess:upperbound:2}\\
                \mu^\top \otimes \Sigma\otimes \mu + \mu \otimes \Sigma\otimes \mu^\top
                    & \preceq \frac{2}{5}\cdot \frac{4}{5}\cdot \fourthevupper{\mu}{\Sigma}\cdot I\,.\yesnum\label{eq:gaussianPreprocess:upperbound:3}
            \end{align*}
            Let $\lambda_1,\lambda_2,\dots,\lambda_d$ be the eigenvalues of $\Sigma$.
            \begin{enumerate}
                \item \cref{eq:gaussianPreprocess:upperbound:1} holds because the eigenvalues of $\Sigma\otimes \Sigma$ is at most $\max_{1\leq i,j\leq d}\lambda_i \lambda_j\leq \fourthevmax{\Sigma}$ and $\fourthevupper{\mu}{\sigma}\geq 5\cdot \fourthevmax{\Sigma}$ by definition (\cref{def:upperAndLowerBoundsonfourthEv}).
                \item \cref{eq:gaussianPreprocess:upperbound:2} holds because eigenvalues of $(\mu\mu^\top)\otimes \Sigma$ and $\Sigma\otimes (\mu\mu^\top)$ are at most $\norm{\mu}_2^2\cdot \max_{1\leq i\leq d} \lambda_i\leq \norm{\mu}_2^2\cdot \fourthevmax{\Sigma}$ and $\fourthevupper{\mu}{\sigma}\geq 7\cdot \fourthevmax{\Sigma}$ by definition (\cref{def:upperAndLowerBoundsonfourthEv}).
                \item To prove \cref{eq:gaussianPreprocess:upperbound:3}, observe that
                \[
                    \mu^\top \otimes \Sigma\otimes \mu + \mu \otimes \Sigma\otimes \mu^\top 
                    = (\mu\otimes I)\cdot (\Sigma \otimes \mu^\top) + (I\otimes \mu)\cdot (\mu^\top \otimes \Sigma)\,,
                \]
                To upper bound eigenvalues of this, consider any unit vector $v\in \R^{d^2}$.
                First, observe that 
                \[
                    \norm{v^\top (\mu\otimes I)}_2^2
                    = v^\top  (\mu\otimes I) (\mu^\top\otimes I) v
                    = v^\top  (\mu\mu^\top\otimes I) v
                    \leq \norm{\mu}_2^2\,.
                \]
                Next, notice that
                \[
                    \norm{(\Sigma \otimes \mu^\top) v}_2^2
                    = v^\top  (\Sigma \otimes \mu) (\Sigma \otimes \mu^\top) v
                    = v^\top  (\Sigma^2\otimes \mu\mu^\top) v
                    \leq \inparen{\max_{1\leq i\leq d}\lambda_i^2} \cdot \norm{\mu}_2^2\,.
                \]
                Therefore, by Cauchy-Schwarz inequality and the fact that $\lambda_1,\dots,\lambda_d\geq 0$, it follows that 
                \[
                    v^\top (\mu\otimes I)\cdot (\Sigma \otimes \mu^\top) v
                    \leq \norm{\mu}_2^2 \cdot \max_{1\leq i\leq d}\lambda_i\,.
                \]
                Similarly, we have that 
                \begin{align*}
                    \norm{v^\top (I\otimes \mu)}_2^2
                    &= v^\top  (I\otimes \mu) (I\otimes \mu^\top) v
                    = v^\top  (I\otimes \mu\mu^\top) v
                    \leq \norm{\mu}_2^2\,,\\
                    \norm{(\mu^\top\otimes \Sigma) v}_2^2
                    &= v^\top  (\mu\otimes \Sigma) (\mu^\top\otimes \Sigma) v
                    = v^\top  (\mu\mu^\top\otimes \Sigma^2) v
                    \leq \inparen{\max_{1\leq i\leq d}\lambda_i^2} \cdot \norm{\mu}_2^2\,,
                \end{align*}
                and, hence, by Cauchy-Schwarz and the fact that $\lambda_1,\dots,\lambda_d\geq 0$, it follows that 
                \[
                    v^\top (I\otimes \mu)\cdot (\mu^\top\otimes \Sigma) v
                    \leq \norm{\mu}_2^2 \cdot \max_{1\leq i\leq d}\lambda_i\,.
                \]
                Therefore, we can conclude that 
                \[
                    \mu^\top \otimes \Sigma\otimes \mu + \mu \otimes \Sigma\otimes \mu^\top \preceq
                    2\cdot \norm{\mu}_2^2 \cdot \max_{1\leq i\leq d}\lambda_i\cdot I\,.
                \]
                \cref{eq:gaussianPreprocess:upperbound:3} follows because $\fourthevupper{\mu}{\Sigma}\geq 7\cdot \norm{\mu}_2^2 \cdot \fourthevmax{\Sigma}$ and $\fourthevmax{\Sigma} \geq \max_{1\leq i\leq d}\lambda_i$ by definition (\cref{def:upperAndLowerBoundsonfourthEv}).
            \end{enumerate}

        \end{proof}

    \subsubsection*{Interiority of $\theta^\star$}
        We divide this proof into two parts: \cref{lem:gaussianPreprocessing:OptInInterior,lem:gaussianPreprocessing:subsetOfRelativeInterior}.
        \begin{lemma}\label{lem:gaussianPreprocessing:OptInInterior}
            {Fix $\delta > 0$. Given $n = \tilde{O}\sinparen{\sinparen{\sfrac{d^3}{\alpha^{10}}}\log^2\inparen{\sfrac{1}{\delta}}}$ samples, apply the preprocessing described in \cref{alg:preprocess}. After this transformation, }
            with probability $1-\delta$, $\theta^\star\in \Theta_{b/2}$.
        \end{lemma}
        \begin{proof}
            Part 3 of \cref{thm:resultsOfdaskalakis2018efficient} implies that 
            \[
                \norm{\sigmaStarThetaExplicit}_2 \leq O\inparen{\log\frac{1}{\alpha}} \leq \frac{b}{2}
                \quadand
                \norm{1 - \inparen{\sigmaStarThetaExplicit}^{-1}}_F
                \leq O\inparen{\frac{1}{\alpha^2}\log{\frac{1}{\alpha}}}
                \leq \frac{b}{2}\,.
            \] 
            To upper bound $\norm{\sigmaStarThetaExplicit^{-1}\muStarThetaExplicit}_2$, we use two results from \citet{daskalakis2018efficient}:
                first, Part 2 of \cref{thm:resultsOfdaskalakis2018efficient} 
                implies that $\norm{\muStarThetaExplicit}_{\sigmaStarThetaExplicit}\leq O(\sqrt{\log(1/\alpha)})$ and, second, Part 3 of \cref{thm:resultsOfdaskalakis2018efficient} implies that the maximum eigenvalue of $\sigmaStarThetaExplicit$ $\lambda_{\max}(\sigmaStarThetaExplicit)$ is at most $O(\alpha^{-2})$.
            Combining the two we have $\norm{\sigmaStarThetaExplicit^{-1}\muStarThetaExplicit}_{2}\leq \norm{\muStarThetaExplicit}_{\sigmaStarThetaExplicit}\cdot \sqrt{\lambda_{\max}(\sigmaStarThetaExplicit)}\leq O(\sqrt{\log(1/\alpha)})\cdot O(\alpha^{-1})\leq b/2.$
            These imply the lemma.
        \end{proof}
        \begin{lemma}\label{lem:gaussianPreprocessing:subsetOfRelativeInterior}
            There exists $\eta\geq \Omega(\alpha^3)$ such that 
            $\Theta_{b/2} \subseteq \Theta(\eta)$.
        \end{lemma}
        \begin{proof} 
        Fix any $\theta\in \Theta_{b/2}$.
        To prove this, we need to show that with probability $1-\delta$ 
        \[
            \text{for any}\qquad
            \text{$[u, T]\in \R^d\times \mathscr{S}^d_{+}$}\,,\qquad 
            \text{such that $\norm{u}_2 + \norm{T}_F\leq \eta$}\,,\qquad
            \theta + \inparen{u, T}\in \Theta\,.
        \]
        (Where $\mathscr{S}^d_{+}$ is the set of $d\times d$ positive semi-definite matrices.)
        To see this, fix any $[u, T]$ where $u\in \R^d$ and $T$ is a $d\times d$ real symmetric matrix, such that, $\norm{u}_2 + \norm{T}_F\leq \eta$.
        To show that $\theta + \inparen{u, T}\in \Theta$, we need to verify the following (see \cref{def:parameterSet}).
        \begin{align*}
            \text{$\sigmaTheta + T$ } &\text{ is positive semi-definite}\,,\yesnum\label{eq:gaussianPreprocess:interiority:1}\\
            \norm{\sigmaTheta + T}_2&\leq b\,,\yesnum\label{eq:gaussianPreprocess:interiority:2}\\
            \norm{\sigmaTheta^{-1}\muTheta + u}_2&\leq b\,,\yesnum\label{eq:gaussianPreprocess:interiority:3}\\
            \norm{I - \inparen{\sigmaTheta + T}^{-1}}_F&\leq b\,.\yesnum\label{eq:gaussianPreprocess:interiority:4}
        \end{align*}
        where $b=O(\log(1/\alpha)/\alpha^2)$ is from \cref{def:parameterSet}.
        Among these requirements, the first one would turn out to be the bottleneck on $\eta.$
        \cref{lem:gaussianPreprocess:covarianceBoundsForParameterSet}  implies \cref{eq:gaussianPreprocess:interiority:1,eq:gaussianPreprocess:interiority:2,eq:gaussianPreprocess:interiority:3,eq:gaussianPreprocess:interiority:4} as follows. %
        \begin{enumerate}
            \item \textit{Proof of \cref{eq:gaussianPreprocess:interiority:1,eq:gaussianPreprocess:interiority:2,eq:gaussianPreprocess:interiority:3}:} 
            Recall that $\norm{T}_2\leq \norm{T}_F\leq \eta$.
            \cref{eq:gaussianPreprocess:interiority:1,eq:gaussianPreprocess:interiority:2,eq:gaussianPreprocess:interiority:3} hold as
            \begin{align*}
                \norm{\sigmaTheta + T}_2
                &\quad\geq\quad \norm{\sigmaTheta}_2 - \norm{T}_2
                ~~\qquad\Stackrel{\eqref{eq:gaussianPreprocess:covarianceMeanBounds}}{\geq}\quad 
                   \wt{\Omega}(\alpha^2) - \eta
                > 0\,,\\
                \norm{\sigmaTheta + T}_2
                &\quad\leq\quad \norm{\sigmaTheta}_2 + \norm{T}_2
                ~\qquad\Stackrel{\theta\in \Theta_{b/2}}{\leq}\quad 
                    \frac{b}{2}+\eta
                \leq b\,,\\
                \norm{\sigmaTheta^{-1}\muTheta + u}_2
                &\quad\leq\quad \norm{\sigmaTheta^{-1}\muTheta}_2+\norm{u}_2
                \quad\Stackrel{\theta\in \Theta_{b/2}}{\leq}\quad 
                    \frac{b}{2}+\eta
                \leq b\,.
            \end{align*} 
            \item \textit{Proof of \cref{eq:gaussianPreprocess:interiority:4}:}
            We claim the following 
            \[
                \inparen{\sigmaTheta + T}^{-1}
                = 
                \sigmaTheta^{-1}
                + C\,,\quadtext{where}
                \norm{C}_F \leq \wt{O}\inparen{\frac{1}{\alpha}}\,.
                \yesnum\label{eq:gaussianPreprocess:frobeniousBound}
            \]
            This implies \cref{eq:gaussianPreprocess:interiority:4} as 
            \[
                \norm{I - \inparen{\sigmaTheta + T}^{-1}}_F
                = \norm{I - \sigmaTheta^{-1}
                    - C
                }_F
                \leq \norm{I - \sigmaTheta^{-1}}_F + \norm{C}_F
                \quad\Stackrel{\theta\in \Theta_{b/2}}{\leq}\quad \frac{b}{2} + \wt{O}\inparen{\frac{1}{\alpha}}
                \quad\Stackrel{b=\Omega(\alpha^2)}{\leq}\quad  b\,.
            \]
            It remains to prove the claim in \cref{eq:gaussianPreprocess:frobeniousBound}.
            Toward this observe that 
            \[
                \inparen{\sigmaTheta + T}^{-1}
                - \sigmaTheta^{-1}
                = 
                \sigmaTheta^{-1}\cdot 
                \sum_{i=0}^\infty \inparen{
                    -\sigmaTheta^{-1} T
                }^i - \sigmaTheta^{-1}
                = \sigmaTheta^{-1}\cdot \sum_{i=1}^\infty   \inparen{
                    -\sigmaTheta^{-1} T
                }^i
                \eqqcolon C\,.
            \]
            Further, we have the following upper bound on $\norm{C}_F$
            \[
                \norm{C}_F\leq 
                \norm{\sigmaTheta^{-1}}_2\cdot 
                \norm{\sum_{i=1}^\infty  \inparen{
                    -\sigmaTheta^{-1} T
                }^i}_F
                \leq \norm{\sigmaTheta^{-1}}_2\cdot 
                \sum_{i=1}^\infty  \norm{\inparen{
                    -\sigmaTheta^{-1} T
                }^i}_F
                \leq \norm{\sigmaTheta^{-1}}_2\cdot 
                \sum_{i=1}^\infty  \norm{{
                    \sigmaTheta^{-1} T
                }}_F^i\,.
            \]
            Since $\theta\in \Theta_{b/2}$, $\norm{\sigmaTheta^{-1}}_2\leq b/2$ and, hence,
            \[
                \norm{\sigmaTheta^{-1}}_2\cdot 
                \sum_{i=1}^\infty  \norm{{
                    \sigmaTheta^{-1} T
                }}_F^i
                \leq {\frac{b^2}{4}}\cdot \sum_{i=1}^\infty  \norm{{T}}_F^i
                \leq {\frac{b^2}{4}} \cdot \sum_{i=1}^\infty  \eta^i
                ~~~\qquad\Stackrel{\eta\leq \min\inbrace{\frac{1}{2}\,, O(\alpha^3)}}{\leq}\qquad~~~
                    {\frac{b^2}{4}}
                    \cdot O(\alpha^3)
                ~~\quad\Stackrel{b=\wt{O}(\alpha^{-2})}{=}\quad~~
                    \wt{O}\inparen{\frac{1}{\alpha}}\,.
            \]
            Therefore the claim follows.
        \end{enumerate} 
        \end{proof}
    \subsubsection*{Existence of a $\chi^2$-Bridge}
        We will use the following lemma to prove the existence of a $\chi^2$-bridge.
        \begin{lemma}\label{lem:GaussianRenyiBridge}
            For any pair of Gaussians $\cN_1\coloneqq \cN(\mu_1,\Sigma_1)$ and $\cN_2\coloneqq \cN(\mu_2,\Sigma_2)$, there exists another Gaussian $\cN\coloneqq \cN(\mu,\Sigma)$ such that 
            \[
                \renyi{3}{\cN_1}{\cN}\,,\ \renyi{3}{\cN_2}{\cN}
                ~~\leq~~
                {
                    \frac{3}{2}\norm{\Sigma_1^{-1/2}\inparen{\mu_1-\mu_2}}_2^2 
                    + \frac{3}{4}\max\inbrace{1, \norm{\Sigma}_2}
                    \norm{\Sigma_1^{-1/2}\Sigma_2\Sigma_1^{-1/2} - I}_F^2
                }
                \,.
            \]
        \end{lemma}
        To see why this result is useful, consider any $\cN_1(\mu_1,\Sigma_1)$ and $\cN_2(\mu_1,\Sigma_1)$ corresponding to some parameters $\theta_1,\theta_2\in \Theta$.
        We claim that distribution $\cN(\mu,\Sigma)$ in \cref{lem:GaussianRenyiBridge} is a suitable bridge.
        To see this, observe that due to the bounds on $\Sigma_1,\Sigma_2,\mu_1,\mu_2$ from \cref{eq:gaussianPreprocess:covarianceMeanBounds}:
        \begin{align*}
            \renyi{3}{\cN_1}{\cN}\,,\ \renyi{3}{\cN_2}{\cN}
                &\leq
                {
                    3\norm{\Sigma_1^{-1/2}\inparen{\mu_1-\mu_2}}_2^2 
                    + \frac{3}{2}\max\inbrace{1, \norm{\Sigma}_2}
                    \norm{\Sigma_1^{-1/2}\Sigma_2\Sigma_1^{-1/2} - I}_F^2
                }\\
                &\leq
                {
                    O(\alpha^{-6})\cdot \norm{{\mu_1-\mu_2}}_2^2 
                    + O(\alpha^{-3})\cdot 
                    \norm{\Sigma_1^{-1/2}\Sigma_2\Sigma_1^{-1/2} - I}_F^2
                }\\
                &\leq
                {
                    O(\alpha^{-8})
                    + O(\alpha^{-3})\cdot 
                    \norm{\Sigma_1^{-1/2}\Sigma_2\Sigma_1^{-1/2} - I}_F^2
                }\,.
                \yesnum\label{eq:gaussianPreprocessing:bridge:upperbound}
        \end{align*}
        To further upper bound $\snorm{\Sigma_1^{-1/2}\Sigma_2\Sigma_1^{-1/2} - I}_F^2$, observe that 
        \begin{align*}
            \norm{\Sigma_1^{-1/2}\Sigma_2\Sigma_1^{-1/2} - I}_F^2
            \leq \norm{
                \Sigma_1^{-1/2}\inparen{
                \Sigma_2 - \Sigma_1}\Sigma_1^{-1/2}
            }_F^2
            \leq \norm{\Sigma_1^{-1}}_2^2
            \norm{\Sigma_2 - \Sigma_1}_F^2\,.
        \end{align*}
        The triangle inequality and $\snorm{\Sigma_1^{-1}}_2\leq \snorm{\Sigma_1^{-1}}_F\leq b$ implies that 
        \begin{align*}
            \norm{\Sigma_1^{-1/2}\Sigma_2\Sigma_1^{-1/2} - I}_F^2
            \leq b^2
            \sum_{i\in \inbrace{1,2}}\norm{\Sigma_i - I}_F^2
            \leq b^2
            \sum_{i\in \inbrace{1,2}}\norm{\Sigma_i}_2^2 
            \norm{ I - \Sigma_i^{-1}}_F^2\,.
        \end{align*}
        This is at most $2b^6$ as $\snorm{\Sigma_i^{-1}}_2\leq \snorm{\Sigma_i^{-1}}_F\leq b$ and $\snorm{ I - \Sigma_i^{-1}}_F^2\leq b^2$.
        Substituting this in \cref{eq:gaussianPreprocessing:bridge:upperbound} implies that 
        $\renyi{3}{\cN_1}{\cN},\ \renyi{3}{\cN_2}{\cN}\leq {
                    O(1/\alpha^{16})
                }$
        and, since, $\cR_3$-divergence dominates the $\cR_2$-divergence, it follows that $\renyi{2}{\cN_1}{\cN}, \renyi{2}{\cN_2}{\cN}\leq {
                    O(1/\alpha^{16})
                }.$
        Finally, the standard relation between the $\chi^2$-divergence and the $\cR_2$-divergence (\cref{sec:preliminaries}) implies that 
        \[
            \chidiv{\cN_1}{\cN}\,,\ \chidiv{\cN_2}{\cN}\leq 
            e^{
                    O(1/\alpha^{16})
                }
            -1 \,.
        \]
        This is consistent with \cref{lem:gaussianPreprocess:covarianceBoundsForParameterSet} since $O(1/\alpha^{16})$ is at most $O\inparen{{\sfrac{\Lambda^2}{\lambda^2}}}$ for $\lambda, \sfrac{1}{\Lambda} = \wt{O}(\alpha^6)$.
        In the rest of this section, we prove \cref{lem:GaussianRenyiBridge}.
        \begin{proof}[Proof of \cref{lem:GaussianRenyiBridge}]
            This proof is inspired by the proof Lemma 14 in \citet{Kontonis2019EfficientTS}, which proves an analogous result for $\cR_2$.
            \citet{Kontonis2019EfficientTS}'s result for $\cR_2$ is sufficient to establish the existence of a $\chi^2$-bridge.
            We nevertheless prove this stronger version in the hope that it will be useful in the future.

            Fix any $1\leq q\leq 3$.
            Consider any $\cN_1\coloneqq \cN(\mu_1,\Sigma_1)$ and $\cN_2\coloneqq \cN(\mu_2,\Sigma_2)$.
            Consider the following sequence of affine transformations.
            First, apply an affine transformation to ensure that the densities have the form: $\cN(0, I)$ and $\cN(\mu',\Sigma')$ respectively.
            Next, apply a unitary transformation to convert the second density to $\cN(\mu,\Lambda)$ for some $\mu\in \R^d$ and some diagonal matrix $\Lambda={\rm diag}(\lambda_1,\dots,\lambda_d)$.
            Observe that the first density is invariant to any unitary transformation.
            Hence, after these affine transformations, we can transform the densities to $\cN(0, I)$ and $\cN(\mu,\Lambda)$.
            Therefore, as $\cR_q$ is invariant to affine transformations, it suffices to consider the case where $\cN_1=\cN(0, I)$ and $\cN_2=\cN(\mu,\Lambda)$.
            
            Now, we define the bridge distribution $\cN=\cN(0,B)$ where $B={\rm diag}(b_1,\dots,b_d)$ and $b_i=\max\inbrace{1, \lambda_i}$ for each $1\leq i\leq d$.
            Recall that for two densities $\cP$ and $\cQ$
            \[
                \renyi{q}{\cP}{\cQ}\coloneqq \frac{1}{q-1}\ln{
                    \Ex_{x\sim \cP}\inparen{\frac{\d \cP(x)}{\d \cQ(x)}}^q
                }\,.
            \]
            We begin by bounding $\Ex_{x\sim \cN_2}\inparen{\sfrac{\d \cN_2}{\d \cN}}^q$:
            \begin{align*}
                \Ex_{x\sim \cN_2}\inparen{\frac{\d \cN_2(x)}{\d \cN(x)}}^q
                &= \int \frac{\cN(x; \mu,\Lambda)^q}{\cN(x;0,B)^{q-1}} \d x 
                = \frac{
                    \abs{B}^{\sfrac{(q-1)}{2}}
                }{
                    \abs{\Lambda}^{\sfrac{q}{2}}
                }
                \cdot \frac{
                    e^{-q\cdot \mu^\top \Lambda^{-1}\mu/2}
                }{
                    \inparen{2\pi}^{\sfrac{d}{2}}
                }
                \cdot \int e^{
                    \frac{1}{2}\cdot x^\top \inparen{(q-1)B^{-1} - q\Lambda^{-1}} x
                    + qx^\top \Lambda^{-1} \mu
                } \d x
                \,.
            \end{align*}
            Let the integral in the RHS be $\mathcal{I}$, we can compute it as follows 
            \begin{align*}
                \mathcal{I}
                ~~&=~~
                    \prod_{i=1}^d~~ \int \exp\inparen{
                        x_i^2 \inparen{
                            \frac{q-1}{2b_i} - \frac{q}{2\lambda_i}
                        }
                        + q\cdot \frac{\mu_i x_i}{\lambda_i}
                    }\d x\\
                &=~~
                    \prod_{i=1}^d~~ 
                        \sqrt{
                            \frac{
                                2\pi 
                            }{
                                \frac{q}{\lambda_i} - \frac{q-1}{b_i}
                            }
                        }
                        \cdot
                    \exp\inparen{
                        \frac{q^2}{2}\cdot 
                            \frac{
                                b_i\mu_i^2
                            }{
                                qb_i \lambda_i - (q-1)\lambda_i^2 
                            }
                    }\,.
            \end{align*}
            Substituting the expression of $\mathcal{I}$ in the previous equality implies that 
            \begin{align*}
                \Ex_{x\sim \cN_2}\inparen{\frac{\d \cN_2(x)}{\d \cN(x)}}^q
                &=
                \prod_{i=1}^d~~\sqrt{
                    \frac{
                        b_i^{q-1}
                    }{
                        q \lambda_i^{q-1} - (q-1)\sfrac{\lambda_i^2}{b_i}
                    }
                }
                \cdot 
                \exp\inparen{
                    \frac{q(q-1)}{2} \cdot \frac{\mu_i^2}{qb_i -(q-1)\lambda_i}
                }\\
                &=
                \exp\inparen{
                    \sum_{i=1}^d~
                    \frac{1}{2} \ln\inparen{
                        \frac{
                            b_i^{q-1}
                        }{
                            q \lambda_i^{q-1} - (q-1)\sfrac{\lambda_i^q}{b_i}
                        }
                    }
                    +
                    \frac{q(q-1)}{2} \frac{\mu_i^2}{qb_i -(q-1)\lambda_i}
                }
                \,.
                \yesnum\label{eq:gaussianPreprosessing:bridge:firstBound}
            \end{align*}
            Next, we will upper bound each term in the exponent.
            \begin{align*}
                \sum_{i=1}^d~
                \ln{
                        \frac{
                            b_i^{q-1}
                        }{
                            q \lambda_i^{q-1} - (q-1)\sfrac{\lambda_i^q}{b_i}
                        }
                    }
                = 
                \sum_{i\colon \lambda_i < 1}~
                \ln{
                        \frac{
                            1
                        }{
                            \lambda_i^{q-1}\inparen{
                                q - (q-1)\lambda_i
                            }
                        }
                    }
                \leq 
                q\cdot \sum_{i\colon \lambda_i < 1}\inparen{\frac{1}{\lambda_i} - 1}^2
                \leq q\cdot \norm{\Lambda^{-1} - I}_F^2
                \,,
            \end{align*}
            where in the first inequality we use the fact that $\ln\inparen{{\frac{1}{x^{q-1}(q-(q-1)x)}}}\leq q\inparen{1-(1/x)}^2$ for all $x\in (0,1)$ and $1\leq q\leq 3$.
            To bound the second term in the expression of $\Ex_{x\sim \cP}\inparen{\sfrac{\d \cN_2}{\d \cN}}^q$ observe that 
            \begin{align*}
                \sum_{i=1}^d~\frac{\mu_i^2}{qb_i -(q-1)\lambda_i}
                =
                \sum_{i\colon \lambda_i <1}
                    \frac{\mu_i^2}{q -(q-1)\lambda_i}
                +
                \sum_{i\colon \lambda_i \geq 1}
                    \frac{\mu_i^2}{\lambda_i}
                \leq 
                \sum_{i=1}^d~
                    \frac{\mu_i^2}{\lambda_i}
                = \norm{\Lambda^{-1/2}\mu}_2^2\,,
                \yesnum\label{eq:gaussianPreprosessing:bridge:N2bound}
            \end{align*}
            where in the inequality we use the fact that $\sfrac{1}{(q-(q-1)x)}\leq \sfrac{1}{x}$ for all $x\in (0,1)$ and $q\geq 0$.
            Hence, it follows that 
            \[
                \Ex_{x\sim \cN_2}\inparen{\frac{\d \cN_2(x)}{\d \cN(x)}}^q
                \leq 
                \exp\inparen{
                    \frac{q}{2}\cdot \norm{\Lambda^{-1} - I}_F^2
                    + \frac{q(q-1)}{2}\cdot \norm{\Lambda^{-1/2}\mu}_2^2
                }\,.
            \]
            Substituting $\Lambda=I$ and $\mu=0$ in \cref{eq:gaussianPreprosessing:bridge:firstBound}, implies that 
            \begin{align*}
                \Ex_{x\sim \cN_1}\inparen{\frac{\d \cN_1(x)}{\d \cN(x)}}^q
                ~~
                &=~~
                    \exp\inparen{
                        \sum_{i=1}^d~~
                        \frac{1}{2}\cdot \ln\inparen{
                            \frac{
                                b_i^{q}
                            }{
                                qb_i  - (q-1)
                            }
                        }
                    }\,.
            \end{align*}
            We can upper bound this using the following 
            \[
                \sum_{i=1}^d
                        \ln{
                            \frac{
                                b_i^{q}
                            }{
                                qb_i  - (q-1)
                            }
                        }
                = 
                \sum_{i\colon \lambda_i \geq 1}
                        \ln{
                            \frac{
                                \lambda_i^{q}
                            }{
                                q\lambda_i  - (q-1)
                            }
                    }
                \leq 
                \sum_{i\colon \lambda_i \geq 1}
                        {
                            q\lambda_i \inparen{1-\frac{1}{\lambda_i}}^2
                        }
                \leq 
                q\max\inparen{\norm{\Lambda}_2, 1}\norm{\Lambda^{-1} - I}_F^2\,,
                \yesnum\label{eq:gaussianPreprosessing:bridge:N1bound}
            \]
            where we use the inequality that $\ln\frac{x^q}{qx-(q-1)}\leq x\inparen{1-\inparen{\sfrac{1}{x}}}^2$ for all $x\geq 1$ and $0\leq q\leq 3$.
            Hence, it holds that 
            \[
                \Ex_{x\sim \cN_1}\inparen{\frac{\d \cN_1(x)}{\d \cN(x)}}^q
                \leq 
                \exp\inparen{
                    \frac{q}{2}\cdot 
                    \max\inparen{\norm{\Lambda}_2, 1}\cdot 
                    \norm{\Lambda^{-1} - I}_F^2
                }\,.
            \]
            Substituting \cref{eq:gaussianPreprosessing:bridge:N2bound,eq:gaussianPreprosessing:bridge:N1bound} into the definition of $\cR_q$ and setting $q=3$ implies the result.
        \end{proof}
        
    \subsubsection*{Projection Oracles}
        Lemma 8 in \citet{daskalakis2018efficient} gives $\poly\inparen{d/\eps,\log\sinparen{r}}$ time projection oracle to $\Theta_r$ for any $r>0$ and any accuracy $\eps>0$.
        For $r=b$, this gives a projection oracle to $\Theta$.
        Further, as $\Theta_{b/2}\subseteq \Theta(\eta)$ and $\Theta_{b/2}$ is convex, it also implies a projection oracle to a (fixed) convex subset of $\Theta(\eta)$.

        \subsubsection*{Feasibility of \cref{asmp:moment}} %
            The following lemma gives the algorithm to satisfy \cref{asmp:moment}. %

        \begin{lemma}\label{lem:gaussianpreprocess:moment}
            Fix any $\eps > 0$.
            Let $t(z)=\begin{bmatrix}
                z & -zz^\top
            \end{bmatrix}$ be the sufficient statistic of the Gaussian family.
            For each vector $v\in \R^{d+d^2}$, let $\mu_v\in\R^d$ be the first $d$-components of $v$ and $\Sigma_v\in \R^{d\times d}$ be $-v_{d+1\colon d^2+d}-\mu_v\mu_v^\top$ where $v_{d+1\colon d^2+d}$ is the last $d^2$ components of $v$ arranged as a matrix.
            Given any vector $v$ satisfying $\norm{v - \E_{\cN(\mu^\star, \Sigma^\star,\Sstar)}[t(x)]}_2 < \eps$, 
            it holds that $\Ex_{\cN(\mu_v,\Sigma_v)}[t(x)]=v$. %
            Hence, for any $v$ satisfying the condition in \cref{asmp:moment}, $\theta=(\Sigma_v^{-1}\mu_v, (\sfrac{1}{2})\Sigma_v^{-1})$ satisfies the property promised in \cref{asmp:moment}.
        \end{lemma}
        \begin{proof}
            From the definition of the sufficient statistic of the Gaussian distribution and the definition of $\mu_v$ and $\Sigma_v$, it follows that $\Ex_{\cN(\mu_v,\Sigma_v)}[t(x)]=v$.
        \end{proof}

\subsection{Product Exponential Distributions}\label{sec:exponentialPreprocess}%
In this section, we give a pre-processing routine that satisfies \cref{asmp:2} when the underlying distribution family is a product exponential distribution. 
Each distribution in the product exponential family is supported on $[0, \infty)^d$ and parameterized by $\theta\in (-\infty, 0)^d$.
For each parameter $\theta\in (-\infty, 0)^d$, the corresponding distribution has the following density
\[ 
    \textsf{Exp}(x; \theta) \propto e^{\inangle{\theta, x}}\,.
\]
As a side note, observe that the log-density of this distribution is linear and the distribution admits efficient sampling oracles (\cref{rem:sampling}) and, hence, it satisfies \cref{asmp:1:polynomialStatistics}.
Now, we proceed with describing the pre-processing routine required to construct $\Theta$ that satisfies \cref{asmp:2}.
\begin{algorithm}[ht!]
    \caption{Pre-processing Routine for Exponential Distributions}
        \label{alg:preprocess:exponential}
    \begin{algorithmic}
        \Require Independent samples $x_1, x_2, \dots,x_n$ from $\textsf{Exp}(\theta^\star, S^\star)$ and constant $\alpha>0$ from \cref{asmp:1:sufficientMass}

        \vspace{2mm}
        
        \State Calculate the sample mean $\wh{\mu}_{\theta^\star, S^\star}  \leftarrow \frac{1}{n} \sum_{i \in [n]} x_i$ %

        \State Define the linear transformation $L\colon\R^d\to\R^d$ as the map $z\mapsto {-z / \wh{\mu}_{\theta^\star, S^\star}}$

        \State \Return $L(\cdot)$ and a description of the set $\Theta$ in Equation~\eqref{def:expTheta} with $r,(1/R)=\poly(\alpha)$

    \end{algorithmic}
\end{algorithm}

\paragraph{Pre-processing Routine.} 
    For each $\theta$, let $\mu_\theta$ be the mean of $\textsf{Exp}(\theta)$; it is related to $\theta$ by $\mu_\theta=-1/\theta$. 
    Let $\mu^\star$ denote $\mu_{\thetaStar}$.
    The inputs to our pre-processing algorithm (\cref{alg:preprocess:exponential}) are a constant $\alpha > 0$ such that $\textsf{Exp}(\Sstar; \thetaStar) \geq \alpha$ and 
    \[ 
        n = \tilde{O}\inparen{\frac{d}{\alpha^2}\cdot\log\frac{1}{\alpha}\cdot \log^2\frac{1}{\delta}} 
    \]
    independent samples from $\textsf{Exp}(\theta^\star, S^\star)$.
    Given this, the algorithm computes the empirical mean $\wh{\mu}_{\theta^\star, S^\star}$ and, then, transforms the space to ensure that the empirical mean is $1$.

\paragraph{Family of Parameter Sets.} 
    We show that this pre-processing will ensure that $\mu^\star$ is close to 1 (the all ones vector) and, hence, as $\muStar=-1/\thetaStar$ also ensure that $\thetaStar$ is close to $-1$.
    Concretely, in \cref{prop:exponential_satisfy_asmp:interiority}, we show that 
    \[ 
        \text{for all $1\leq i\leq d$}\,,\quad 
        \Omega(\alpha) \leq \thetaStar_i \leq O\inparen{\log\frac{1}{\alpha}}
        \qquadand
        \norm{\thetaStar + 1}_2 \leq \poly\inparen{\frac{1}{\alpha}}
        \,.
        \yesnum\label{eq:preprocess:exp_mean}
    \]
    This, in particular, allows us to construct a set $\Theta$ containing $\theta^\star$ while satisfying the bounded covariance assumption. 
    We select $\Theta$ from the following family:
    Given $r,R>0$, $\Theta_{r,R}$ is
    \begin{equation*}
            \Theta_{r, R} \coloneqq \insquare{-\frac{1}{r}, -r}^d \cap B\inparen{-1,R}\,.
        \tagnum{Parameter Set for Product Exponentials}{def:expTheta}
    \end{equation*}
Where $B(-1,R)$ is the $L_2$-ball of radius $R$ centered at $(-1,-1,\dots,-1)$.
To select a suitable $r$ and $R$, observe that \cref{eq:preprocess:exp_mean} implies that, with high probability, $\thetaStar\in \Theta_{r,R}$ for $r=\poly(\alpha)$ and $R=\poly(1/\alpha)$.
\begin{definition}[Parameter Set Satisfying \cref{asmp:2}]\label{def:parameterSetExp}
    Fix $r=\poly(\alpha)$ and $R=\poly(1/\alpha)$.
    Define $\Theta \coloneqq \Theta_{r,R}$ (Equation~\eqref{def:expTheta}.)
\end{definition}
In the remainder of this section, we verify that this parameter set satisfies all properties listed in \cref{asmp:2}:
\begin{enumerate}
    \item \textbf{Convexity:} $\Theta$ is convex as it is an intersection of two convex sets: an $L_\infty$-ball and an $L_2$-ball.
    \item \textbf{Bounds on covariance of sufficient statistic:} 
    For each $\theta$, it can be shown that $\cov_{x\sim \textsf{Exp}(\theta)}[t(x)] = {\rm diag}\inparen{1/\theta_1^2, 1/\theta_2^2,\dots,1/\theta_d^2}$; where the non-diagonal entries are $0$ since $\textsf{Exp}(\theta)$ is a product distribution.
    This combined with the fact that for each $\theta\in \Theta$, $\poly(\alpha)\leq \norm{\theta}_\infty\leq \poly(1/\alpha)$, implies that for each $\theta\in \Theta$
    \[
        \poly\inparen{\alpha}\cdot I\preceq
        \cov_{x\sim \textsf{Exp}(\theta)}[t(x)]\preceq 
        \poly\inparen{\frac{1}{\alpha}}\cdot I
        \,.
    \]
    \item \textbf{Interiority of optimal:} 
    \cref{prop:exponential_satisfy_asmp:interiority} shows that $\theta^\star \in \Theta(\poly(\alpha))$.
    \item \textbf{Existence of a $\chi^2$-bridge:} 
    \cref{prop:exponential_satisfy_asmp:chi2_bridge} shows that for any $\theta_1, \theta_2 \in \Theta$, there is a $\theta$ such that  
    \[
            \chidiv{\textsf{Exp}(\theta_1)}{\textsf{Exp}(\theta)}\,,~~
            \chidiv{\textsf{Exp}(\theta_2)}{\textsf{Exp}(\theta)}\leq 
            e^{\poly\inparen{\sfrac{1}{\alpha}}} - 1
            \,.
    \]
    \item \textbf{Projection oracles:} %
            Fix $\eta=\poly(r)$. 
            First, observe that $O=\Theta_{r-\eta,R-\eta}$ is a convex set in the $\eta$-interior of $\Theta=\Theta_{r,R}$.
            Second, we can construct projection oracles to both $O$ and $\Theta$:
                this is because both $O$ and $\Theta$ are intersections of a known $L_\infty$-ball and a known $L_2$-ball (see Equation~\eqref{def:parameterSetExp}) and, hence, we can implement separation oracles for them and then use the Ellipsoid method \cite{grotschel1988geometric} to project to their intersection.
    \item \textbf{Feasiblity of \cref{asmp:moment}:} %
        \cref{prop:exponential_satisfy_asmp:feasibility_trunc_mle} shows that for any vector $v\in \R^m$ such that $v\approx \E_{\textsf{Exp}(\theta^\star, S^\star)}[t(x)]$, there exists $\theta \in \Theta$, such that $\E_{\textsf{Exp}(\theta)}[t(x)] = v$.
\end{enumerate}
In the remainder of this section we prove \cref{prop:exponential_satisfy_asmp:interiority,prop:exponential_satisfy_asmp:chi2_bridge,prop:exponential_satisfy_asmp:feasibility_trunc_mle}.

\subsubsection*{Interiority of $\theta^\star$}
    First, we prove some results required to prove \cref{prop:exponential_satisfy_asmp:interiority}.
\begin{lemma}[Concentration of Empirical Mean]\label{prop:exponential_satisfy_asmp:trun_concen}
    Suppose $\textsf{Exp}(S^\star; \theta^\star) \geq \alpha > 0$.
    Define $\mu_{\theta^\star, S^\star} \coloneqq \E_{\textsf{Exp}(\theta^\star, S^\star)}[x]$. 
    \sloppy
    For any $\eps, \delta > 0$, given $n = \tilde{O}\sinparen{\sinparen{\sfrac{d}{\eps^2}}\log(1/\alpha)\log^2(1/\delta)}$ independent samples ${x_1, x_2, \dots, x_n}$ from $\textsf{Exp}(\theta^\star, S^\star)$, 
    with probability $1-\delta$,
    the sample mean $\wh{\mu}_{\theta^\star, S^\star} = (1/n)\sum_i x_i$ satisfies the following: 
    \[ 
        \text{for each $1\leq i\leq d$}\,,\quad 
        \abs{
        {\theta^\star_i}
            \cdot 
            \inparen{
                \wh{\mu}_{\theta^\star, S^\star}
                - \mu_{\theta^\star, S^\star}
            }_i
        } 
        \leq 
        \frac{\eps}{\sqrt{d}}\,. 
    \]
    \fussy 
\end{lemma}
    \begin{proof}[Proof sketch of \cref{prop:exponential_satisfy_asmp:trun_concen}]
        \cref{prop:exponential_satisfy_asmp:trun_concen} follows from arguments similar to the proof of Lemma 5 of \citet{daskalakis2018efficient}:
        Without loss of generality, we can assume that $\thetaStar=-1$ by performing the affine transformation $x \mapsto x / \E_{\textsf{Exp}(\theta^\star)}[x]$.
        Now, the $n$ samples from $\textsf{Exp}(-1, S^\star)$ can be thought of as $O(n/\alpha)$ samples from $\textsf{Exp}(-1)$ where we only keep those in $S^\star$.
        With probability $1-\delta$, the $L_\infty$-norm of all $O(n/\alpha)$ samples is at most $O\inparen{\log\inparen{\sfrac{nd}{\alpha\delta}}}$ due to tail bounds on the exponential distribution $\textsf{Exp}(-1)$.
        Finally, roughly speaking, conditioned on the event that this bound on the $L_\infty$-norm holds, we can apply Hoeffding's inequality to deduce \cref{prop:exponential_satisfy_asmp:trun_concen}.
    \end{proof}
    Next, we bound the distance between the mean of the truncated and non-truncated distributions.
\begin{lemma}[Truncated and Non-Truncated Mean Distances]\label{prop:exponential_satisfy_asmp:trunc_nontrunc_params}
    Recall that $\mu^\star = \E_{\textsf{Exp}(\theta^\star)}[x]$.
    Given a set $S$, define $\mu_{\theta^\star, S} \coloneqq {\E_{\textsf{Exp}(\theta^\star, S)}[x]}$.
    Suppose the set $S$ satisfies $\textsf{Exp}(S; \theta^\star) \geq \alpha$.
    Then it holds that, 
    \[ 
        \text{for each $1\leq i\leq d$\,,}\quad
        \frac{\alpha}{2}\cdot \mu^\star_i
        \leq 
        {[\mu_{\theta^\star, S}]_i} 
        \leq 
        \inparen{1+ \log{\frac{1}{\alpha}}}\cdot \mu^\star_i\,. 
    \]
\end{lemma}
\begin{proof}
    For notational convenience, in this proof let $\mu^S \coloneqq \mu_{\theta^\star, S}$. 
    Fix any $i \in [d]$. 
    
    \paragraph{Upper bound.} For the upper bound, note that
    \begin{align*}
        \mu^S_i 
        &~~\leq~~ 
            \max_{S \subseteq \R^d}~
                \frac{\E_{\textsf{Exp}(\theta^\star)}[x_i \cdot \mathds{1}_S(x)]}{\textsf{Exp}(S;\theta^\star)} 
                \quadtext{such that} \textsf{Exp}(S;\theta^\star) \geq \alpha \\
        &~~\leq~~ 
            \max_{\alpha\leq \beta\leq 1}~
                \max_{S \subseteq \R^d}~
                \frac{
                    \E_{\textsf{Exp}(\theta^\star)}[x_i \cdot \mathds{1}_S(x)] 
                }{
                    \beta
                }
            \quadtext{such that}
            \textsf{Exp}(S;\theta^\star) =\beta\,.
    \end{align*}
    The inner maximization problem is a fractional knapsack problem whose solution is
    \[ 
        S_\opt = \inbrace{x : x_i \geq \tau} 
        \quadtext{for threshold $\tau \geq 0$ that ensures} 
        \textsf{Exp}(S_\opt; \theta^\star) = \beta
        \,.
    \]
    Since $\textsf{Exp}(S_\opt; \theta)  = \inparen{-\sfrac{1}{\theta^\star_i}} \cdot \int_\tau^\infty \exp(\theta_i x_i)\d x_i $, this implies that
    \[ 
        \tau = -\frac{1}{\thetaStar_i}\log\frac{1}{\beta}\,. 
    \]  
    Substituting the form of $S_{\rm OPT}$ in the optimization problem above implies that 
    \[
        \mu^S_i ~~\leq~~
        \max_{\alpha\leq \beta\leq 1}~ 
        \frac{
            \E_{\textsf{Exp}(\theta^\star)}\insquare{x_i \cdot \mathds{I}[x_i \geq \tau]}
        }{
            \Pr_{\textsf{Exp}(\thetaStar)}[x_i \geq \tau]
        }
        ~~=~~ 
        \max_{\alpha\leq \beta\leq 1}~ 
        \E_{\textsf{Exp}(\theta^\star)}\insquare{x_i \mid x_i \geq \tau}
        \,.
    \]
    We now have that 
    \[
        \mu^S_i 
        ~~\leq~~  \max_{\alpha\leq \beta\leq 1}~ 
            \E_{\textsf{Exp}(\theta^\star)}\insquare{x_i \mid x_i \geq \tau}
        ~~=~~ \max_{\alpha\leq \beta\leq 1}~ 
            -\frac{1}{\thetaStar_i}\inparen{1 + \log\frac{1}{\beta}}
        ~~\Stackrel{\thetaStar_i < 0}{\leq}~~
        -\frac{1}{\thetaStar_i}\inparen{1 + \log\frac{1}{\alpha}}\,. 
    \]
    \paragraph{Lower bound.} For the lower bound, similarly
    \begin{align*}
        \mu^S_i &~~\geq~~ 
            \min_{S \subseteq \R^d}~ 
                \frac{\E_{\textsf{Exp}(\theta^\star)}[x_i \cdot \mathds{1}_S(x)]}{\textsf{Exp}(S;\theta^\star)} \quadtext{such that} 
                \textsf{Exp}(S;\theta^\star) \geq \alpha \\
        &~~\geq~~ 
            \min_{S \subseteq \R^d}~ 
                \E_{\textsf{Exp}(\theta^\star)}[x_i \cdot \mathds{1}_S(x)] ~~~\quadtext{such that} 
                \textsf{Exp}(S;\theta^\star) \geq \alpha\,.
    \end{align*}
    This is a fractional covering problem which is solved by the following set 
    \[ 
        S_\opt = \inbrace{x : x_i \leq \tau} 
            \quadtext{for threshold $\tau\geq 0$ that ensures} 
            \textsf{Exp}(S_\opt; \theta^\star) = \alpha\,.
    \]
    Since $\textsf{Exp}(S_\opt; \theta^\star) = \inparen{-\sfrac{1}{\theta^\star_i}} \int_0^\tau \exp(\theta_i x_i)\d x_i$, this implies that 
    \[ 
        \tau = -\frac{1}{\thetaStar_i} \log{{\frac{1}{1-\alpha}}}\,. 
    \] 
    Therefore, 
    \[ 
        \mu_i^S \geq \E_{\textsf{Exp}(\thetaStar)}[x_i \mid x_i \leq \tau] = -\frac{1}{\thetaStar} \inparen{1 - \inparen{\frac{1-\alpha}{\alpha}}\log \frac{1}{1-\alpha}} \geq -\frac{\alpha}{2\theta^\star}\,. 
    \]
    Thus,
    \[  
        \mu^S_i \in -\frac{1}{\theta_i^\star} \cdot \insquare{\frac{\alpha}{2},\, 1+ \log\inparen{\frac{1}{\alpha}}}
        \quad\Stackrel{(\theta^\star = -\sfrac{1}{\mu^\star})}{\iff}\quad 
        \mu_i^S \in \mu_i^\star \cdot \insquare{\frac{\alpha}{2},\, 1+  \log\inparen{\frac{1}{\alpha}}}\,.
    \]
\end{proof}
Combining \cref{prop:exponential_satisfy_asmp:trun_concen,prop:exponential_satisfy_asmp:trunc_nontrunc_params} implies the following corollary.
\begin{corollary}[Empirical and Non-Truncated Mean Distances]\label{coro:exponential_satisfy_asmp:emp_nontr_means}
     Let ${x_1, x_2, \dots, x_n}$ be independent samples from $\textsf{Exp}(\theta^\star, S^\star)$. 
     For any $\eps, \delta > 0$, given $n = \tilde{O}\sinparen{\sinparen{\sfrac{d}{\eps^2}}\log(1/\alpha)\log^2(1/\delta)}$ independent samples ${x_1, x_2, \dots, x_n}$ from $\textsf{Exp}(\theta^\star, S^\star)$, 
     with probability $1-\delta$,
     the sample mean $\wh{\mu}_{\thetaStar,\Sstar}$ satisfies the following
    \[ 
        \text{for each $1\leq i\leq d$}\,,\quad 
        \quad 
        \inparen{\frac{\alpha}{2}-\frac{\eps}{\sqrt{d}}}\cdot 
        \muStar_i
        \leq 
        [\wh{\mu}_{\theta^\star, \Sstar}]_i 
        \leq 
        \inparen{1+ \log{\frac{1}{\alpha}} + \frac{\eps}{\sqrt{d}}}
        \cdot 
        \mu^\star_i\,. 
    \]
\end{corollary}
Now we are ready to prove \cref{prop:exponential_satisfy_asmp:interiority}.
\begin{lemma}[Interiority of $\theta^\star$ after processing]\label{prop:exponential_satisfy_asmp:interiority}\sloppy 
    Fix $\delta > 0$ and $0 < \eps < \alpha/4$. Given $n = \tilde{O}\sinparen{\sinparen{\sfrac{d}{\eps^2}}\log(1/\alpha)\log^2(1/\delta)}$ independent samples $x_1,x_2,\dots,x_n$ from $\textsf{Exp}(\thetaStar, \Sstar)$, 
    divide all points by the sample mean $\wh{\mu}_{\thetaStar, \Sstar}$. 
    After this transformation, with probability $1-\delta$, the following hold 
    \begin{enumerate}
        \item For each $1\leq i\leq d$, $\poly(\alpha)\leq -\thetaStar_i\leq \poly(1/\alpha)$; and
        \item $\norm{\thetaStar+1}\leq \poly(1/\alpha)$.
    \end{enumerate}
    Hence, $\theta^\star \in \Theta_{r,R}(\eta)$ for $\eta = \poly(\alpha)$ and appropriate $r=\poly(\alpha)$ and $R=\poly(1/\alpha)$.
\end{lemma}
\begin{proof}
    We claim that \cref{prop:exponential_satisfy_asmp:interiority} follows if we prove that, with probability $1-\delta$, 
    the following list of inequalities hold.
    \begin{align*}
        \norm{\thetaStar}_\infty &~~\in~~ 
            \insquare{
                \poly(\alpha)\,, 
                \poly\inparen{\frac{1}{\alpha}}
            }\,,
            \yesnum\label{eq:prop:exponential_satisty_asmp:1}\\
        \norm{\wh{\mu}_{\thetaStar, \Sstar} - \mu_{\thetaStar, \Sstar}}_\infty &~~\leq~~ 
            \frac{\eps}{\sqrt{d}}\cdot 
            \max_{1\leq i\leq d} \frac{1}{\abs{\thetaStar_i}}\,,
            \yesnum\label{eq:prop:exponential_satisty_asmp:2}\\
        \norm{\muStar - \mu_{\thetaStar, \Sstar}}_2 &~~\leq~~ 
            \poly\inparen{\frac{1}{\alpha}}
            \,.
            \yesnum\label{eq:prop:exponential_satisty_asmp:3}
    \end{align*}
    First, we show that the above inequalities are sufficient to prove \cref{prop:exponential_satisfy_asmp:interiority}.
    Since $\thetaStar\in (-\infty,0]^d$, \cref{eq:prop:exponential_satisty_asmp:1} implies the first claim in \cref{prop:exponential_satisfy_asmp:interiority}.
    The second claim follows since, after the pre-processing $\wh{\mu}_{\thetaStar, \Sstar}=1$ and, hence, \cref{eq:prop:exponential_satisty_asmp:1,eq:prop:exponential_satisty_asmp:2,eq:prop:exponential_satisty_asmp:3} imply that 
    \begin{align*}
        \norm{\muStar - 1}_2
        &\leq 
        \norm{\muStar - \mu_{\thetaStar, \Sstar}}_2
        + \norm{1 - \mu_{\thetaStar, \Sstar}}_2\\
        &\leq 
        \norm{\muStar - \mu_{\thetaStar, \Sstar}}_2
        + \sqrt{d}\cdot \norm{1 - \mu_{\thetaStar, \Sstar}}_\infty\\
        &=
        \norm{\muStar - \mu_{\thetaStar, \Sstar}}_2
        + \sqrt{d}\cdot \norm{\wh{\mu}_{\thetaStar, \Sstar} - \mu_{\thetaStar, \Sstar}}_\infty\\
        &\leq \poly\inparen{\frac{1}{\alpha}} + \eps \cdot \poly\inparen{\frac{1}{\alpha}}\\
        &\leq \poly\inparen{\frac{1}{\alpha}}\,.
    \end{align*} 
    Now, the second claim can be deduced by substituting $\muStar=-1/\thetaStar$ and using \cref{eq:prop:exponential_satisty_asmp:1}:
    \[
        \norm{\thetaStar + 1}
        \leq \max_{1\leq i\leq d} \frac{1}{\abs{\thetaStar_i}} \cdot \poly\inparen{\frac{1}{\alpha}}
        ~~\Stackrel{\eqref{eq:prop:exponential_satisty_asmp:1}}{\leq}~~ \poly{\inparen{\frac{1}{\alpha}}}\,.
    \]
    In the remainder of the proof, we prove \cref{eq:prop:exponential_satisty_asmp:1,eq:prop:exponential_satisty_asmp:2,eq:prop:exponential_satisty_asmp:3}.

    \paragraph{Proof of \cref{eq:prop:exponential_satisty_asmp:1}.}
        Recall that after pre-processing $\wh{\mu}_{\thetaStar,\Sstar}=1$.
        Substituting this in \cref{coro:exponential_satisfy_asmp:emp_nontr_means} implies that, with probability $1-(\delta/3)$, \cref{eq:prop:exponential_satisty_asmp:1} holds.
        
    \paragraph{Proof of \cref{eq:prop:exponential_satisty_asmp:2}.}
        \cref{prop:exponential_satisfy_asmp:trun_concen} shows that \cref{eq:prop:exponential_satisty_asmp:2}  holds with probability $1-(\delta/3)$.
        
    \paragraph{Proof of \cref{eq:prop:exponential_satisty_asmp:3}.}
        Toward proving \cref{eq:prop:exponential_satisty_asmp:3}, define the following set 
        \[
            \wt{\Theta}\coloneqq \insquare{-1 - \poly\inparen{\frac{1}{\alpha}}, \frac{1}{2}\cdot \poly(\alpha)}\,.
        \]
        Where the $\poly(1/\alpha)$ and $\poly(\alpha)$ terms are the same as the ones in \cref{eq:prop:exponential_satisty_asmp:1}.
        Conditioned on the event that \cref{eq:prop:exponential_satisty_asmp:1} holds, \cref{eq:prop:exponential_satisty_asmp:1} and the definition of $\wt{\Theta}$ imply that $\thetaStar$ is in the $\poly(\alpha)$ interior of $\wt{\Theta}$.
        Now, \cref{coro:SSMeansClose} implies \cref{eq:prop:exponential_satisty_asmp:3}; to see this, observe that the sufficient statistic of the product exponential distributions $t(x)=x$.

\end{proof}

    \subsubsection*{Existence of a $\chi^2$-Bridge}
        Next, we prove \cref{prop:exponential_satisfy_asmp:chi2_bridge} which proves the existence of a $\chi^2$-Bridge.
\begin{lemma}\label{prop:exponential_satisfy_asmp:chi2_bridge}
    For any pair of exponential distributions $\cD_1\coloneqq \textsf{Exp}(\phi)$ and $\cD_2\coloneqq \textsf{Exp}(\gamma)$ with $\phi,\gamma\in \Theta_{\poly(\alpha),\poly(1/\alpha)}$, there exists another exponential distribution $\cD\coloneqq \textsf{Exp}(\lambda)$ such that 
            \[
                \renyi{3}{\cD_1}{\cD}\,,\ \renyi{3}{\cD_2}{\cD}
                ~~\leq~~
                \poly\inparen{\frac{1}{\alpha}}
                \,.
            \]
    Hence, it also follows that, 
    \[
                \chidiv{\cD_1}{\cD}\,,\ \chidiv{\cD_2}{\cD}
                ~~\leq~~
                e^{\poly\inparen{\sfrac{1}{\alpha}}}-1
                \,.
    \]
    
\end{lemma}
 
\begin{proof}
    Define $\lambda_i=\min\inbrace{\phi_i,\gamma_i}$ for each $i$.
    Since $\phi, \gamma \in \Theta_{\poly(\alpha),\poly(1/\alpha)}$, it follows that 
    \[
        \sum_i \inparen{\frac{1}{\phi_i}+1}^2
        \leq  \frac{
            \sum_i \inparen{1+\phi_i}^2
        }{
            \min_j \phi_j^2
        } 
        \leq \poly\inparen{\frac{1}{\alpha}}
        \quadand
        \sum_i \inparen{\frac{1}{\gamma_i}+1}^2
        \leq  \frac{
            \sum_i \inparen{1 + \gamma_i}^2
        }{
            \min_j \gamma_j^2
        } 
        \leq 
        \poly\inparen{\frac{1}{\alpha}}\,.
    \]
    Observe that 
    \begin{align*}
        e^{2\cdot \renyi{3}{\cD_1}{\cD}}
        &= \prod_{i=1}^d \int_{x=0}^\infty 
            \frac{-{\phi_i }^3}{
                {\lambda^2_i }
            }
            \frac{e^{3\phi_i x}}{e^{2\lambda_i  x}}
            \d x
        = \prod_{i=1}^d 
            \frac{{\phi_i }^3}{
                {\lambda^2_i }
            }
            \frac{1}{{3\phi_i  - 2\lambda_i }}\,.
    \end{align*}
    Taking the logarithm implies that
    \[
        2\cdot \renyi{3}{\cD_1}{\cD}
        = \sum_{i=1}^d 
        \log\inparen{
                \frac{{\phi_i }^3}{
                    {\lambda^2_i }
                }
                \frac{1}{{3\phi_i  - 2\lambda_i }}
        }
        = \sum_{i: \phi_i\geq \gamma_i} 
        \log\inparen{
                \frac{(\phi_i/\gamma_i)^3}{{3(\phi_i/\gamma_i)  - 2}}
        }
        \leq \sum_{i: \phi_i\geq \gamma_i} 3 \inparen{\frac{\phi_i}{\gamma_i}-1}^2\,,
    \]
    where in the inequality we used that $\log\inparen{\sfrac{z^3}{(3z-2)}}\leq 3(z-1)^2$ for all $z\geq 1.$
    Further, 
    \[ 
        2\cdot \renyi{3}{\cD_1}{\cD}
        \leq \sum_{i: \phi_i\geq \mu_i} 3\phi_i^2 \inparen{\frac{1}{\gamma_i}-\frac{1}{\phi_i}}^2
        \leq \poly\inparen{\frac{1}{\alpha}}
            \sum_{i: \phi_i\geq \gamma_i} \inparen{\frac{1}{\gamma_i}-\frac{1}{\phi_i}}^2\,.
    \]
    Where we used the fact that $\phi\in \Theta_{\poly(\alpha),\poly(1/\alpha)}$ and, hence, $\norm{\phi}_\infty\leq \poly(1/\alpha)$.
    Finally, we have the following upper bound:
    \[ 
        \renyi{3}{\cD_1}{\cD}
        \leq \poly\inparen{\frac{1}{\alpha}}
            \sum_{i: \phi_i\geq \gamma_i} 
                \inparen{
                    \inparen{\frac{1}{\gamma_i}+1}^2
                    +
                    \inparen{\frac{1}{\phi_i}+1}^2
                }
        \leq \poly\inparen{\frac{1}{\alpha}}\,.
    \]
    By symmetry, we also have that $\renyi{3}{\cD_2}{\cD}\leq \poly\inparen{\sfrac{1}{\alpha}}$ completing the proof.
\end{proof}

\subsubsection*{Feasibility of \cref{asmp:moment}}
Finally, we prove \cref{prop:exponential_satisfy_asmp:feasibility_trunc_mle} that gives an algorithm to satisfy \cref{asmp:moment}.
\begin{lemma}\label{prop:exponential_satisfy_asmp:feasibility_trunc_mle}
    Let $\eps > 0$. 
    For each vector with positive entries $v\in \R^d_{>0}$, define $\theta_v=\inparen{-\sfrac{1}{v_1},-\sfrac{1}{v_2},\dots,-\sfrac{1}{v_d}}$.
    Given any vector $v\in \R^d_{>0}$ satisfying, $\norm{v - \E_{\textsf{Exp}\inparen{\theta^\star, S^\star}}[t(x)]} < \eps$, it holds that $\Ex_{\textsf{Exp}(\theta_v)}[t(x)] = v$.
    Hence, for any $v$ satisfying the condition in \cref{asmp:moment}, $\theta_v$ satisfies the property promised in \cref{asmp:moment}.
\end{lemma}
\begin{proof}
    The result follows from the definition of $\theta_v$, since for any $\theta$, $\Ex_{\textsf{Exp}(\theta)}[t(x)]=-\sfrac{1}{\theta}$.
\end{proof}

\section*{Acknowledgments}
    We thank Sinho Chewi, Alkis Kalavasis, Siddharth Mitra, Vishwak Srinivasan, Andre Wibisono, and Felix Zhou for useful discussions and references. 
    We thank Alex Kouridakis  for pointing out an error in the original proof of  \cref{prop:intro:cor10cor11}.
    Jane H.\ Lee was supported by a Graduate Fellowship for STEM Diversity sponsored by the U.S. National Security Agency (NSA).

\newpage

\newpage

\printbibliography

\newpage

\appendix

\newpage

\section{Additional Preliminaries}\label{sec:additionalPreliminaries}

        \begin{lemma}[{Exact Samples From TV Sampler}]\label{lem:exactSampleAccess}
            For any $\delta\in (0,\sfrac{1}{2})$, $n\geq 1$, and distribution $\cD^\star$, 
                if $n$ samples $x_1,x_2,\dots,x_n\in \R^d$ drawn independently from distributions $\cD_1,\cD_2,\dots,\cD_n$ such that, for each $1\leq i\leq n$, $\tv{\cD_i}{\cE(\theta)}\leq \sfrac{\delta}{(2n)}$,
                then there is an event $\evE$ that happens with probability at least $1-\delta$, 
                such that conditioned on $\evE$, the distribution of $x_i$ is exactly $\cD^\star$ for each $1\leq i\leq n$.
        \end{lemma}
        Thus, using given access to a TV sampler for $\cE=\inbrace{\cE(\theta)\colon \Theta\in \R^m}$ that, given $\theta$ and $n$, generates $n$ samples $\delta$-close in TV distance to $\cE(\theta)$ in $\poly(nm/\delta)$ time, one can generate samples that with $1-\delta$ probability are drawn \textit{exactly} from $\cE(\theta)$ in $\poly(nm/\delta)$ time.
        \begin{proof}[Proof of \cref{lem:exactSampleAccess}]
            Since $\tv{\cD_i}{\cE(\theta)}\leq \sfrac{\delta}{(2n)}$, for each $i$, there is a density $\cP_i$ such that 
            \[
                \cD_i = \inparen{1-\frac{\delta}{2n}} \cdot \cE(\theta) + \frac{\delta}{2n}\cdot \cP_i\,.
            \]
            This implies that, with probability $1-\inparen{\sfrac{\delta}{(2n)}}$, $x_i\sim\cD_i$ is drawn from $\cE(\theta)$.
            Let $\evE$ be the event that the samples from $\cD_i$ are exact samples from $\cE(\theta)$ for all $1\leq i\leq n$.
            Since $x_1,x_2,\dots,x_n$ are drawn independently
            \[
                \Pr\insquare{\evE} = \inparen{1-\frac{\delta}{2n}}^n
                ~~\quad\Stackrel{
                    \delta\leq \sfrac{1}{2},\ n\geq 1
                }{\leq}\quad~~ \delta\,.  
            \]
        \end{proof}
     Next, we list useful bounds on the covariance of truncations of exponential family distributions.
    \begin{lemma}[Lemmas 3.2 and 3.3 of \citet{lee2023learning}]\label{lem:pres_sc_smooth}
        Suppose \cref{asmp:1:sufficientMass,asmp:1:polynomialStatistics,asmp:cov,asmp:int} hold. For any $\theta \in \Theta$ and set $S\subseteq\R^d$ satisfying $\cE(S; \theta)>0$,
        \[ \Omega\inparen{\inparen{\frac{\cE(S;\theta)}{k}}^{2k}}\lambda I \preceq \cov_{\cE(\theta, S)}[t(x)] \preceq \frac{\Lambda}{\cE(S;\theta)}\cdot I\,. \]
    \end{lemma}
    The next result enables us to use distances between parameters of exponential distributions to control the total variation distance between them.
    \begin{lemma}[{TV Distance between Smooth Exponential Family Distributions}]\label{lem:tv_smooth}
        Fix some set $\Theta\subseteq\R^m$ and constants $\Lambda> 0$.
        Let $\cE(\cdot)$ be an exponential family parameterized by $m$-dimensional vectors such that $\cov_{\cE(\theta)}[t(x)] \preceq \Lambda I$ for all $\theta \in \Theta$. 
        Then for any $\theta, \theta' \in \Theta$ it holds that
        $\tv{\cE\sinparen{\theta'}}{ \cE\sinparen{\theta}} \leq \sqrt{\sfrac{\Lambda}{2}}\cdot \norm{\theta' - \theta}_2\,.$
    \end{lemma}
    \begin{proof}
        For any $\theta,\theta'$ in the natural parameter space, there exists a $\xi\in L(\theta, \theta')$ (where $L(x,y) = \{ z \mid z = tx + (1-t)y, t \in [0,1] \}$) such that $\kl{\cE\sinparen{\theta'}}{\cE\sinparen{\theta}} = (\theta' - \theta)^\top \cov_{\cE\sinparen{\xi}}[t(x)](\theta' - \theta)\,.$
        (See Theorem 1 of \citet{BusaFekete2019OptimalLO} for a proof.)
    Due to the upper bound on the covariance matrix %
    \[
        \kl{\cE\sinparen{\theta'}}{\cE\sinparen{\theta}} \leq \Lambda \cdot \norm{\theta' - \theta}_2^2\,.
    \]
    Therefore, the Pinsker's inequality implies that 
    \[ 
        \tv{\cE\sinparen{\theta'}}{\cE\sinparen{\theta}} \leq \sqrt{\frac{1}{2}\kl{\cE\sinparen{\theta'}}{\cE\sinparen{\theta}}} \leq
        \sqrt{\frac{\Lambda}{2}}\cdot \norm{\theta' - \theta}_2\,.
    \]
    \end{proof}
    Our final set of results in this section are useful facts about the truncated Gaussian distributions.
    \begin{theorem}\label{thm:resultsOfdaskalakis2018efficient}
        Let $\inparen{\mu_S,\Sigma_S}$ be the mean and covariance of the truncated Gaussian $\cN(\mu,\Sigma, S)$ where $S\subseteq\R^d$ is such that $\cN(S;\mu,\Sigma)=\alpha$.
        Fix $\eps, \delta > 0$. 
        Given $n=\wt{O}\inparen{\inparen{\sfrac{d}{\eps^2}}\log\inparen{\sfrac{1}{\alpha}}\log^2\inparen{\sfrac{1}{\delta}}}$ independent samples from $\cN(\mu,\Sigma,S)$, with probability at least $1-\delta$, their empirical mean and covariance $\wh{\mu}_S$ and $\wh{\Sigma}_S$ satisfy the following properties
        \begin{enumerate}
            \item \textit{(Relation between $\sinparen{\wh{\mu},\wh{\Sigma}}$ and $\inparen{\mu_S,\Sigma_S}$; Lemma 5 of \citet{daskalakis2018efficient})}
            \[
                \norm{
                \Sigma^{-1/2}\inparen{\wh{\mu}_S-\mu_S}
            }_2
                \leq\eps
                \qquadand
                (1-\eps)\Sigma_S
                \preceq \wh{\Sigma}_S 
                    \preceq (1+\eps)\Sigma_S\,;
            \]
            \item \textit{(Relation between $\sinparen{\wh{\mu},\wh{\Sigma}}$ and $\inparen{\mu,\Sigma}$; Corollary 1 of \citet{daskalakis2018efficient})}
            \[
                \norm{\Sigma^{-1/2}\inparen{\wh{\mu}_S-\mu}}
                    \leq O\inparen{{\log^{1/2}{\frac{1}{\alpha}}}},\quad
                \wh{\Sigma}_S
                    \succeq \Omega\inparen{\alpha^2} \cdot \Sigma,
                \quadand 
                \norm{\Sigma^{-1/2}\wh{\Sigma}_S\Sigma^{-1/2}-I}_F\leq O\inparen{\log{\frac{1}{\alpha}}}
                \,;
            \]
            \item \textit{(Further relations between $\sinparen{\wh{\mu},\wh{\Sigma}}$ and $\inparen{\mu,\Sigma}$; Proposition 1 of \citet{daskalakis2018efficient})}
            \begin{align*}
                &\norm{
                  I - \Sigma^{1/2}\wh{\Sigma}^{-1}\Sigma^{1/2}
                }_F \leq O\inparen{\frac{\log\inparen{\sfrac{1}{\alpha}}}{\alpha^2}}
                \qquadand
                \norm{
                  I - \wh{\Sigma}_S^{1/2} {\Sigma}^{-1}\wh{\Sigma}_S^{1/2}
                }_F \leq O\inparen{\frac{\log\inparen{\sfrac{1}{\alpha}}}{\alpha^2}}\,,\\
                &\Omega\inparen{\alpha^2}
                \cdot I
                \preceq \Sigma^{-1/2}\wh{\Sigma}_S\Sigma^{-1/2}
                \preceq O\inparen{\frac{1}{\alpha^2}}\cdot I
                \qquadand
                \Omega\inparen{\alpha^2}
                \cdot I
                \preceq \wh{\Sigma}_S^{-1/2}{\Sigma}\wh{\Sigma}_S^{-1/2}
                \preceq O\inparen{\frac{1}{\alpha^2}}\cdot I
                \,,\\
                &\norm{
                  \wh{\Sigma}_S^{-1} {\Sigma}^{1/2}\inparen{
                    \wh{\mu}_S - \mu
                  }
                }_2 \leq O\inparen{\frac{\log\inparen{\sfrac{1}{\alpha}}}{\alpha^2}}
                \qquadand
                \norm{
                    {\Sigma}^{-1}\wh{\Sigma}_S^{1/2}\inparen{
                    \wh{\mu}_S - \mu
                  }
                }_2 \leq O\inparen{\frac{\log\inparen{\sfrac{1}{\alpha}}}{\alpha^2}}
                \,.
            \end{align*}
        \end{enumerate}
        
    \end{theorem}

\section{Proofs Deferred From the Main Body} %

    \subsection{Proof of \texorpdfstring{\cref{fact:gradOfNoisyMLE}}{Fact 3.5}: Gradient and Hessian of \texorpdfstring{$\negLL{}_S(\cdot)$}{LS(.)}} %
    \label{sec:proofof:fact:gradOfNoisyMLE}
    
    In this section, we prove \cref{fact:gradOfNoisyMLE}, restated below.
    \gradHessNoisyMLE*

    \begin{proof}[Proof of \cref{fact:gradOfNoisyMLE}]
            Recall the definition of $\negLL$:
            \begin{align*}
                \negLL_S\sinparen{\theta}
                \coloneqq
                - \Ex_{x\sim \cE(\theta^\star,S^\star\cap S)}\insquare{\log{\cE(x; \theta,S )} }
                = - \Ex_{x\sim \cE(\theta^\star,S^\star)}\insquare{\log{\cE(x; \theta,S )}\mid x\in {S}}\,.
            \end{align*}
            Its derivative with respect to $\theta$ is
            \begin{align*}
                \nabla_\theta \negLL_S(\theta)
                &= - \Ex_{x\sim \cE(\theta^\star, S^\star)}\insquare{\frac{\nabla_\theta {\cE(x; \theta, S)}}{\cE(x; \theta, S)}\mid x\in S}\,.
            \end{align*}
            Recall that, for any $x$, $\cE(x;\theta)= h(x)e^{t(x)^\top\theta - A(\theta)}$ and 
            $\cE(x; \theta, S) = \cE(x;\theta)\cdot \frac{\mathbb{I}\insquare{x\in S}}{\cE(S; \theta)}$.
    Hence
    \begin{align*}
        \nabla_\theta\ {\cE(x; \theta)}
            &~~=~~ h(x)e^{t(x)^\top\theta - A(\theta)}\inparen{t(x) - \nabla A(\theta)}\\
            &~~\Stackrel{\eqref{eq:derivativesOfA}}{=}~~ h(x)e^{t(x)^\top\theta - A(\theta)}\inparen{t(x) - \Ex_{z\sim \cE(\theta)}\insquare{t(z)}}\\  
            &~~=~~ \cE(x; \theta)\inparen{t(x) - \Ex_{z\sim \cE(\theta)}\insquare{t(z)}},\yesnum\label{eq:gradDensity}\\ 
        \nabla_\theta {\cE(x; \theta, S)}
            &~~=~~ \underbrace{\nabla_\theta {\cE(x; \theta)} \frac{\mathbb{I}\insquare{x\in S}}{\cE(S; \theta)}}_{\text{Term 1}}
            - \underbrace{{\cE(x; \theta)} \frac{\mathbb{I}\insquare{x\in S}}{\inparen{\cE(S; \theta)}^2}
            \nabla_\theta \cE(S; \theta)}_{\text{Term 2}}\,.
            \yesnum\label{eq:grad:TruncatedDensity}
    \end{align*}
    Where $\nabla_\theta \cE(S; \theta)$ is as follows
    \begin{align*}
        \nabla_\theta \cE(S; \theta)
        &~~=~~
        \int_{x\in S} \nabla_\theta \cE(x;\theta) \d x\\ 
        &~~\Stackrel{\eqref{eq:gradDensity}}{=}~~
        \int_{x\in S} \cE(x; \theta)\inparen{t(x) - \Ex_{z\sim \cE(\theta)}\insquare{t(z)}} \d x\\ 
        &~~=~~ \cE(S;\theta)\inparen{
            \Ex_{x\sim \cE(\theta)}[t(x)\mid x\in S]
            - \Ex_{z\sim \cE(\theta)}\insquare{t(z)}
        }\,.
    \end{align*}
    Substituting this in \cref{eq:grad:TruncatedDensity}
    \begin{align*}
        \nabla_\theta {\cE(x; \theta, S)}
            &~~\Stackrel{\eqref{eq:expressionTruncatedDensity}}{=}~~  
                \cE(x; \theta,S)\inparen{
                    \underbrace{
                        t(x) - \Ex_{z\sim \cE(\theta)}\insquare{t(z)}}
                    _{\text{Term 1}}
                    + \underbrace{ 
                        \Ex_{z\sim \cE(\theta)}\insquare{t(z)}
                        -\Ex_{x\sim \cE(\theta)}[t(x)\mid x\in S] 
                    }_{\text{Term 2}}
            }\\
            &~~=~~  \cE(x; \theta,S)\inparen{
                t(x)
                -\Ex_{x\sim \cE(\theta)}[t(x)\mid x\in S]
            }\,.\yesnum\label{eq:GradTruncatedDensity2}
    \end{align*}
    Thus  
    \begin{align*}
        \nabla_\theta\negLL_S\inparen{\theta}
        &= - \Ex_{x\sim \cE(\theta^\star,S^\star)}\insquare{
            t(x) - \Ex_{z\sim \cE(\theta)}\insquare{t(z) \mid z\in S}
            \mid x\in S
        }\\ 
         &= - \Ex_{x\sim \cE(\theta^\star,S^\star\cap S)}\insquare{
            t(x)}
         + 
         \Ex_{x\sim \cE(\theta,S)}\insquare{t(x)}\,.
    \end{align*}
    Now, it follows that the second derivative is 
    \begin{align*}
        \nabla_\theta^2 \negLL_S(\theta)
        &~~=~~ \int_x t(x) \nabla_\theta \cE(x; \theta,S)^\top \d x\\
        &~~\Stackrel{\eqref{eq:GradTruncatedDensity2}}{=}~~ \int_x t(x) \cE(x; \theta,S)\inparen{
            t(x)
            -\Ex_{z\sim \cE(\theta)}[t(z)\mid z\in S]
        }^\top \d x\\
        &~~=~~ \Ex_{\cE(\theta,S)}\insquare{t(x) t(x)^\top } 
        - \Ex_{\cE(\theta, S)}[t(x)]\Ex_{\cE(\theta, S)}[t(x)]^\top\\ 
        &~~=~~ \cov_{\cE(\theta,S)}\insquare{t(x)}\,.
    \end{align*}
     
    \end{proof}

        \subsection{Proof of \texorpdfstring{\cref{prop:perturbed_mle_alt_proof:trunc_nontru_mean_suff}}{Lemma 6.8}: Bound on Truncation's Effect on the Moments of \texorpdfstring{$t(\cdot)$}{t(.)}} %
        In this section, we prove the following result.

        \changeInMomentSuffTrunc*
        \noindent The proof of \cref{prop:intro:cor10cor11} relies on the property that when $x\sim \cE(\thetaStar, \Sstar)$, $t(x)$ is a sub-exponential random variable.
        Before proving this, we recall the following definition.
        \begin{definition}[Sub-exponential Random Variable \cite{vershynin2018high}]
            Given $\nu,\beta\geq 0$, a zero-mean and real-valued random variable $z$ is said to be $(\nu^2,\beta)$-sub-exponential if, for all $-\sfrac{1}{\beta} < \gamma < \sfrac{1}{\beta}$, $\Ex\insquare{e^{\gamma z}}\leq e^{(\nu\gamma)^2/2}$.
        \end{definition}
        We use the following result.
        \begin{lemma}[Claim 1 of \citet{lee2023learning}]\label{prop:claim1lwz}
            Suppose \cref{asmp:1:sufficientMass,asmp:1:polynomialStatistics,asmp:cov,asmp:int} hold.
            For any unit vector $u$, it holds that the zero-mean and real-valued random variable $u^\top\inparen{t(x) - \E_{\cE(\theta^\star)}[t(x)]}$ is $\inparen{\Lambda, \sfrac{1}{\eta}}$-sub-exponential when $x\sim \cE(\thetaStar)$. 
        \end{lemma}
        In addition to the above property of sufficient statistics, the proof also uses the following property.
        \begin{lemma}[Concentration of Sufficient Statistics]\label{prop:perturbed_mle_alt_proof:trunc_concentration}
           Suppose \cref{asmp:1:sufficientMass,asmp:1:polynomialStatistics,asmp:cov,asmp:int} hold. 
           Fix any constants $\zeta,\delta\in (0,1)$ and set $T\subseteq\R^d$ with $\cE(T;\theta^\star)\geq\alpha$.
            Fix 
            \[
                n = \tilde{\Omega}\inparen{{\frac{m\Lambda^2}{\zeta^{2}\eta^2}}\log^2\inparen{\frac{1}{\alpha}}\log^2\inparen{\frac{1}{\delta}}}\,.   
            \]
            Given independent samples ${x_1,x_2,\dots,x_n}$ from $\cE(\theta^\star, T)$, with probability at least $1 - \delta$, 
            \[ 
                \norm{\frac{1}{n} \sum_{i \in [n]}t(x_i) - \E_{\cE(\theta^\star, T)}[t(x)] }_2\leq \zeta\,. 
            \]
        \end{lemma} 
        We prove \cref{prop:perturbed_mle_alt_proof:trunc_concentration} at the end of this section.
        \begin{proof}[Proof of \cref{prop:perturbed_mle_alt_proof:trunc_nontru_mean_suff}]\label{proofof:changeInMomentSuffTrunc}
            Let $u_1, u_2, \dots, u_m$ be an orthonormal basis for $\R^m$, where %
            \[ u_1 \propto {\E_{\cE(\theta^\star, T)}[t(x)] - \E_{\cE(\theta^\star)}[t(x)]}\,. 
            \yesnum\label{eq:pgsg:ss_bounds1}
            \]
            Consider the following equality %
            \begin{align*}
                &\norm{\frac{1}{n} \sum_{i \in [n]}t(x_i)  - \E_{\cE(\theta^\star)}[t(x)]}_2^2\\
                &\qquad= 
                \inparen{u_1^\top \inparen{\frac{1}{n} \sum_{i \in [n]}t(x_i)  - \E_{\cE(\theta^\star)}[t(x)]}}^2 + 
                \sum_{j\in [m]\colon j\neq 1} 
                \inparen{u_j^\top\inparen{\frac{1}{n} \sum_{i \in [n]}t(x_i)  - \E_{\cE(\theta^\star)}[t(x)]}}^2\,.
                \yesnum\label{eq:pgsg:ss_bounds2}
            \end{align*}
            To bound the first term in \cref{eq:pgsg:ss_bounds2}, we use \cref{prop:claim1lwz} along with concentration bounds for sub-exponential random variables \cite{vershynin2018high}:
            for any $c\geq \eta \Lambda$, it holds that 
            \[ 
                \Pr_{z_1,z_2,\dots,z_n \sim \cE(\theta^\star)}\inparen{u_1^\top\inparen{
                \sum_{i \in [n]}t(z_i) - \E_{\cE(\theta^\star)}[t(x)]} \geq c} \leq \exp\inparen{-\frac{nc \eta}{2}}\,. 
            \]
            If the samples $x_1,\dots,x_n\sim\cE(\thetaStar)$, then the above inequality is sufficient to bound \cref{eq:pgsg:ss_bounds2}.
            However, as $x_1,\dots,x_n\sim\cE(\thetaStar,T)$, we perform a change of measure from $\cE(\thetaStar)$ to $\cE(\thetaStar,T)$ to get the following:
            for any $c\geq \eta \Lambda$, it holds that 
            \[ 
                \Pr_{x_1,x_2,\dots,x_n \sim \cE(\theta^\star,T)}\inparen{u_1^\top\inparen{
                \sum_{i \in [n]}t(x_i) - \E_{\cE(\theta^\star)}[t(x)]} \geq c} 
                \leq
                \inparen{\frac{1}{\cE(S; \theta^\star)}}^n
                \cdot 
                \exp\inparen{-\frac{nc \eta}{2}}\,. 
            \]
            Since $n=\Omega\inparen{\sfrac{1}{(\zeta\eta)}\log\inparen{\sfrac{1}{\delta}}}$, it follows that with probability $1-(\delta/2)$ 
            \[
                u_1^\top\inparen{\sum_{i \in [n]}t(x_i) - \E_{\cE(\theta^\star)}[t(x)]} \leq 
                \max\inbrace{
                    \eta \Lambda\,, 
                    \frac{2}{\eta} \log \frac{1}{\cE(S;\theta^\star)} + \zeta
                }\,.
            \]
            This upper bounds the first term in \cref{eq:pgsg:ss_bounds2}.
            Next, we upper bound the second term in \cref{eq:pgsg:ss_bounds2}.
            Consider any $j\neq 1$:
            \begin{align*}
                u_j^\top \inparen{\sum_{i \in [n]}t(x_i) - \E_{\cE(\theta^\star)}[t(x)]} 
                &\qquad= \qquad 
                u_j^\top \inparen{
                    \sum_{i \in [n]}t(x_i)
                    - \E_{\cE(\theta^\star, S)}[t(x)] 
                    + {\E_{\cE(\theta^\star, S)}[t(x)] 
                    - \E_{\cE(\theta^\star)}[t(x)]}
                }\\
                &\qquad 
                \Stackrel{\eqref{eq:pgsg:ss_bounds1},\ \inangle{u_i,u_j}=0}{=}
                \qquad
                u_j^\top \inparen{\sum_{i \in [n]}t(x_i) - \E_{\cE(\theta^\star, S)}[t(x)]} 
                \,.
            \end{align*} 
            Further, since $n=\wt{\Omega}\inparen{{(\sfrac{m\Lambda^2}{(\eta^2\zeta^2)}) \log^2(1/\alpha)\log^2(1/\delta)}}$, \cref{prop:perturbed_mle_alt_proof:trunc_concentration} implies that with probability $1-(\delta/2)$: 
            \[
                \norm{\frac{1}{n}{\sum_{i \in [n]}t(x_i) - \E_{\cE(\theta^\star, S)}[t(x)]} }\leq \zeta\,.
            \]
            Hence, a union bound over the two aforementioned events implies that with probability $1-\delta$
            \[
                \norm{\frac{1}{n} \sum_{i \in [n]}t(x_i)  - \E_{\cE(\theta^\star)}[t(x)]}_2
                \leq 
                \max\inbrace{
                    \eta \Lambda\,,
                    \frac{2}{\eta} \cdot \log \frac{1}{\cE(S;\theta^\star)} + \zeta
                }\,.
            \] 
        \end{proof}
        We also have the following corollary of \cref{prop:perturbed_mle_alt_proof:trunc_nontru_mean_suff}.
        \begin{corollary}\label{coro:SSMeansClose}
            Consider the setup in \cref{prop:perturbed_mle_alt_proof:trunc_nontru_mean_suff}.
            It holds that 
            \[ 
                \norm{ \E_{\cE(\theta^\star)}[t(x)] -  \E_{\cE(\theta^\star, S)}[t(x)] }_2
                \leq 
                \max\inbrace{
                \eta \Lambda\,,
                    {\frac{2}{\eta}\cdot \log \frac{1}{\cE(S; \theta^\star)}}
                }\,. 
            \]
        \end{corollary}
        \begin{proof}
            Let ${x_1,x_2,\dots,x_n}$ be independent samples from $\cE(\theta^\star, S)$.
            Observe that 
            \[
                \norm{ \E_{\cE(\theta^\star)}[t(x)] -  \E_{\cE(\theta^\star, S)}[t(x)] }_2 
                \leq 
                \norm{ \E_{\cE(\theta^\star)}[t(x)] -  \frac{1}{n} \sum_{i \in [n]}t(x_i) }_2 
                + 
                \norm{ \frac{1}{n} \sum_{i \in [n]}t(x_i) -  \E_{\cE(\theta^\star, S)}[t(x)] }_2\,.
            \]
            Taking a union bound over \cref{prop:perturbed_mle_alt_proof:trunc_nontru_mean_suff,prop:perturbed_mle_alt_proof:trunc_concentration} with $\delta < 1/2$ implies that with positive probability 
            \[
                \norm{ \E_{\cE(\theta^\star)}[t(x)] -  \frac{1}{n} \sum_{i \in [n]}t(x_i) } 
                \leq 
                \max\inbrace{
                    \eta \Lambda\,,
                    {\frac{2}{\eta}\cdot \log \frac{1}{\cE(S; \theta^\star)}} + \zeta
                }
                \quadand
                \norm{ \frac{1}{n} \sum_{i \in [n]}t(x_i) -  \E_{\cE(\theta^\star, S)}[t(x)] }
                \leq \zeta\,.
            \]
            Substituting this in the previous inequality implies that with positive probability 
            \[
                \norm{ \E_{\cE(\theta^\star)}[t(x)] -  \E_{\cE(\theta^\star, S)}[t(x)] } 
                \leq 
                \max\inbrace{
                    \eta \Lambda\,,
                    {\frac{2}{\eta}\cdot \log \frac{1}{\cE(S; \theta^\star)}} + 2\zeta
                }
                \,.
            \]
            Further, since the left side is a \textit{constant} and, in particular, this constant is not dependent on the values of the random variables $x_1,x_2,\dots,x_n$, the above inequality always holds.
            While this is a looser bound than in the statement,  this argument works for any $\zeta>0$ and the desired statement follows by taking the limit $\zeta\to 0^+$.
        \end{proof}
        
        \subsubsection*{Proof of \cref{prop:perturbed_mle_alt_proof:trunc_concentration}}
        \begin{proof}[Proof of \cref{prop:perturbed_mle_alt_proof:trunc_concentration}]

            The $n$ samples $x_1,x_2,\dots,x_n$ from $\cE(T,\thetaStar)$ can be thought of as $O(n/\alpha)$ samples from $\cE(\thetaStar)$ where we only keep the ones in $T$.
            For these $O(n/\alpha)$ samples and any unit vector $u$, $u^\top \inparen{ t(x) - \E_{\cE(\theta^\star)}[t(x)]} $ is a $(\Lambda,1/\eta)$-sub-exponential random variable due to \cref{prop:claim1lwz}.
            Hence, for any constant $c\geq \eta \Lambda$
            \[
                \Pr_{x \sim \cE(\theta^\star)}\inparen{u^\top \inparen{ t(x) - \E_{\cE(\theta^\star)}[t(x)]} \geq c} \leq \exp\inparen{-\frac{c \eta}{2}}
                \,.
            \]
            In particular, this holds when $u$ is any of the standard basis vectors or their negations.
            Applying a union bound over $O(n/\alpha)$ samples and $2m$ choices of $u$, it follows that with probability $1-\delta$,
            every coordinate $j$ and sample $i$ satisfies
            \[
                \abs{t(x_i)[j] - \E_{\cE(\theta^\star)}[t(x)][j]} \leq 
                \max\inbrace{
                    \Lambda \eta\,, 
                    O\inparen{
                        \frac{1}{\eta}\log\inparen{\frac{2nm}{\alpha\delta}}
                    }
                }
                \leq O\inparen{
                       \frac{\Lambda}{\eta}\log\inparen{\frac{2nm}{\alpha\delta}}
                }
                \,.
            \]
            Now conditioning on this event and applying Hoeffding's inequality, we get that for all $j$
            \[ \Pr\inparen{\abs{\frac{1}{n}\sum_{i \in [n]} t(x_i)[j] - \E_{\cE(\theta^\star, S)}[t(x)][j]} \geq \frac{\zeta}{\sqrt{m}}} \leq 2\exp\inparen{-
                \Omega\inparen{
                    \frac{n\zeta^2\eta^2}{m \Lambda^2 \log^2(2nm/(\alpha\delta))}
                }
                }\,. 
            \]
            Since $n = \tilde{\Omega}\inparen{{\inparen{\sfrac{m\Lambda^2}{(\zeta^2\eta^2)}} \log^2(1/\alpha)\log^2(1/\delta)}}$, it follows that with probability $1-\delta$, $\inparen{\sfrac{1}{n}}\sum_{i \in [n]} t(x_i)$ is $\zeta$ close to $\E_{\cE(\theta^\star, S)}[t(x)]$.

        \end{proof}

        \subsection{Proof of \texorpdfstring{\cref{prop:intro:cor10cor11:informal}}{Lemma 3.6}: Norm of Gradient at \texorpdfstring{$\theta^\star$}{θ*}}
        \label{sec:perturbed_mle_proof}
            In this section, we prove \cref{prop:intro:cor10cor11:informal} whose formal statement is as follows. %
            
                \begin{restatable}[{Norm of Gradient at $\theta^\star$}]{lemma}{boundingGradientNorm}\label{prop:intro:cor10cor11} 
                    For any $S\subseteq\R^d$  
                    such that $\sfrac{
                            \cE(S\triangle S^\star; \theta^\star)
                        }{
                            \cE(S^\star; \theta^\star)
                        }\leq 3/7$,
                    \[
                        \norm{
                            \nabla \negLL_S(\theta)|_{\theta^\star}
                        }_2
                        \leq
                        \frac{3\Lambda e^{2 + 2\Lambda\eta^2}}{\alpha\eta}
                        \cdot \sqrt{\frac{
                            \cE(S\triangle S^\star; \theta^\star)
                        }{
                            \cE(S^\star; \theta^\star)
                        }}
                        \,. 
                    \]
                \end{restatable}
            \noindent The proof of \cref{prop:intro:cor10cor11} relies on the property that when $x\sim \cE(\thetaStar, \Sstar)$, $t(x)$ is a sub-exponential random variable.
            In the previous section, \cref{prop:claim1lwz} shows that $t(x)$ is sub-exponential when $x\sim \cE(\thetaStar)$ the result below extends this to the case 
            $x\sim \cE(\thetaStar, \Sstar)$ using an analogous proof. 
            Actually, we prove a stronger result that applies to $x\sim \cE(\thetaStar, T)$ for any $T$ with $\cE(T;\thetaStar)\geq \Omega(1)$.

        \begin{lemma}[Sufficient Statistics of $\cE(\theta^\star, S^\star)$ is Sub-Exponential]\label{lem:truncatedSubexponential}
            Suppose \cref{asmp:1:polynomialStatistics,asmp:cov,asmp:int} hold.
            Fix any set $T\subseteq \Theta$, with $\cE(T;\thetaStar)\geq \beta$ (for some $\beta\in (0,1]$.
            For any unit vector $u$, it holds that the zero-mean and real-valued random variable $u^\top\inparen{t(x) - \E_{\cE(\theta^\star, S^\star)}[t(x)]}$ is $\inparen{e^{2+2\Lambda\eta^2}\cdot \sfrac{\Lambda}{\beta^2}, \sfrac{2}{\eta}}$-sub-exponential when $x\sim \cE(\thetaStar, T)$. 
        \end{lemma}
        \begin{proof}
            Consider the density of the truncated distribution $\cE(\theta^\star,  T ).$
            \[ 
                \cE(x; \theta^\star,  T ) = 
                \frac{
                    h(x)\ind\{x \in  T  \} 
                }{\cE( T ; \theta^\star)} 
                \cdot 
                \exp\inparen{\theta^\top t(x) - A(\theta)}\,.
            \]
            Observe that this is a member of the exponential family whose base measure is $\sfrac{
                    h(x)\ind\{x \in  T  \} 
                }{\cE( T ; \theta^\star)} $ and sufficient statistic is $t(x)$.
            To simplify the notation, define 
            \[
                \wt{h}(x)\coloneqq \frac{
                    h(x)\ind\{x \in  T  \} 
                }{\cE( T ; \theta^\star)}
                \qquadand
                \wt{A}(x)\coloneqq \log\inparen{\frac{1}{\cE(\Sstar; \thetaStar)}}
                + 
                \log\int_{\Sstar} h(z)e^{\inangle{\thetaStar, t(z)}} \d z\,.
            \]
            Standard properties of the exponential families \cite{wainwright2008graphical} imply that for all $\theta$ in the domain 
            \begin{equation}
                \nabla \wt{A}(\theta) = \E_{\cE(\theta,  T )}[t(x)]\qquadand \nabla^2 \wt{A}(\theta) = \cov_{\cE(\theta,  T )}[t(x)]\,. \label{eq:perturbed_mle_alt_proof:truncsubexp:grad_hess}
            \end{equation}
            Further, since $\cov_{\cE(\theta)}[t(x)] \preceq \Lambda I$ for all $\theta\in\Theta$, it is straightforward to show that for all $\theta\in\Theta$
            \begin{equation}
                \nabla^2 \wt{A}(\theta) = \cov_{\cE(\theta,  T )}[t(x)] \preceq \frac{\Lambda}{\cE( T ;\theta)}\cdot I\,,
                \label{eq:perturbed_mle_alt_proof:truncsubexp:smooth}
            \end{equation}
            Hence, $\wt{A}(\cdot)$ is $\sfrac{\Lambda}{\cE( T ;\theta)}$-smooth at each $\theta\in \Theta.$

            Now, consider any $\theta\in B(\thetaStar, \eta/2)\cap \Theta$ where $B(z,r)$ is the $L_2$-ball of radius $r$ centered at $z$.
            \cref{thm:module:unlabeledSamples:measureGuarantees} implies that, for each such $\theta$
            \[
                \cE( T ; \theta) ~\geq~
                e^{-(2+2\Lambda \eta^2)}\cdot \cE( T ; \thetaStar)^{2}
                ~~~~~\qquad\qquad\Stackrel{\cE(T,\thetaStar)\geq \beta~{\rm and~\cref{asmp:1:polynomialStatistics}}}{\geq}\qquad\qquad~~~~~
                \beta^2\cdot {e^{-(2 + 2\Lambda \eta^2)}}\,.
            \]
            In particular, combined with \cref{eq:perturbed_mle_alt_proof:truncsubexp:smooth}, and the above inequality implies that $\wt{A}(\cdot)$ is $ \sfrac{\Lambda\cdot e^{(2+2\Lambda \eta^2)}}{\beta^2}$-smooth over $B(\thetaStar, \eta/2)\cap \Theta$.
            In fact, since $\thetaStar$ is in the $\eta$-relative interior of $\Theta$, $B(\thetaStar, \eta/2)\cap \Theta$ is equal to $B(\thetaStar, \eta/2)$.
            Here, we use that $\Theta$ is full-dimensional, which can be ensured by selecting an appropriate parameterization of the underlying exponential family.

            Now, we are ready to show that, for any unit vector $u$, $u^\top\inparen{t(x) - \E_{\cE(\theta^\star,  T )}[t(x)]}$ is sub-exponential.
            Consider any $-\eta/2 < \gamma < \eta/2$.
            Observe that 
            \begin{align*}
                \Ex\insquare{
                    e^{\gamma \cdot u^\top\inparen{t(x) - \E_{\cE(\theta^\star,  T )}[t(x)]}}
                }
                &= 
                \int_{\Sstar} 
                    \frac{
                        h(z)\ind\{z \in  T  \} 
                    }{\cE( T ; \theta^\star)} 
                    \cdot 
                    e^{
                        \theta^\top t(z) - \wt{A}(\thetaStar) + \gamma \cdot u^\top\inparen{t(z) - \E_{\cE(\theta^\star,  T )}[t(z)]}
                    }
                    \d z\\
                &= e^{\wt{A}(\thetaStar + \gamma u) - \wt{A}(\thetaStar)}\cdot e^{-\gamma\cdot u^\top \E_{\cE(\theta^\star,  T )}[t(z)]}\,.
                \yesnum\label{eq:perturbed_mle_alt_proof:truncsubexp:bound1}
            \end{align*} 
            Since $\abs{\gamma} \leq \eta/2$ and $u$ is a unit vector, the vector $\thetaStar+\gamma u$ is in $B(\thetaStar, \eta/2)$.
            Further, since $\wt{A}(\cdot)$ is $\sfrac{\Lambda \cdot e^{2\Lambda \eta^2}}{\beta^2}$-smooth over $B(\thetaStar, \eta)$, it follows that 
            \[
                \wt{A}(\thetaStar + \gamma u) - \wt{A}(\thetaStar)
                \leq 
                \sinangle{\grad \wt{A}(\thetaStar), \gamma u} + 
                \frac{\gamma^2\cdot \Lambda\cdot e^{2+2\Lambda \eta^2}}{2\beta^2}
                ~~\Stackrel{\eqref{eq:perturbed_mle_alt_proof:truncsubexp:grad_hess}}{=}~~
                \gamma\cdot u^\top \E_{\cE(\theta^\star,  T )}[t(z)]
                + 
                \frac{\gamma^2\cdot \Lambda \cdot e^{(2+2\Lambda \eta^2)}}{2\beta^2}\,.
            \]
            Combining this with \cref{eq:perturbed_mle_alt_proof:truncsubexp:bound1} implies the desired sub-exponential property.
        \end{proof}
        \paragraph{Proof of \cref{prop:intro:cor10cor11}} 
        \begin{proof}
            Our goal is to upper bound $\norm{
                            \nabla \negLL_S(\theta)|_{\theta^\star}
                        }_2=\snorm{
                            \E_{\cE(\theta^\star, S^\star \cap {S})}[t(x)]
                            - 
                            \E_{\cE(\theta^\star, {S})}[t(x)]
                        }_2$.
            Let $z_1,z_2,\dots$ be \iid{} samples from $\cE(\thetaStar,\Sstar\cap S)$.
                Following the proof of \cref{prop:perturbed_mle_alt_proof:trunc_concentration} and using \cref{lem:truncatedSubexponential} instead of \cref{prop:claim1lwz} implies that for any $\zeta\in (0,1)$ and
                $   n
                    =
                    \tilde{\Omega}\inparen{
                        {\inparen{\sfrac{m\Lambda^2e^{2\Lambda\eta^2}}{(\zeta^2\eta^2\alpha^4)}} \log^2(1/\alpha)\log^2(1/\delta)}
                    }
                $
                with probability $1-(\delta/2)$,
                \[  
                    \norm{\frac{1}{n}\sum_{i \in [n]} t(z_i) - \E_{\cE(\theta^\star, \Sstar\cap S)}[t(x)]}_2 \leq \zeta\,.
                    \yesnum\label{eq:cor10cor11:newEvent}
                \]
            We will first upper bound $\snorm{
                            \frac{1}{n}\sum_i t(z_i)
                            - 
                            \E_{\cE(\theta^\star, {S})}[t(x)]
                        }_2$.
            
            \paragraph{Upper Bound on $\snorm{
                            \frac{1}{n}\sum_i t(z_i)
                            - 
                            \E_{\cE(\theta^\star, {S})}[t(x)]
                        }_2$.}
            Fix an orthonormal basis $u_1,u_2,\dots,u_m$ where $u_1\propto \E_{\cE(\theta^\star, S^\star\cap S)}[t(x)] - \E_{\cE(\theta^\star, S)}[t(x)]$.
                With respect to this basis, $\snorm{
                            \frac{1}{n}\sum_i t(z_i)
                            - 
                            \E_{\cE(\theta^\star, {S})}[t(x)]
                        }_2^2$ can be expressed as follows
                \[
                    \inparen{u_1^\top \inparen{\E_{\cE(\theta^\star, S)}[t(x)]
                    -
                    \frac{1}{n} \sum_i t(z_i)}}^2
                    +
                    \sum_{j\neq 1} \inparen{u_j^\top \inparen{\E_{\cE(\theta^\star, S^\star\cap S)}[t(x)] - \frac{1}{n} \sum_i t(z_i)}}^2\,.
                    \yesnum\label{eq:cor10cor11:expansionOfNorm}
                \]
                Since $t(z_i)$ is subexponential (\cref{lem:truncatedSubexponential}) and $\E_{\cE(\theta^\star, S^\star\cap S)}[t(x)]$ is the mean of $t(z_i)$, standard arguments (as in \cref{prop:perturbed_mle_alt_proof:trunc_nontru_mean_suff}) imply that for $\zeta\in (0,1)$ (to be fixed later) and 
                \[
                     n=\wt{\Omega}\inparen{
                        \frac{
                            m\Lambda^2 e^{4\Lambda\eta^2}
                        }{
                            (\eta\zeta)^2
                            \alpha^4
                        }
                        \cdot 
                        \log^2\inparen{\frac{1}{\alpha}}
                        \log^2\inparen{\frac{1}{\delta}}
                    }   \,,
                \]
                with probability $1-O(\delta)$, the second term is at most $\zeta^2$.
                Toward bounding the first term in the above sum, we will apply a change of measure to the following concentration inequality:
                for any $c\geq 0$
                \[
                    \Pr_{z_1,z_2,\dots,z_n \sim \cE(\theta^\star, S)}\inparen{
                        u_1^\top \inparen{
                            \frac{1}{n} \sum_i t(z_i) - 
                        \E_{\cE(\theta^\star, S)}[t(x)]
                        } \geq c 
                    } 
                    \leq 
                    e^{
                        -\Omega(n)\cdot \min\inbrace{
                            c\eta\,,
                            \frac{\alpha^2 c^2}{\Lambda}\cdot
                            e^{-2\Lambda\eta^2}
                        }
                    }
                    \,. 
                \]
                Changing the measure from $\cE(\thetaStar, S)$ to $\cE(\thetaStar, \Sstar\cap S)$ implies that 
                \[
                    \Pr_{z_1,z_2,\dots,z_n \sim \cE(\theta^\star, \Sstar\cap S)}\inparen{
                        u_1^\top \inparen{
                            \frac{1}{n} \sum_i t(z_i) - 
                        \E_{\cE(\theta^\star, S)}[t(x)]
                        } \geq c 
                    } 
                    \leq 
                    \inparen{
                        \frac{
                            \cE(S; \thetaStar)
                        }{
                            \cE(\Sstar\cap S; \thetaStar)
                        }
                    }^n
                    \cdot 
                    e^{
                        -\Omega(n)\min\inbrace{
                            c\eta\,,
                            \frac{\alpha^2 c^2}{\Lambda}\cdot
                            e^{-2\Lambda\eta^2}
                        }
                    }\,. 
                \]
                In particular, the right hand side is at most $\delta$ for the following $n$ and $c$ satisfying the following 
                \begin{align*}
                    c &\geq 
                    \Omega\inparen{
                        \max\inbrace{
                            \frac{1}{\eta}\cdot \log{\frac{
                                \cE(S; \thetaStar)
                            }{
                                \cE(\Sstar\cap S; \thetaStar)
                            }}
                            \,,
                            \frac{\sqrt{\Lambda}}{\alpha}\cdot 
                            e^{\Lambda\eta^2}\cdot 
                            \sqrt{\log{\frac{
                                \cE(S; \thetaStar)
                            }{
                                \cE(\Sstar\cap S; \thetaStar)
                            }}}
                        }
                    }\,,\\
                    n&\geq 
                    \Omega\inparen{
                        \max\inbrace{
                            \frac{1}{c\eta}\,,
                            \frac{\Lambda}{\alpha^2c^2}
                            \cdot 
                            e^{2\Lambda\eta^2}
                        }
                        \cdot 
                        \log{\frac{1}{\delta}}
                    }\,.
                \end{align*} 
                Since $\alpha,\eta\leq 1$, $e^{\Lambda\eta^2},\Lambda\geq 1$, and $\sfrac{1}{r}\geq \log(\sfrac{1}{r}),\sqrt{\log(\sfrac{1}{r})}$ for all $0<r\leq 1$, the above inequalities can be relaxed to 
                \begin{align*}
                    c\geq \frac{\sqrt{\Lambda}\cdot e^{\Lambda\eta^2}}{\alpha\eta}
                    \cdot \sqrt{\log{\frac{
                            \cE(S; \thetaStar)
                        }{
                            \cE(\Sstar\cap S; \thetaStar)
                        }}}
                    \quadand
                    n\geq \Omega\inparen{
                        \frac{\Lambda\cdot e^{2\Lambda\eta^2}}{\alpha^2\eta}\cdot \max\inbrace{\frac{1}{c}\,, \frac{1}{c^2}}
                        \cdot \log{\frac{1}{\delta}}
                    }
                    \,.
                \end{align*}
                Hence, for a sufficiently large $n$, with probability $1-(\delta/2)$
                \[
                    u_1^\top \inparen{
                            \frac{1}{n} \sum_i t(z_i) - 
                        \E_{\cE(\theta^\star, S)}[t(x)]
                        } 
                    \leq 
                    \frac{\sqrt{\Lambda}\cdot e^{\Lambda\eta^2}}{\alpha\eta}
                    \cdot \sqrt{\log{\frac{
                            \cE(S; \thetaStar)
                        }{
                            \cE(\Sstar\cap S; \thetaStar)
                        }}}
                    \,.
                \]
                Fix the constant $\zeta$ introduced earlier to be the right side of the above equation.
                Combining with the bound we showed on the second term of \cref{eq:cor10cor11:expansionOfNorm}, implies that with a sufficiently large $n$, probability $1-(\delta/2)$
                \[
                    \norm{
                        \E_{\cE(\theta^\star, S)}[t(x)]
                        -
                        \frac{1}{n} \sum_i t(z_i)
                    }_2
                    \leq 
                    \frac{\sqrt{2\Lambda}\cdot e^{\Lambda\eta^2}}{\alpha\eta}
                    \cdot \sqrt{\log{\frac{
                            \cE(S; \thetaStar)
                        }{
                            \cE(\Sstar\cap S; \thetaStar)
                        }}}
                    \,.
                    \yesnum\label{eq:cor10cor11:bound1}
                \] 
            Therefore, a union bound over the events in \cref{eq:cor10cor11:newEvent,eq:cor10cor11:bound1} followed by an application of the triangle inequality, implies that with probability $1-\delta$
                \begin{align*}
                \norm{
                            \E_{\cE(\theta^\star, S^\star\cap S)}[t(x)]
                            - 
                            \E_{\cE(\theta^\star, S)}[t(x)]
                        }_2
                        &\leq 
                    \frac{2\sqrt{2\Lambda}\cdot e^{\Lambda\eta^2}}{\alpha\eta}
                    \cdot \sqrt{\log{\frac{
                            \cE(S; \thetaStar)
                        }{
                            \cE(S\cap \Sstar; \thetaStar)
                        }}}
                        +\zeta \,.
                \end{align*}
                Moreover, since the left side is a \textit{constant} and, in particular, this constant is not dependent on the values of the random variables $z_1,z_2,\dots,z_n$, the above inequalities always hold.
                Furthermore, since $\log(\sfrac{
                            \cE(S; \thetaStar)
                        }{
                            \cE(\Sstar\cap S; \thetaStar)
                        })\leq \log{\frac{1}{1-(\sfrac{
                            \cE(S\triangle S^\star; \theta^\star)
                        }{
                            \cE(S; \theta^\star)
                        })}}$, $\sfrac{
                            \cE(S\triangle S^\star; \theta^\star)
                        }{
                            \cE(S; \theta^\star)
                        }\leq \sfrac{3}{4},$\footnote{To see this, note that, by assumption $\frac{
                            \cE(S\triangle S^\star; \theta^\star)
                        }{
                            \cE(\Sstar; \theta^\star)
                        }\leq \frac{3}{7}$. Hence, $\frac{
                            \cE(S\triangle S^\star; \theta^\star)
                        }{
                            \cE(S; \theta^\star)
                        }\leq \frac{3/7}{1-(3/7)}=\frac{3}{4}.$}
                        and $\log{\frac{1}{1-z}}\leq 2z$ for $0\leq z\leq 3/4$, it follows that 
                         \begin{align*}
                    \norm{
                            \E_{\cE(\theta^\star, S^\star \cap {S})}[t(x)]
                            - 
                            \E_{\cE(\theta^\star, S)}[t(x)]
                        }_2
                        &\leq 
                    \frac{2\sqrt{2\Lambda}\cdot e^{\Lambda\eta^2}}{\alpha\eta}
                    \cdot \sqrt{\frac{
                            \cE(S\triangle \Sstar; \thetaStar)
                        }{
                            \cE(S; \thetaStar)
                        }}
                        +\zeta \,.
                \end{align*}
                Finally, taking the limit $\zeta\to 0^+$ with $n\to\infty$ (at an appropriately large rate), and substituting $\cE(S;\thetaStar)\geq \cE(\Sstar;\thetaStar)/2$ with the simplification that $2\sqrt{2\Lambda}\leq 3\Lambda$ implies \cref{prop:intro:cor10cor11}.

        \end{proof}

        \subsection{Proof of \texorpdfstring{\cref{prop:solving_pmle_psgd:bounded_var_step}}{Lemma 6.12}: Bound on the Second Moment of the Stochastic Gradient}\label{sec:oneERM:efficientImplementationOfSubroutines}
 
        In this section, we prove the following result.
        \upperBoundVar*
        \begin{proof}
        \label{sec:sgd_analysis:proofof:bounded_var_step}
            By construction of the stochastic gradient, we have
            \[
                \E\insquare{\norm{v}_2^2 \mid \theta} = \E_{x \sim \cE(\theta^\star, S^\star \cap S)\,,~~ z\sim \cE(\theta, S)}\insquare{\norm{t(x) - t(z)}_2^2}\,.
            \]
            The right side can be rewritten as follows
            \[
                \Tr(\cov_{\cE(\theta^\star, S^\star \cap S)}[t(x)]) 
                + 
                \Tr(\cov_{\cE(\theta, S)}[t(x)]) 
                + 
                \norm{\E_{x \sim \cE(\theta^\star, S^\star \cap S)}[t(x)] - \E_{x \sim \cE(\theta, S)}[t(x)]}_2^2\,.
            \]
            Toward upper bounding the above expression, observe that since $\cE(\Sstar\triangle S; \thetaStar)\leq \alpha\eps\leq \alpha/2$, $\cE(\Sstar\cap S; \thetaStar)\geq \alpha/2$.
            Moreover, due to \cref{prop:necessary_cond_s_tilde:mass_lower_bound}, $\cE(S;\theta)\geq e^{\frac{-16\Lambda^2}{(\alpha\eta\lambda)^2}\log(\frac{1}{\alpha})}$.
            Hence, \cref{lem:pres_sc_smooth} implies the following upper bounds 
            \[
                \Tr(\cov_{\cE(\theta^\star, S^\star \cap S)}[t(x)]) \leq \frac{2m\Lambda}{\alpha}
                \qquadand
                \Tr(\cov_{\cE(\theta, S)}[t(x)])  \leq m\Lambda\cdot e^{\frac{-16\Lambda^2}{(\alpha\eta\lambda)^2}\log(\frac{1}{\alpha})}\,.
            \]
            To upper bound the third term, namely $\snorm{\E_{x \sim \cE(\theta^\star, S^\star \cap S)}[t(x)] - \E_{x \sim \cE(\theta, S)}[t(x)]}$, we divide into two parts
            \[
                \norm{
                    \E_{x \sim \cE(\theta^\star, S^\star \cap S)}[t(x)] - \E_{x \sim \cE(\theta^\star, S)}[t(x)]
                }_2
                + 
                \norm{
                    \E_{x \sim \cE(\theta^\star, S)}[t(x)] - \E_{x \sim \cE(\theta, S)}[t(x)]
                }_2
                \,.
            \]
            \cref{prop:intro:cor10cor11} and the fact that $\cE(S\triangle \Sstar; \thetaStar)\leq \poly(\alpha\eps)$, imply that the first term is at most 
            $\inparen{\sfrac{(24\Lambda)}{(\alpha\eta)}}\cdot e^{2\Lambda\eta^2}\cdot \eps$.
            To bound the second term, recall that $\grad \negLL_S(\thetaStar)=\E_{x \sim \cE(\theta^\star, S)}[t(x)]$ and $\grad\negLL(\theta)=\E_{x \sim \cE(\theta, S)}[t(x)]$ and, hence,
            \begin{align*}
                \norm{
                    \E_{x \sim \cE(\theta^\star, S)}[t(x)] - \E_{x \sim \cE(\theta, S)}[t(x)]
                }_2
                &= \norm{
                    \grad \negLL_S(\thetaStar) - \grad \negLL_S(\theta)
                }_2\,.
            \end{align*}
            Now, the upper bound follows since $\negLL_S(\cdot)$ is $\Lambda\cdot e^{\frac{16\Lambda^2}{(\alpha\eta\lambda)^2}\log(\frac{1}{\alpha})}$-smooth (\cref{lem:pres_sc_smooth}) and $\norm{\thetaStar-\theta}_2\leq \sfrac{3\Lambda}{(\eta\alpha\lambda)}$ (\cref{lem:findingTheta0Formal}).
        \end{proof}

        \subsection{Proof of \texorpdfstring{\cref{lem:halfspaceLearner:constantLB:secondTerm}}{Lemma 7.4}: Lower Bound on a Polynomial of the Hazard Rate Function}
            \label{sec:proofof:lem:halfspaceLearner:constantLB:secondTerm} 
                In this section, we prove \cref{lem:halfspaceLearner:constantLB:secondTerm}.
                Consider the following two functions
                \[
                    H_1(t)\coloneqq \frac{t^7 + 21t^5+105t^3+105t}{t^6+20t^4+87t^2+48}
                    \quadand
                    H_2(t)\coloneqq \frac{2t+\sqrt{(\pi-2)^2t^2+2\pi}}{\pi}\,.
                \]
                These functions are relevant as they lower bound the Hazard rate on positive reals.
                \begin{theorem}[\cite{gasull2014mills}]\label{thm:lb_on_hazardRate}
                    For all $t> 0$, it holds that $H(t)\geq H_1(t)$ and $H(t)\geq H_2(t)$.
                \end{theorem}
                In fact, polynomials of these functions also lower bound the relevant polynomial of the Hazard rate: $2 H^2(t) - 3 t H(t)$.
                \begin{lemma}\label{lem:lb_on_polynomial_of_hazardRate}
                    For all $t>0$, it holds that 
                    \[
                        2 H^2(t) - 3 t H(t)\geq 2 H_1^2(t) - 3 t H_1(t)
                        \quadand
                        2 H^2(t) - 3 t H(t)\geq 2 H_2^2(t) - 3 t H_2(t)\,.
                    \]
                \end{lemma}
                These lower bounds are useful as it is significantly easier to work with the expressions $2 H_1^2(t) - 3 t H_1(t)$ and $2 H_2^2(t) - 3 t H_2(t)$ and to lower bound them than to lower bound the corresponding expressions of Hazard rate.
                Concretely, we prove the following lower bounds.
                \begin{lemma}\label{lem:lb_on_lb_of_polynomial_of_hazardRate}
                    The following hold.
                    \begin{align*}
                        \text{if } t\geq 1.2\,,\qquad &t^2-1+2 H_1^2(t) - 3 t H_1(t) \geq \Omega\inparen{t^{-12}}\,, \yesnum\label{eq:lb_on_lb_of_polynomial_of_hazardRate1}\\
                        \text{if } t\in (0, 1.2]\,,\qquad &t^2-1+2 H_2^2(t) - 3 t H_2(t) \geq \Omega\inparen{1}\,.\yesnum\label{eq:lb_on_lb_of_polynomial_of_hazardRate2}
                    \end{align*}
                \end{lemma}
                Combining \cref{lem:lb_on_polynomial_of_hazardRate,lem:lb_on_lb_of_polynomial_of_hazardRate} implies that
                \[
                    \text{for any }t>0\,,\qquad 
                    t^2 - 1 + 2 H^2(t) - 3 t H(t)
                    \geq 
                    \Omega\inparen{\min\inbrace{1, \tau^{-12}}}\,.
                \]
                Finally, we extend this lower bound to the negative reals, where we can directly lower bound $t^2 - 1 + 2 H^2(t) - 3 t H(t)$.
                \begin{lemma}\label{lem:lb_on_polynomial_of_hazardRate_on_negatives}
                    For any $t\leq 0$, $t^2 - 1 + 2 H^2(t) - 3 t H(t)\geq \Omega(1).$
                \end{lemma}
                It remains to prove \cref{lem:lb_on_polynomial_of_hazardRate,lem:lb_on_lb_of_polynomial_of_hazardRate,lem:lb_on_polynomial_of_hazardRate_on_negatives}.
                \begin{proof}[Proof of \cref{lem:lb_on_polynomial_of_hazardRate}]
                    We make the following observation:
                    for any $t>0$
                    \begin{align*}
                        H_1(t)
                            &=t + \frac{t^5+18t^3+57t}{t^6+20t^4+87t^2+48}
                            \quad\Stackrel{(t\geq 0)}{\geq}\quad t\,, 
                        \yesnum\label{eq:hazardRate:claim1}\\
                        H_2(t)
                            &= \frac{2t+\sqrt{(\pi-2)^2t^2+2\pi}}{\pi}
                            \geq \inparen{\frac{2}{\pi}+\frac{\pi-2}{\pi}} t
                            = t\,.
                        \yesnum\label{eq:hazardRate:claim2}
                    \end{align*}
                    Next, we observe that 
                    \[
                        2 H^2(t) - 3 t H(t)
                            = \inparen{H(t)-\frac{3t}{4}}^2 - \frac{9t^2}{8}\,.
                    \]
                    Further, \cref{lem:lb_on_polynomial_of_hazardRate,eq:hazardRate:claim1,eq:hazardRate:claim2} imply that, for all $t>0$, $H(t)\geq H_1(t)\geq \sfrac{3t}{4}$ and $H(t)\geq H_2(t)\geq \sfrac{3t}{4}$.
                    Hence, for any $t>0$
                    \begin{align*}
                        2 H^2(t) - 3 t H(t)
                         &\geq \inparen{H_1(t)-\frac{3t}{4}}^2 - \frac{9t^2}{8}
                         = 2 H_1^2(t) - 3 t H_1(t)\,,\\
                        2 H^2(t) - 3 t H(t)
                         &\geq \inparen{H_2(t)-\frac{3t}{4}}^2 - \frac{9t^2}{8}
                         = 2 H_2^2(t) - 3 t H_2(t)\,.
                    \end{align*}
                \end{proof}
                \begin{proof}[Proof of \cref{lem:lb_on_lb_of_polynomial_of_hazardRate}]
                    We divide the proof into two parts that respectively prove \cref{eq:lb_on_lb_of_polynomial_of_hazardRate1,eq:lb_on_lb_of_polynomial_of_hazardRate2}.
                    
                    \begin{subenvironment}
                        \paragraph{Proof of \cref{eq:lb_on_lb_of_polynomial_of_hazardRate1}.}   
                            Straightforward computation shows that 
                            \[
                                t^2-1+2 H_1^2(t) - 3 t H_1(t)
                                =
                                \frac{
                                    2t^8+54t^6+438t^4+882t^2-2304
                                }{
                                    \inparen{t^6+20t^4+87t^2+48}^2
                                }\,.
                            \]
                            Define the numerator as $n(t)\coloneqq 2t^8+54t^6+438t^4+882t^2-2304.$
                            One can verify that $n(t)$ has (i) a single critical point at 0 and (2) two real roots $r_1=-1.19...<1.2$ and $r_2=1.19...<1.2$.
                            Hence, $n(t)$ is non-decreasing in the interval $[1.2,\infty)$.
                            Moreover, $n(1.2)\geq \sfrac{9}{10^4}$.
                            It follows that, for all $t\geq 1.2$, $n(t)\geq \sfrac{9}{10^4}$ and, hence,
                            \[
                                t^2-1+2 H_1^2(t) - 3 t H_1(t)
                                \geq \frac{9}{10^4}\cdot \frac{1}{
                                    \inparen{t^6+20t^4+87t^2+48}^2
                                }
                                \geq \frac{9}{10^4\cdot (156)^2}\cdot \frac{1}{t^{12}}\,.
                            \]

                        \paragraph{Proof of \cref{eq:lb_on_lb_of_polynomial_of_hazardRate2}.}
                            Substituting the expression of $H_2$ shows that 
                            \[
                                \text{\small ${
                                    t^2-1+2 H_2^2(t) - 3 t H_2(t)
                                    =
                                    \inparen{1+\frac{8}{\pi^2}-\frac{6}{\pi}+\frac{2(\pi-2)^2}{\pi^2}}t^2
                                    +
                                    \inparen{\frac{8}{\pi^2} - \frac{3}{\pi}} t\sqrt{(\pi-2)^2t^2+2\pi}
                                    +\frac{4}{\pi}-1\,.
                                }$}
                            \]
                            Further, since $1+\frac{8}{\pi^2}-\frac{6}{\pi}+\frac{2(\pi-2)^2}{\pi^2}\geq 0.16$, $\frac{8}{\pi^2} - \frac{3}{\pi}\geq -0.15$, and $t\geq 0$ %
                            \[
                                t^2-1+2 H_1^2(t) - 3 t H_1(t)
                                \geq 
                                \frac{16}{100}\cdot t^2
                                -\frac{15}{100}\cdot t\sqrt{(\pi-2)^2t^2+2\pi}
                                +\frac{4}{\pi}-1\,.
                            \]
                            Furthermore, since $(z^2/4)+2.51\geq \sqrt{(\pi-2)^2z^2+2\pi}$ for all $z\in \R$
                            \[
                                t^2-1+2 H_1^2(t) - 3 t H_1(t)
                                \geq 
                                \frac{16}{100}\cdot t^2
                                -\frac{15}{100}\cdot t\inparen{\frac{t^2}{4}+2.51}
                                +\frac{4}{\pi}-1
                                \eqqcolon f(t)\,.
                            \]
                            $f(\cdot)$ is a cubic with a unique zero at $1.23..>1.2$ that approaches $-\infty$ as $t\to-\infty$.
                            Hence, $f(\cdot)$ is a decreasing function on $(-\infty, 1.2]$.
                            Moreover, $f(1.2)\geq \sfrac{4}{1000}$ and, hence, for any $t\leq 1.2$
                            \[
                                t^2-1+2 H_1^2(t) - 3 t H_1(t)
                                \geq f(t)
                                \geq f(1.2)
                                \geq \frac{4}{1000}\,.
                            \]
                    \end{subenvironment}
                    
                \end{proof}
                \begin{proof}[Proof of \cref{lem:lb_on_polynomial_of_hazardRate_on_negatives}]
                    Like with $t>0$, with $t\leq 0$ (as required in this result), we expect that the Hazard rate can be lower bounded by an easy-to-analyze function, which can then be used to prove the result. 
                    However, we did not find any references for bounds on the Hazard rate in the $t\leq 0$ regime.
                    Instead, below, we take a ``more direct'' approach to prove the result.
                    
                    We divide the proof into several cases.
                    \begin{enumerate}
                        \item \textbf{Case A ($t\leq -1.05$):}
                            Since $H(\cdot)\geq 0$ and $t\leq 0$, $2H^2(t)-3tH(t)\geq 0$ and, hence, 
                            \[
                                t^2-1+2H^2(t)-3tH(t)
                                \geq t^2-1
                                \quad\Stackrel{(t\leq -1.05)}{>}\quad \frac{1}{10}\,.
                            \]
                        \item \textbf{Case B ($t\in \left(-1.05,-\nfrac{2}{3}\right]$):}
                            Since $H(\cdot)\geq 0$ and $t\leq -\nfrac{2}{3}$, $2H^2(t)-3tH(t)\geq 2H^2(t)+2H(t)$.
                            Further, on $t\geq -1.05$, $H(t)\geq \frac{e^{-1.05^2/2}}{\sqrt{2\pi}(1-\Phi(-1.05))}\eqqcolon c(1.05)$.
                            Hence, $2H^2(t)-3tH(t)\geq 2H^2(t)+2H(t)\geq 2c(1.05)^2+2c(1.05)\geq 0.68.$
                            Further, $t^2-1\geq -\nfrac{5}{9}\geq -0.56$.
                            It follows that $t^2-1+2H^2(t)-3tH(t)\geq 0.1$.
                        \item \textbf{Case C ($t\in \left(-\nfrac{2}{3}, -\nfrac{2}{5}\right]$):}
                            Since $H(\cdot)\geq 0$ and $t\leq -\nfrac{2}{5}$, $2H^2(t)-3tH(t)\geq 2H^2(t)+\nfrac{6H(t)}{5} $.
                            Further, on $t\geq -\nfrac{2}{3}$, $H(t)\geq \frac{e^{-(\nfrac{2}{3})^2/2}}{\sqrt{2\pi}(1-\Phi(-\nfrac{2}{3}))}\eqqcolon c(\nfrac{2}{3})$.
                            Hence, $2H^2(t)-3tH(t)\geq 2H^2(t)+\nfrac{6H(t)}{5} \geq 2c(\nfrac{2}{3})^2+(\nfrac{6}{5}) c(\nfrac{2}{3})\geq 0.878.$
                            Further, $t^2-1\geq -0.85$.
                            It follows that $t^2-1+2H^2(t)-3tH(t)\geq 0.02$.
                            
                        \item \textbf{Case D ($t\in \left(-\nfrac{2}{5}, -\nfrac{1}{5}\right]$):}
                            Since $H(\cdot)\geq 0$ and $t\leq \nfrac{1}{5}$, $2H^2(t)-3tH(t)\geq 2H^2(t)+\nfrac{3H(t)}{5} $.
                            Further, on $t\geq -\nfrac{2}{5}$, $H(t)\geq \frac{e^{-(\nfrac{2}{5})^2/2}}{\sqrt{2\pi}(1-\Phi(-\nfrac{2}{5}))}\eqqcolon c(\nfrac{2}{5})$.
                            Hence, $2H^2(t)-3tH(t)\geq 2H^2(t)+\nfrac{3H(t)}{5} \geq 2c(\nfrac{2}{5})^2+(\nfrac{3}{5})c(\nfrac{2}{5})\geq 0.968.$
                            Further, $t^2-1\geq -0.96$.
                            It follows that $t^2-1+2H^2(t)-3tH(t)\geq 0.008$.

                        \item \textbf{Case E ($t\in \left(-\nfrac{1}{5}, -\nfrac{1}{20}\right]$):}
                            Since $H(\cdot)\geq 0$ and $t\leq \nfrac{1}{20}$, $2H^2(t)-3tH(t)\geq 2H^2(t)+ \nfrac{3H(t)}{20}$.
                            Further, on $t\geq -\nfrac{1}{5}$, $H(t)\geq \frac{e^{-(\nfrac{1}{5})^2/2}}{\sqrt{2\pi}(1-\Phi(-\nfrac{1}{5}))}\eqqcolon c(\nfrac{1}{5})$.
                            Hence, $2H^2(t)-3tH(t)\geq 2H^2(t) + \nfrac{3H(t)}{20}\geq 2c(\nfrac{1}{5})^2 + (\nfrac{3}{20}) c(\nfrac{1}{5})\geq 1.01.$
                            Further, $t^2-1\geq - 1$.
                            It follows that $t^2-1+2H^2(t)-3tH(t)\geq 0.01$.
                            
                        \item \textbf{Case F ($t\in \left(-\nfrac{1}{20}, 0\right]$):}
                            Since $H(\cdot)\geq 0$ and $t\leq 0$, $2H^2(t)-3tH(t)\geq 2H^2(t)$.
                            Further, on $t\geq -\nfrac{1}{20}$, $H(t)\geq \frac{e^{-(\nfrac{1}{20})^2/2}}{\sqrt{2\pi}(1-\Phi(-\nfrac{1}{20}))}\eqqcolon c(\nfrac{1}{20})$.
                            Hence, $2H^2(t)-3tH(t)\geq 2H^2(t)\geq 2c(\nfrac{1}{20})^2\geq 1.17.$
                            Further, $t^2-1\geq -1$.
                            It follows that $t^2-1+2H^2(t)-3tH(t)\geq 0.17$.
                    \end{enumerate}
                \end{proof}

\subsection{Proof of \texorpdfstring{\cref{thm:sampleComplexity}}{Theorem 8.1}: Sample Complexity of Estimation with Unknown Truncation}
    \label{sec:additionalProofs:sec:sampleEfficiency}

    In this section, we prove the first part of \cref{thm:sampleComplexity}. The second part of \cref{thm:sampleComplexity} (that uses a single ERM call) follows due to our analysis of the computationally efficient algorithm as explained in the \cref{sec:sampleEfficiency}.
    
    The proof of this result straightforwardly generalizes the proof of Informal Theorem 1 in \citet{Kontonis2019EfficientTS} from Gaussian distributions to exponential families satisfying \cref{asmp:1:sufficientMass,asmp:1:polynomialStatistics,asmp:cov,asmp:int}.

    \paragraph{Proof Outline.}
    We divide the proof into three lemmas: \cref{cor:nontr_tv_to_trunc,lem:lemma2KTZ19,lem:extrapolationOfTV}. 
    The first two lemmas \cref{cor:nontr_tv_to_trunc,lem:lemma2KTZ19} follow by combining the proofs in \citet{Kontonis2019EfficientTS} (which use the properties of Gaussians) with the properties in \cref{asmp:1:sufficientMass,asmp:1:polynomialStatistics,asmp:cov,asmp:int}.
    We state these lemmas without proof:
    
    \begin{lemma}\label{cor:nontr_tv_to_trunc}
        Let $\cS$ be a family of subsets in $\R^d$ and let $\cE\inparen{\theta^\star, S^\star}$ be an exponential distribution truncated to some set $S^\star \in \cS$, and $\cE\sinparen{\theta^\star; S^\star} = \alpha > 0$. 
        Fix $\epsilon \in (0, 1)$ and $\delta \in (0, 1/4)$ and let
        \[ N = \wt{O}\left({\frac{\vc{}(\cS)}{\eps}} + \log{\frac{1}{\delta}} \right). \]
        Furthermore, let $\wt{\theta}$ be such that $\tv{\cE\sinparen{\wt{\theta}}}{\cE\sinparen{\theta^\star}} \leq \eps$. 
        Assume we draw $N$ independent and identically distributed samples from $\cE\inparen{\theta^\star, S^\star}$, $x_1,x_2,\dots,x_N$. 
        Let $\wt{S}$ be the solution to the following 
        \[ 
            \min_S~ \cE\sinparen{S; \wt{\theta}}\,,
            \quadtext{such that,}
            S\supseteq \inbrace{x_1,x_2,\dots,x_N}\,.
            \yesnum
            \label{program:minimizeProbabilityMass}
        \]
        Then with probability at least $1-\delta$,
        \[
            \tv{\cE\inparen{\theta^\star, S^\star}}{\cE\sinparen{\wt{\theta}, \wt{S}}}
                \leq \frac{3\epsilon}{\alpha - \epsilon}\,.
        \]
    \end{lemma}

    \begin{lemma}\label{lem:lemma2KTZ19}
        Let $S^\star \in \cS$ be a subset of $\R^d$ and $\cE\inparen{\theta^\star, S^\star}$ the corresponding truncated exponential family distribution satisfying \cref{asmp:1:sufficientMass,asmp:1:polynomialStatistics,asmp:cov,asmp:int}. Then 
        \[
            \wt{O}\left( 
                {\frac{\vc{}(\cS)}{\eps}} + {
                    \frac{m}{\eps^2}\log^2\frac{1}{\delta}
                } 
            \right)
        \]
        samples are sufficient to find parameters $(\wt{\theta}, \wt{S})$ such that $\tv{\cE\inparen{\theta^\star, S^\star}}{ \cE\sinparen{\wt{\theta}, \wt{S}}} \leq 
        O\inparen{\sfrac{\eps}{(\alpha-\eps)}}$ 
        with probability at least $1-\delta$. 
        Where $\wt{O}(\cdot)$ hides polynomial factors in $\sfrac{1}{\alpha},\sfrac{1}{\eta}$, $\sfrac{\Lambda}{\lambda}$, and $k$. 
    \end{lemma}
    
    \noindent The last lemma shows that $\eps$ error in total variation of the truncated distributions translates to an $O(\eps)$ bound in total variation distance of the untruncated distributions (\cref{lem:extrapolationOfTV}). 
    This is true for any exponential family satisfying \cref{asmp:1:sufficientMass,asmp:1:polynomialStatistics,asmp:cov,asmp:int} and, in particular, does not depend on the complexity of the set. 
     
        \begin{lemma}[{Extrapolating from truncated to untruncated TV-distance}]\label{lem:extrapolationOfTV}
            Suppose \cref{asmp:1:sufficientMass,asmp:1:polynomialStatistics,asmp:cov,asmp:int} hold with constants $\alpha,\lambda,\Lambda>0$ and $k\geq 1$.
            Assume that $\theta_1,\theta_2\in \Theta$.
            Let $T_1=\cE(\theta_1,S_1)$ and $T_2=\cE(\theta_2,S_2)$ be two truncated distributions 
            \[
                \cE(S_1; \theta_1),~ \cE(S_2;\theta_2)\geq \alpha\,.
            \]
            Then 
            \[
                \tv{T_1}{T_2} \geq 
                \alpha \sqrt{\frac{\lambda}{32\Lambda} } \inparen{\frac{\alpha}{kC}}^k
                \cdot \tv{D_1}{D_2}
            \]
            where $D_1=\cE(\theta_1)$, $D_2=\cE(\theta_2)$, and $C$ is a universal constant. 
        \end{lemma}

        \begin{proof}
            We will consider two cases.

            \paragraph{Case A ($T_1(S_1\backslash S_2)\geq \sfrac{\alpha}{2}$):}
                Since $\tv{T_1}{T_2}\geq \frac{1}{2}\abs{T_1(S_1\backslash S_2) - T_2(S_1\backslash S_2)}$ and $T_2$ is supported on $S_2$, it follows that 
                \[
                    \tv{T_1}{T_2}\geq \frac{\alpha}{4}\,.
                \]
                This implies the claim as $\tv{D_1}{D_2} \leq 1$.

            \paragraph{Case B ($T_1(S_1\backslash S_2) < \sfrac{\alpha}{2}$):}   
                Define $S\coloneqq S_1\cap S_2$ and $\beta\coloneqq D_2(S_2)=\cE(S_2; \theta_2)$.
                Since $T_1(S_1\backslash S_2) < \frac{\alpha}{2}$, 
                \[
                    T_1(S)\geq 1-\frac{\alpha}{2}\geq \frac{\alpha}{2}\,. \tag{as $\alpha\leq 1$}
                \]
                Using this and the standard inequality that $\tv{T_1}{T_2}\geq \frac{1}{2}\int_{x\in S}\abs{T_1(x)-T_2(x)} \d x$, we get the following lower bound on $\tv{T_1}{T_2}$
                \begin{align*}
                    \tv{T_1}{T_2}
                    &~~\geq~~ 
                    \min_{R\subseteq \R^n}~~ 
                    \frac{1}{2}\int_{x\in R}\abs{\frac{D_1(x)}{\alpha} - \frac{D_2(x)}{\beta}} \mathds{1}_S(x) \d x\,, %
                    \quad \text{s.t.}\,,\quad \int_{x\in R}\frac{D_1(x)}{\alpha} \mathds{1}_S(x) \geq \frac{\alpha}{2}\,.
                    \yesnum\label{program:fractionalKnapsack}
                \end{align*}
                Since Program~\eqref{program:fractionalKnapsack} is a fractional Knapsack problem, there exists a threshold $\phi\in \R$ such that the optimal solution of Program~\eqref{program:fractionalKnapsack} is:
                \begin{align*}
                    R^\star 
                    &= \inbrace{x \in \R^n \colon \abs{\frac{\inparen{\sfrac{D_1(x)}{\alpha}} - \inparen{\sfrac{D_2(x)}{\beta}}}{\sfrac{D_1(x)}{\alpha}}} \leq \phi
                    }\\
                    &= \inbrace{x \in \R^n \colon \abs{
                        1 - \frac{\alpha}{\beta} \exp\inparen{(\theta_1-\theta_2)^\top t(x) -A(\theta_1)+A(\theta_2)}
                    } \leq \phi
                    }\,.
                \end{align*}
                For some $\gamma>0$ that we will fix later, define
                \[
                    f(x) \coloneqq \inparen{\theta_1 - \theta_2}^\top t(x) - A(\theta_1) + A(\theta_2) + \log\inparen{\frac{\alpha}{\beta}}
                    \quadand 
                    Q\coloneqq \inbrace{x\colon \abs{f(x)}\leq \gamma}\,.
                \]
                Since the sufficient statistics $t(\cdot)$ have degree at most $k$, $f(\cdot)$ also has a degree at most $k$ and, consequently, Theorem 8 of \citet{carbery2001distributional}, implies that 
                there exists a constant $c>0$ such that for all $q>0$
                \[
                    D_1(Q) \leq \frac{cq\gamma^{1/k}}{\inparen{\Ex_{D_1}\insquare{f(x)^{q/k}}}^{1/q}}\,.
                    \yesnum\label{eq:anticoncentration}
                \]
                Fix 
                \[
                    q=2k
                    \quadand
                    \gamma = \inparen{\frac{\alpha}{8kc}}^k\inparen{\Ex_{D_1}\insquare{f(x)^2}}^{1/2}\,.
                \]
                Observe that \cref{eq:anticoncentration} implies that 
                \[
                    D_1(Q)\leq \frac{\alpha}{4}\,.
                \]
                We claim that to prove the proposition it suffices to show that 
                \[
                    \gamma \geq \sqrt{\frac{2\lambda}{\Lambda} } \inparen{\frac{\alpha}{8kc}}^k \tv{D_1}{D_2}\,.
                    \label{eq:extrapolation:lowerBound:onGamma}
                    \yesnum
                \]
                To see this, consider two cases.
                
                \paragraph{Case B.1 ($\gamma\leq 1$):}
                    It holds that 
                    \begin{align*}
                        \int_{x}\abs{T_1(x)-T_2(x)} \d x
                        &\geq \int_{x\in S\backslash Q} \abs{T_1(x)-T_2(x)} \d x\\
                        &= \Ex_{D_1}\insquare{
                            \abs{1-\exp\inparen{f(x)}}\cdot \mathds{1}_{S\backslash Q}(x)
                        } \tag{by the definition of $f(\cdot)$}\\
                        &\geq \Ex_{D_1}\insquare{
                            \frac{f(x)}{2} \mathds{1}_{S\backslash Q}(x)
                        } \tag{as $\gamma\leq 1$ and $\abs{1 - e^z}\geq \frac{\abs{z}}{2}$ if $\abs{z}\leq 1$}\\ 
                        &\geq \frac{\alpha\gamma}{8} \tag{as $D_1(S)\geq \frac{\alpha}{4}$ and $D_1(Q)\leq \frac{\alpha}{4}$}\\
                        &\geq \alpha \sqrt{\frac{\lambda}{32\Lambda} } \inparen{\frac{\alpha}{8kc}}^k \tv{D_1}{D_2}\,. \tag{using \cref{eq:extrapolation:lowerBound:onGamma}}
                    \end{align*}
                \paragraph{Case B.2 ($\gamma > 1$):}
                    It holds that 
                    \begin{align*}
                        \int_{x}\abs{T_1(x)-T_2(x)} \d x
                        &\geq \int_{x\in S\backslash Q} \abs{T_1(x)-T_2(x)} \d x\\
                        &= \Ex_{D_1}\insquare{
                            \abs{1-\exp\inparen{f(x)}}\cdot \mathds{1}_{S\backslash Q}(x)
                        } \tag{by the definition of $f(\cdot)$}\\
                        &\geq \frac{1}{2} \Ex_{D_1}\insquare{
                            \mathds{1}_{S\backslash Q}(x)
                        } \tag{as $\gamma\geq 1$ and $\abs{1 - e^z}\geq \frac{1}{2}$ if $\abs{z} > 1$}\\ 
                        &\geq \frac{\alpha}{8}\,.
                    \end{align*}
            It remains to prove \cref{eq:extrapolation:lowerBound:onGamma}.
            To prove \cref{eq:extrapolation:lowerBound:onGamma}, it suffices to show that 
            \[ 
                \Ex_{D_1}\insquare{f(x)^2}  
                \geq  \var_{D_1}\insquare{f(x)}   
                \geq  \frac{2\lambda}{\Lambda} \tv{D_1}{D_2}^2\,. 
           \]
            This can be proved as follows
            \begin{align*}
                 \var_{D_1}\insquare{f(x)}  
                 &= \var_{D_1}\insquare{
                    \inparen{\theta_1-\theta_2}^\top t(x) - \inparen{
                        A(\theta_1) - A(\theta_2)
                    }
                    + \log\inparen{\frac{\alpha}{\beta}}
                 }\\ 
                 &= \var_{D_1}\insquare{
                    \inparen{\theta_1-\theta_2}^\top t(x) 
                 } \tag{as variance is translation invariant and $\alpha,\beta,\theta_1,$ and $\theta_2$ are independent of $x$}\\ 
                 &= \inparen{\theta_1-\theta_2}^\top 
                    \cov_{D_1}\insquare{ t(x)  }
                    \inparen{\theta_1-\theta_2}\\
                &\geq \lambda \norm{\theta_1-\theta_2}_2^2 \tag{as $\theta_1\in \Theta$}\\
                &\geq \frac{2\lambda}{\Lambda} \tv{D_1}{D_2}^2\,. %
                \tag{using \cref{lem:tv_smooth}}
            \end{align*}
            
        \end{proof}

    \section{PSGD Procedure to Satisfy \texorpdfstring{\cref{asmp:moment}}{Assumption 3.3} Given a Warm Start}\label{sec:psgd:theta0}
        At the core, \cref{asmp:moment} requires an efficient algorithm that given a vector $v\approx \Ex_{\cE(\thetaStar, \Sstar)}[t(x)]$ outputs a parameter $\theta$ such that $\Ex_{\cE(\theta)}[t(x)]=v$.
        The solution to the previous equation can be written in closed form for many distributions, including the Gaussian and Exponential distributions.
        As a concrete example consider the family of Gaussian distributions with identity covariance: here, $t(x)=x$ and $\theta$ is the mean of the distribution.
        Hence, for any $v\in \R^m$, $\Ex_{\cE(v)}[t(x)]=v$ and, hence, it suffices to select $\theta=v$. 
        When such closed-form expressions are not available, one can hope to use continuous optimization algorithms to find an approximate solution to the system $\Ex_{\cE(\theta)}[t(x)]=v$.
        One such approach is to minimize the following negative log-likelihood function using stochastic gradient descent 
        \[
            \negLL(\theta)\coloneqq -\Ex_{x\sim \cE(\theta^\star,\Sstar)}\log{\cE(x;\theta)}\,.
        \]
        Minimizing this function is useful as its minimizer $\wt{\theta}$ satisfies 
        \[
            \grad \negLL(\wt{\theta}) = 
            \Ex_{\cE(\thetaStar, \Sstar)}[t(x)]
            - 
            \Ex_{\cE(\wt{\theta})}[t(x)]
            = 0\,.
        \]
        Hence, in particular, the minimizer of $\negLL(\cdot)$ satisfies $\Ex_{\cE(\wt{\theta})}[t(x)]\approx v$ as $v\approx \Ex_{\cE(\thetaStar, \Sstar)}[t(x)]$.
        If we know that, for some $D\geq 0$
        \[
            \wt{\theta}\in \Theta
            \qquadand
            {\rm diam}\inparen{{\Theta}}\leq D\,,
        \]
        then $\negLL(\cdot)$ has several desirable properties required to run SGD, which can be derived by substituting $S=\R^d$ in the analysis in \cref{sec:reduction_to_known_truncation:solving_pmle_psgd}.
        These properties are:
        \begin{enumerate}
            \item \textbf{(Strong Convexity)}
                $\negLL(\cdot)$ is $\Omega\inparen{\lambda\cdot k^{-k}}$ strongly convex.
            \item \textbf{(Stochastic gradient and their Second moments)}
                The random variable $v=t(x)-t(z)$ where $x\sim \cE(\thetaStar, \Sstar)$ and $z\sim \cE(\wt{\theta})$ is an unbiased estimate of $\negLL(\cdot)$ and satisfies  
                $\Ex\insquare{\norm{v}_2^2\mid \theta}\leq D\Lambda$.
                Moreover, given a sampler for $\cE(\cdot)$, $v$ can be efficiently computed.
            \item \textbf{(Starting point)} 
                Any point $\theta\in \Theta$ can be used as a starting point and satisfies $\snorm{\theta-\wt{\theta}}_2\leq D$
            \item \textbf{(Projection to $\Theta$)}
                Finally, we inherit the projection oracle to the domain $\Theta$ which contains the optimizer $\wt{\theta}$ from \cref{asmp:moment}.
        \end{enumerate}
        The following result follows.

        \begin{proposition}\label{prop:initial_theta_psgd}
            Suppose \cref{asmp:1:sufficientMass,asmp:1:polynomialStatistics,asmp:cov,asmp:int,asmp:proj} hold.
            Suppose ${\rm diam}(\Theta)\leq D$.
            Fix any $\eps,\delta\in(0,1)$.
            There is an algorithm that, given $n$ independent samples from $\cE(\thetaStar,\Sstar)$ for $n=\poly(dmD/\eps,\log(1/\delta))$ outputs an estimate $\theta$ such that with probability $1-\delta$, 
            $\snorm{\wt{\theta}-\theta}_2\leq \eps$.
            The algorithm performs $\poly(n)$ computation and makes $\poly(n)$ calls to the sampling oracle for $\cE(\cdot)$ and to the projection oracle for $\theta$.
            
        \end{proposition}

\end{document}